\documentclass[11pt]{amsart}
\usepackage{amsthm}
\usepackage[english]{babel}
\usepackage[autostyle]{csquotes}
\usepackage{geometry}
\newtheorem{theorem}{Theorem}[section]
\newtheorem{lemma}[theorem]{Lemma}
\newtheorem{proposition}[theorem]{Proposition}
\newtheorem{corollary}[theorem]{Corollary}
\newtheorem{question}[theorem]{Question}
\newtheorem{conjecture}[theorem]{Conjecture}
\theoremstyle{definition}
\newtheorem{definition}[theorem]{Definition}
\newtheorem{example}[theorem]{Example}

\newtheorem{innercustomthm}{Theorem}
\newenvironment{customthm}[1]
{\renewcommand\theinnercustomthm{#1}\innercustomthm}
{\endinnercustomthm}
\newtheorem{innercustomrmk}{Remark}
\newenvironment{customrmk}[1]
{\renewcommand\theinnercustomrmk{#1}\innercustomrmk}
{\endinnercustomrmk}

\newtheorem{innercustomnt}{Notation}
\newenvironment{customnt}[1]
{\renewcommand\theinnercustomnt{#1}\innercustomnt}
{\endinnercustomnt}
\theoremstyle{remark}
\newtheorem{remark}[theorem]{Remark}

\newtheorem{claim}[theorem]{Claim}

\numberwithin{equation}{section}
\usepackage{amssymb}
\usepackage{amsmath}
\begin{document}

\title{Markov Processes and Some PCF quadratic polynomials}


\author[V.Goksel]{Vefa Goksel}
\address{Mathematics Department\\ University of Wisconsin\\
	Madison\\
	WI 53706, USA}
\email{goksel@math.wisc.edu}

\subjclass[2000]{Primary }

\keywords{}

\date{}

\dedicatory{}

\begin{abstract}
For any $n\geq 1$, let $T_n$ be the complete binary rooted tree of height $n$, and $f(x)=(x+a)^2-a-1$ such that $a\neq \pm b^2$ for any $b\in \mathbb{Z}$. In \cite{Settled}, Jones and Boston empirically observed that iteratively applying a certain Markov process on the factorization types of $f$ gives rise to certain permutation groups $M_n(f)\leq \text{Aut}(T_n)$ for $n\leq 5$. We prove a refined version of this phenomenon for all $n$, and for all the irreducible post-critically finite quadratic polynomials with integer coefficients, except for certain conjugates of $x^2-2$. We do this by constructing these groups explicitly. Although there have already been some conjectures relating the Markov processes to the dynamics of quadratic polynomials, our results are the first to prove such a connection. If $f(x)\in \mathbb{Z}[x]$ is a post-critically finite quadratic polynomial, and $G_n(f)$ is the Galois group of $f^n$ over $\mathbb{Q}(i)$, then we conjecture that for all $n\geq 1$, $M_n(f)$ contains a subgroup isomorphic to $G_n(f)$, analogous to the role of Mumford-Tate groups in the classical arithmetic geometry. We provide evidence that this is implied by a purely group theoretical statement.
\end{abstract}
\subjclass[2010]{Primary 11R32, 12E05, 20E08, 37P15}

\keywords{Markov process, post-critically finite, automorphism group of the rooted binary tree}
\maketitle
\section{Introduction}
Let $K$ be a field, and $f(x)\in K[x]$ a quadratic polynomial. We denote by $f^n(x)$ the $n$th iterate of $f(x)$ for $n\geq 1$, and we also make the convention $f^0(x) = x$. Suppose that all the iterates of $f(x)$ are separable. Then $f^n$-pre-images of $0$ form a complete binary rooted tree, as follows: For each root $\alpha$ of $f^{n-1}(x)$, we draw edges from $\alpha$ to $\beta_1$ and $\beta_2$, where $\beta_1$ and $\beta_2$ are the two roots of $f^n$ such that $f(\beta_1) = f(\beta_2) = \alpha$. We will call this tree the \textbf{pre-image tree} of $f$, and denote it by $T$.\\

The absolute Galois group $\text{Gal}(K^{\text{sep}}/K)$ acts on $T$, and it also preserves the connectivity relation in $T$, since the Galois elements commute with $f$. Hence, we obtain a homomorphism
$$\rho : \text{Gal}(K^{\text{sep}}/K) \rightarrow \text{Aut}(T).$$

The image of $\rho$ is one main object of study in the area of arboreal Galois representations. Let $G_n(f)$ be the Galois group of $f^n$ over $K$. These Galois groups form an inverse system, via the natural surjections $G_{n+1}(f)\twoheadrightarrow G_n(f)$. Then, we have the following concrete description of im$(\rho)$.
$$\text{im}(\rho) = \varprojlim G_n(f).$$
We set $G(f) = \text{im}(\rho)$. The question of understanding this image for various $f$ and $K$ has recently drawn great attention. Except for some well-known special cases, it is expected that the index $[\text{Aut}(T):G(f)]$ is finite. See for instance (\cite{survey}, Conjecture $3.11$) for a precise conjecture in this direction. One of these well-known special cases is when the polynomial $f$ is post-critically finite, or PCF for short, by which we mean that the orbit of its critical point under the iteration of $f$ is finite. Besides the fact that $[\text{Aut}(T):G(f)]$ is infinite  (\cite{survey}, Theorem $3.1$), very little is known in general about $G(f)$ when $f$ is PCF.\\

In the simplest case that $f$ has integer coefficients, it follows from a straightforward calculation that if $f(x)$ is PCF, then $f(x)$ is conjugate to $x^2$, $x^2-1$ or $x^2-2$ by the linear map $x\rightarrow x+a$, $a\in \mathbb{Z}$. $G(f)$ is known when $f$ is conjugate to $x^2$ or $x^2-2$ (\cite{Gottesman}). $G(f)$ for the conjugates of $x^2-1$ is still unknown, although there is a concrete conjecture in a special case (\cite{JonesBoston1}, Conjecture $4.6$).\\

In this paper, our goal is to introduce the so-called Markov group of $f$ that conjecturally contains a subgroup isomorphic to $G(f)$. As the name suggests, our main tool for constructing these groups will be a certain Markov process.\\

Although our methods could be applicable to larger families of PCF polynomials, in this paper we will focus on the simplest case given above, namely $f$ will be a PCF quadratic polynomial with integer coefficients. We will treat some other families of PCF quadratic polynomials in a future work.\\

From now on, we take $G_n(f)$ to be the Galois group of $f^n$ over $\mathbb{Q}(i)$. The choice of the base field $\mathbb{Q}(i)$ is for a technical reason, which will be made clear in Section $7$, as it will not be needed in the earlier sections.\\

In \cite{Settled}, Jones and Boston attempted to explain factorization of large iterates of $f$ over finite fields by using a certain Markov process. See Section $2$ for a summary of this Markov model. In fact, Goksel, Xia and Boston refined this model by describing a multi-stage Markov model in \cite{goksel}, but this refined model will not be of our interest in this paper. In the very last section of \cite{Settled}, Jones and Boston used their model to estimate the factorization data of $f^n$ modulo all primes when $n$ gets large. Their idea was that if one can get an accurate description of the data for the factorizations of $f^n$ modulo all primes via this Markov process, since by the Chebotarev's density theorem this data corresponds to the cycle structure data of the Galois group of $f^n$, then the Galois group of $f^n$ could be recovered this way. They did this in the special case $f(x)=(x+a)^2-a-1$, and observed that the Markov process fails to recover the actual Galois group for $n=5$. However, they interestingly observed that the cycle data arising from this Markov process still corresponds to an actual subgroup of $\text{Aut}(T_5)$, where we denote by $T_n$ the truncation of $T$ to the first $n$ levels. In this paper, we will study a refined version of this phenomenon. Namely, using the same Markov model, we will estimate the factorization data of $f^n$ modulo primes $\equiv 1 (\text{mod }4)$, and then construct the so-called Markov groups using this data.\\

We now state our main theorems somewhat informally, which will be made precise in the next sections, after we precisely define the Markov process.

\begin{theorem}
Let $f(x)\in \mathbb{Z}[x]$ be a conjugate of $x^2$ or $x^2-1$ by the linear map $x\rightarrow x+a$, $a\in \mathbb{Z}$. Suppose $f(x)$ is irreducible. For any $n\geq 1$, let $F_n$ be the estimated factorization data of $f^n$ modulo the primes $\equiv 1(\text{mod }4)$ given by the Markov model. Then there exists a permutation group $M_n(f)\leq \text{Aut}(T_n)$ whose cycle data corresponds to $F_n$. 
\end{theorem}

\begin{theorem}
Let $f(x) = (x+a)^2-a-2$, where $a=2\pm b^2$ for some $b\in \mathbb{Z}$. Suppose $f(x)$ is irreducible. For any $n\geq 1$, let $F_n$ be the estimated factorization data of $f^n$ modulo the primes $\equiv 1(\text{mod }4)$ given by the Markov model. Then there exists a permutation group $M_n(f)\leq \text{Aut}(T_n)$ whose cycle data corresponds to $F_n$.
\end{theorem}

We prove both theorems by constructing the corresponding groups explicitly. Namely, we define what is called a \enquote{Markov map}, which allows us to obtain the generators of $M_{n+1}(f)$ from the generators of $M_n(f)$ using the Markov model of $f$.\\

In the case $f(x) = (x+a)^2-a-2$, $a\neq 2\pm b^2$ for any $b\in \mathbb{Z}$, we show that a modified version of the Markov model of $f$ gives permutation groups analogous to $M_n(f)$ above, and then deduce Theorem $1.2$ using these groups (See Corollary $5.17$). We think that the fact that we still get some permutation groups even when we modify the Markov model is particularly interesting, because it seems to indicate that our constructions might be part of a more general phenomenon which includes these connections with the dynamics of quadratic polynomials as a special case. See Section $5$ for more details about this case.\\

The structure of the paper is as follows: In Section $2$, we describe the Markov process and precisely state the main problem. In Section $3$, we will give some preliminary definitions and results from group theory. Sections $4$ and $6$ together will take care of the proof of Theorem $1.1$. Section $5$ will be devoted to the proof of Theorem $1.2$. In Section $7$, we introduce a purely group theoretical conjecture, and discuss why we believe that it would imply that $M_n(f)$ contains a subgroup isomorphic to $G_n(f)$, where $f$ is as in Theorem $1.1$ or Theorem $1.2$. We will finish the paper with some final remarks and speculations in Section $8$.

\section{The markov process and the main problem}
Let $K$ be a field of characteristic $\neq 2$, and $f(x)\in K[x]$ a quadratic polynomial. Let $c$ be the unique critical point of $f$. The \textbf{post-critical orbit} of $f$ is defined to be the set
$$O_f= \{f(c),f^2(c),\dots\}.$$
When this set is finite, we say \textbf{f is post-critically finite}, or \textbf{PCF} for short.
$|O_f|$ is the size of the post-critical orbit, which we will denote by $o$. Moreover, the \textbf{tail} of $f$ is defined to be the set
$$T_f = \{f^k(c) | k\geq 1,\text{ }f^i(c)\neq f^k(c)\text{ for any }i\neq k\}.$$
$|T_f|$ is the size of the tail, which we will denote by $t.$ If $t=0$ (i.e., there exists $i\in \mathbb{N}$ such that $f^i(c) = c$), then we say $c$ is \textbf{periodic} under $f$.
\begin{example}
Take $K=\mathbb{Q}$, and $f(x)=(x+1)^2-2\in \mathbb{Q}[x].$ The critical point is $-1.$ Then the critical orbit becomes $\{-2,-1\}.$ $f$ is PCF  with orbit size $2$ ($o=2$) and tail size $0$ ($t=0$), thus $-1$ is periodic under $f.$
\end{example}

Since we will study the factorization data of quadratic polynomials over finite fields, we will now give a couple of definitions for quadratics over finite fields. We denote the finite field of size $q$ by $\mathbb{F}_q$. \textbf{Throughout,} $\boldsymbol{q}$ \textbf{is an odd prime power.}
 \begin{definition}
 \cite{goksel} Let $f(x)\in \mathbb{F}_q[x]$ be a quadratic polynomial with post-critical orbit $O_f$, and $g(x)\in \mathbb{F}_q[x]$ be an irreducible polynomial. We define the type of $g(x)$ at $\beta$ to be $s$ if $g(\beta)$ is a square in $\mathbb{F}_q$ and $n$ if it is not a square. The type of $g$ is a string of length $|O_f|$ whose $k$th entry is the type of $g(x)$ at the $k$th entry of $O_f.$ The $k$th entry is also called the $k$th digit.
 \end{definition}
\begin{example}
Take $f(x) = (x+1)^2-2\in \mathbb{F}_5[x]$, and $g(x) = x^2+2\in \mathbb{F}_5[x].$ Then $O_f = \{3,4\}.$ We have $g(3) = 1\in \mathbb{F}_5$, $g(4) = 3\in \mathbb{F}_5$, which shows that $g$ has type $sn.$
\end{example}
\begin{definition}
Let $f(x)\in \mathbb{F}_q[x]$ be a quadratic polynomial with post-critical orbit $O_f$, and $g(x)\in \mathbb{F}_q[x]$ be any polynomial. Suppose $g(x)$ factors as $g(x) = g_1(x)g_2(x)\dots g_k(x)$, where $g_i(x)\in \mathbb{F}_q[x]$ is irreducible for $i=1,\dots,k.$ Then the \textbf{factorization type} of $g$ is defined to be the formal product $\prod_{i=1}^{k} [t_i,d_i]$, where $t_i$ is the type of $g_i$, and $d_i$ is the degree of $g_i$.
\end{definition}
To illustrate the definition, if we consider the polynomial $g$ in Example $2.3$, $g$ has factorization type $[sn,2].$
\begin{definition}
Given a quadratic polynomial $f(x)\in \mathbb{F}_q[x]$ and a polynomial $g(x)\in \mathbb{F}_q[x]$, we call the factors of $g(f(x))$ the \textbf{children} of $g.$ 
\end{definition}
\begin{lemma}
\cite{Settled} Suppose that $f(x)\in \mathbb{F}_q[x]$ is a quadratic polynomial with post-critical orbit of length $o$, and all iterates separable. Let $g(x)\in \mathbb{F}_q[x]$ be irreducible of even degree. Suppose that $h_1h_2$ is a non-trivial factorization of $g(f(x))$, and let $d_i$ (resp. $e_i$) be the $i$th digit of the type of $h_1$ (resp. $h_2$). Then there is some $k$, $1\leq k\leq o$, with $d_o = e_k$ and $e_o = d_k.$ Moreover, $k=o$ if and only if $c$ is periodic, and in the case $c$ is not periodic, we have $k=t$, where $t$ is the tail size of $f.$
\end{lemma}
In \cite{Settled}, Jones and Boston attempted to explain the distrubition of the types of the factors of $f^n$ for large values of $n$ by a Markov process. In this paper, we will use their model to estimate the factorization data of $f^n$ modulo primes for all $n$. We first recall the Markov process:\\

\begin{definition}
\cite{goksel} We introduce a time-homogeneous Markov process $Y_1,Y_2,\dots$ related to $f$. The state space is the set of types of $f$, namely $\{n,s\}^o$, ordered lexicographically. We define the Markov process by giving its transition matrix $ M = (\mathcal{P}(Y_m = T_j | Y_{m-1} = T_i))$, where $Y_i$ and $Y_j$ vary over all types. To define $M$, we assume that the \textbf{allowable types} of children arise with equal probability. To define an allowable type: Note that the polynomial $f$ acts on the state space as follows: If $T$ is a type, then $f(T)$ is another type obtained by shifting each entry one position to the left, and using the $m$th entry of $T$ as the final entry of $f(T)$, where $m$ is the integer such that $f^{o+1}(c) = f^m(c).$ If $g$ has type starting with $n$, then it has only one child with type $f(T)$, which is the only allowable type. If the type of $g$ starts with $s$, then it has two children with types $T_1$ and $T_2$ such that $T_1T_2 = f(T).$ Now we call the pair of types $(T_1,T_2)$ allowable if it satisfies the conditions in Lemma $2.6$.
\end{definition}

Our next goal is to describe our model to estimate the factorization data of the large iterates of $f$ modulo primes. We first need the following definition:

\begin{definition}
Let $K$ be a field, and $f(x)\in K[x]$ be a PCF quadratic polynomial. A \textbf{level n datum} associated to $f$ is defined to be the pair of objects $(\prod_{i=1}^{k}[t_i,n_i],r)$, where $t_i$ for each $i=1,\dots,k$ is a string of length $|O_f|$ with letters $n$ and $s$, $[n_1,\dots,n_k]$ is a partition of $2^n$, and $0 \leq r\leq 1$. Then we define a \textbf{level n data} to be a collection of level $n$ datum whose second coordinates add up to $1.$
\end{definition}
\begin{example}
Let $f(x) = (x+3)^2-3\in \mathbb{Z}[x]$. We have $O_f = \{-3\}$. Hence, the set of all possible types for the factors of iterates of $f$ is $\{n,s\}$. First consider the primes $p$ such that $\bar{f}(x)\in \mathbb{F}_p[x]$ is irreducible, i.e. $(\frac{3}{p}) = -1$, i.e. $p\equiv \pm5$ (mod $12$). In this case, we have that $f(-3)=-3$ is a square in $\mathbb{F}_p$ iff $p\equiv -5$ (mod $12$), which correspond to $\frac{1}{4}$ of all primes. Hence, $([s,2],\frac{1}{4})$ becomes a level $1$ datum. Similarly, $([n,2],\frac{1}{4})$ also becomes a level $1$ datum. Next consider the primes $p$ such that $\bar{f}(x)\in \mathbb{F}_p[x]$ is reducible, i.e. $(\frac{3}{p}) = 1$, i.e. $p\equiv \pm1$ (mod $12$). It follows by an elementary calculation that $\bar{f}(x)\in \mathbb{F}_p[x]$ has factorization type $[n,1][n,1]$ for $\frac{1}{8}$ of all primes, hence $([n,1][n,1],\frac{1}{8})$ becomes a level $1$ datum. Similarly, each of $([s,1][s,1],\frac{1}{8})$ and $([n,1][s,1],\frac{1}{4})$ also becomes a level $1$ datum. Using these, the set $\{([n,2],\frac{1}{4}), ([s,2],\frac{1}{4}), ([n,1][s,1],\frac{1}{4}), ([n,1][n,1],\frac{1}{8}),([s,1][s,1],\frac{1}{8})\}$ becomes a level $1$ data.
\end{example}

Next example gives us a generalization of Example $2.9$: 
\begin{example}
Let $f(x)\in \mathbb{Z}[x]$ be a PCF quadratic polynomial with post-critical orbit $O_f$. If we fix a factorization type $T$, and define $r$ to be the density of primes $p$ such that $\bar{f}\in \mathbb{F}_p[x]$ has factorization type $T$, then $(T,r)$ defines a level $1$ datum. If we calculate each possible datum over all primes, then their collection gives a level $1$ data.
\end{example}

Now we will describe our model based on Example $2.10$. We take a PCF quadratic polynomial $f(x)\in \mathbb{Z}[x]$, and calculate its level $1$ data. Then for each $n\geq 1$, we can apply the Markov process described above to level $n$ data, and get the level $n+1$ data. This model will estimate the factorization data of $f^n$ modulo primes for all $n$. However, we need a little twist in our Markov process for the linear factors when $p$ is a prime of the form $4k+3$, for the following reason: If $g$ is a linear polynomial with type starting with $n$ (resp. $s$), and if look at the composition $g(f(x))$ mod $p$ for $p\equiv 1 (\text{mod }4)$, then $g(f(x))$ stays irreducible (resp. factors). However, if $g$ is a linear polynomial with type starting with $n$ (resp. $s$), and if look at the composition $g(f(x))$ mod $p$ for $p\equiv 3 (\text{mod }4)$, then $g(f(x))$ factors (resp. stays irreducible). This is due to (Lemma $2.5$, \cite{Settled}), since in the linear case $-1$ being a square or not in $\mathbb{F}_p$ makes a difference. For example, let $f(x)= (x+1)^2-2$. Then a linear factor $g$ of type $nn$ $(\text{mod } p)$ for $p\equiv 1 (\text{mod }4)$ will lead to a type $[nn,2]$, but it will lead to $[nn,1][ss,1]$ and $[ns,1][sn,1]$ equiprobably for the primes $p \equiv 3 (\text{mod } 4).$\\

Because of the twist explained above, our model for each level will have two different parts. The data corresponding to the primes of the form $4k+1$ will be called \textbf{even data}, and the data corresponding to the primes of the form $4k+3$ will be called \textbf{odd data}.\\

We now give a detailed example to explain how we get each level's data:

\begin{example}
Let $f(x)=(x+a)^2-a-1\in \mathbb{Z}[x]$, where $a\in \mathbb{Z}$ is such that $a\neq \pm b^2$ for any $b\in \mathbb{Z}.$  If $g(x)\in \mathbb{Z}[x]$ is any irreducible polynomial, using Lemma $2.6$, we have the following possible transitions for $\overline{g(x)} \rightarrow \overline{g(f(x))}$ (mod $p$), depending on whether $p$ is $1$ or $3$ (mod $4$). Note that $k$ is greater than $0$.\\

\underline{p $\equiv$ 1 (mod 4)}\\

$[nn,1] \rightarrow [nn,2]$\\

$[ns,1] \rightarrow [sn,2]$\\

$[sn,1] \rightarrow [ns,1][ss,1] \text{ or } [nn,1][sn,1]$\\

$[ss,1] \rightarrow [nn,1][nn,1]\text{ or }[ns,1][ns,1]\text{ or }[sn,1][sn,1]\text{ or }[ss,1][ss,1]$\\

$[nn,2^k] \rightarrow [nn,2^{k+1}]$\\

$[ns,2^k] \rightarrow [sn,2^{k+1}]$\\

$[sn,2^k] \rightarrow [ns,2^k][ss,2^k] \text{ or } [nn,2^k][sn,2^k]$\\

$[ss,2^k] \rightarrow [nn,2^k][nn,2^k]\text{ or }[ns,2^k][ns,2^k]\text{ or }[sn,2^k][sn,2^k]\text{ or }[ss,2^k][ss,2^k].$\\

\underline{p $\equiv$ 3 (mod 4)}\\

$[nn,1] \rightarrow [nn,1][ss,1]$ or $[ns,1][sn,1]$\\

$[ns,1] \rightarrow [nn,1][ns,1]$ or $[ss,1][sn,1]$\\

$[sn,1] \rightarrow [ns,2]$\\

$[ss,1] \rightarrow [ss,2]$\\

$[nn,2^k] \rightarrow [nn,2^{k+1}]$\\

$[ns,2^k] \rightarrow [sn,2^{k+1}]$\\

$[sn,2^k] \rightarrow [ns,2^k][ss,2^k] \text{ or } [nn,2^k][sn,2^k]$\\

$[ss,2^k] \rightarrow [nn,2^k][nn,2^k]\text{ or }[ns,2^k][ns,2^k]\text{ or }[sn,2^k][sn,2^k]\text{ or }[ss,2^k][ss,2^k].$\\

By a direct elementary computation, the first level data for $f$ is as follows:\\

\underline{Level $1$ - Even data}\\

$([nn,2],\frac{1}{8}),([sn,2],\frac{1}{8}) ,([nn,1][sn,1], \frac{1}{16}), ([ns,1][ss,1], \frac{1}{16})$,\\ 

$([nn,1][nn,1], \frac{1}{32}), ([ns,1][ns,1], \frac{1}{32}), ([sn,1][sn,1], \frac{1}{32}), ([ss,1][ss,1], \frac{1}{32}).$\\

\underline{Level $1$ - Odd data}\\

$([nn,1][ns,1],\frac{1}{16}), ([nn,1][ss,1],\frac{1}{16}), ([ns,1][sn,1],\frac{1}{16}), ([ns,2],\frac{1}{8}),$ \\

$([sn,1][ss,1],\frac{1}{16}), ([ss,2],\frac{1}{8}).$\\

If we combine the data with the same cycle types, we obtain the cycle data set $\{([1,1],\frac{1}{2}),([2],\frac{1}{2})\}$, which corresponds to the group $W_1$.\\

To get the data for the second level, we apply the Markov process. For instance, if we consider the datum $([nn,2],\frac{1}{4})$ in the even data, it will give the datum $([nn,4],\frac{1}{4})$, where the density does not change because we only have one allowable type, namely $nn$. Similarly, the datum $([sn,1][sn,1]$ will give the data $([ns,1][ss,1][ns,1][ss,1],\frac{1}{64})$, $([nn,1][sn,1][nn,1][sn,1],\frac{1}{64})$ and $([ns,1][ss,1][nn,1][sn,1],\frac{1}{32})$, because $[sn,1]$ leads to $[nn,1][sn,1]$ or $[ns,1][ss,1]$ equiprobably.\\

On the other hand, for instance, if we consider the datum $([nn,1][ss,1],\frac{1}{16})$ in the odd data, it leads to the data $([nn,1][ss,1][ss,2],\frac{1}{32})$ and $([ns,1][sn,1][ss,2],\frac{1}{32})$, because $[ss,1]$ leads to the unique allowable type $[ss,2]$, while $[nn,1]$ leads to $[nn,1][ss,1]$ or $[ns,1][sn,1]$ equiprobably.\\

For the sake of clarity, we now give the second level data as well. Doing the calculations as described above, we get the following level $2$ data:\\

\underline{Level $2$ - Even data}\\

$([nn,4],\frac{1}{8}),([nn,2][sn,2],\frac{1}{16}) ,([ns,2][ss,2],\frac{1}{16}),([nn,2][nn,1][sn,1], \frac{1}{32}),([nn,2][ns,1][ss,1], \frac{1}{32})$, \\

$([sn,2][nn,1][nn,1], \frac{1}{64}), ([sn,2][ns,1][ns,1], \frac{1}{64}), ([sn,2][sn,1][sn,1], \frac{1}{64}), ([sn,2][ss,1][ss,1], \frac{1}{64})$,\\ 

$([nn,2][nn,2], \frac{1}{32}), ([sn,2][sn,2], \frac{1}{32}), ([nn,1][sn,1][nn,1][sn,1], \frac{1}{128}),([ns,1][ss,1][ns,1][ss,1], \frac{1}{128})$,\\

$([nn,1][sn,1][ns,1][ss,1], \frac{1}{64}),([nn,1][nn,1][nn,1][nn,1], \frac{1}{512}), ([ns,1][ns,1][ns,1][ns,1], \frac{1}{512})$,\\

 $([sn,1][sn,1][sn,1][sn,1], \frac{1}{512}), ([ss,1][ss,1][ss,1][ss,1], \frac{1}{512}), ([nn,1][nn,1][ns,1][ns,1],\frac{1}{256})$,\\
 
  $([nn,1][nn,1][sn,1][sn,1],\frac{1}{256}), ([nn,1][nn,1][ss,1][ss,1],\frac{1}{256}), ([ns,1][ns,1][sn,1][sn,1],\frac{1}{256})$,\\
  
   $([ns,1][ns,1][ss,1][ss,1],\frac{1}{256}), ([sn,1][sn,1][ss,1][ss,1],\frac{1}{256})$.\\

\underline{Level $2$ - Odd data}\\

$([nn,1][ss,1][nn,1][ns,1],\frac{1}{64}),([nn,1][ss,1][ss,1][sn,1],\frac{1}{64}) ,([ns,1][sn,1][nn,1][ns,1],\frac{1}{64})$,\\

$([ns,1][sn,1][ss,1][sn,1],\frac{1}{64}),
([nn,1][ss,1][ss,2],\frac{1}{32}),([ns,1][sn,1][ss,2],\frac{1}{32}) ,([nn,1][ns,1][ns,2],\frac{1}{32})$,\\

$([ss,1][sn,1][ns,2],\frac{1}{32}) ,([sn,4],\frac{1}{8}), ([ns,2][ss,2],\frac{1}{16}), ([nn,2][nn,2],\frac{1}{32}), ([ns,2][ns,2],\frac{1}{32})$,\\ 

$([sn,2][sn,2],\frac{1}{32}), ([ss,2][ss,2],\frac{1}{32}).$\\

If we combine the data with the same cycle types, we obtain the cycle data set $$\{([4],\frac{1}{4}),([2,2],\frac{3}{8}), ([2,1,1],\frac{1}{4}),([1,1,1,1],\frac{1}{8})\},$$which corresponds to the group $W_2$.\\

Thus, by iteratively applying this process we can get level $n$ data for all $n\geq 1$. Jones and Boston made this model in the hope that the level $n$ data for all $n$ may give the actual factorization data of $f^n$, hence the cycle structure of the Galois group of $f^n$ (by Chebotarev's density theorem), which is enough for recovering the Galois group most of the time. However, they observed that at level $5$, the cycle data arising from the Markov model fails to match the actual cycle data of the corresponding Galois group.\\

For any $n\geq 1$, by further computations, we observed that although the Markov model Jones and Boston suggested does not always give the cycle data of the actual Galois group of $f^n$, it still appears to give a cycle data that corresponds to an actual subgroup of $\text{Aut}(T_n)$. We will study a refinement of this phenomenon for some families of PCF quadratic polynomials.
\end{example}

In this paper, we will focus on the PCF quadratic polynomials with integer coefficients. There are three such families of polynomials; Namely, they are conjugates of the polynomials $x^2$, $x^2-1$ and $x^2-2$ under the linear map $x\mapsto x+a$, $a\in \mathbb{Z}$. We state the problem we described in Example $2.11$ in a more general and precise way:

\begin{question}
Let $f(x)\in \mathbb{Z}[x]$ be a PCF quadratic polynomial. For any $n\geq 1$, consider the level $n$ cycle data given by the Markov model of $f$. Does there exist a permutation group $M_n^*(f)\leq \text{Aut}(T_n)$ whose cycle data match the level $n$ cycle data given by the Markov model?
\end{question}

This version of the problem appears to be difficult. See the last section for a discussion of it. Next we will introduce a refined version of this problem. Let $n\geq 1$, and consider only the even data for the level $n$. The corresponding densities will add up to $\frac{1}{2}$, so we first normalize them by multiplying each density by $2$. We will call this new model \textbf{even Markov model}. We have the following question, which we will study in this article:

\begin{question}
Let $f(x)\in \mathbb{Z}[x]$ be a PCF quadratic polynomial. For any $n\geq 1$, consider the level $n$ cycle data given by the even Markov model of $f$. Does there exist a permutation group $M_n(f)\leq \text{Aut}(T_n)$ whose cycle data match the level $n$ cycle data given by the even Markov model?
\end{question}

Let $G\leq \text{Aut}(T_n)$ be a subgroup of $\text{Aut}(T_n)$. If the cycle data of $G$ match the level $n$ cycle data given by the even Markov model of $f$, then we say \textbf{$\boldsymbol{G}$ satisfies the even Markov model of $\boldsymbol{f}$.} We also sometimes say \textbf{$\boldsymbol{G}$ is a level $\boldsymbol{n}$ even Markov group of $\boldsymbol{f}$}.\\

Before we finish this section, we introduce what is called a \textbf{restricted Markov model}, since we will need it as a tool in the next sections.

\begin{definition}
Let $q\equiv 1 (\text{mod }4)$ be a prime power, and $f(x)\in \mathbb{F}_q[x]$. To define a \textbf{restricted Markov process associated to} $\boldsymbol{f}$, we make the following modification in the Markov process described in Definition $2.7$: Using the notation in Definition $2.7$, recall that if $T$ is a type that starts with $s$, then there are more than one allowable pair of types with $f(T) = T_1T_2$. For each type $T$ that starts with $s$, we choose a unique allowable pair of types $(T_1,T_2)$ with $f(T)=T_1T_2$ that is assumed to arise with $100\%$ probability.
\end{definition}

Following example illustrates a restricted Markov model.
\begin{example}
Let $f(x) = (x+1)^2-2\in \mathbb{F}_5[x]$. $f$ has the post-critical orbit $O_f = \{3,4\}$, hence it has orbit size $2$ and tail size $0$. Using Lemma $2.6$, we can define a restricted Markov process by giving the following transitions. In what follows, $k\geq 0$.\\

$[nn,2^k] \rightarrow [nn,2^{k+1}]$, $[ns,2^k]\rightarrow [sn,2^{k+1}]$, $[sn,2^k]\rightarrow [nn,2^k][sn,2^k]$, $[ss,2^k]\rightarrow [ss,2^k][ss,2^k].$
\end{example}
\section{Preliminaries from group theory}
Let $T$ be the complete rooted binary tree. We denote by $T_n$ the truncation of $T$ to the first $n$ levels. We use the notations $W:=$ Aut$(T)$ and $W_n :=$ Aut$(T_n)$ for the automorphism groups of $T$ and $T_n$, respectively. It is well-known that $W_n$ is isomorphic to the $n$-fold wreath product of $\mathbb{Z}/2\mathbb{Z}$. We also have $W = \varprojlim W_n$, via the natural restriction maps $\pi_n : W_n \twoheadrightarrow W_{n-1}$. \\

Throughout the paper, we will use the standard minimal set of generators of $W_n$, namely $a_1 = (1,2)$, $a_2 = (1,3)(2,4)$, $a_3 = (1,5)(2,6)(3,7)(4,8)$, $\dots ,a_n = (1,2^{n-1}+1)(2,2^{n-1}+2) \dots
(2^{n-1},2^n)$.
\begin{definition}
	We call $w\in W$ an \textbf{odometer} if $w$ acts transitively on $T_n$ for all $n\geq 1$. 
\end{definition}
\begin{definition}
	We call $w_n\in W_n$ a \textbf{n-odometer} if $w_n$ acts transitively on $T_n$.
\end{definition}
Note that any $n$-odometer $w_n \in W_n$ is the image of an odometer $w\in W$ under the natural projection $W\twoheadrightarrow W_n$.
\begin{lemma}
	$x_n := a_1a_2\dots a_n \in W_n$ is an $n$-odometer.
\end{lemma}
\begin{proof}
	We start by recalling an alternative definition of $W$: \begin{equation}
	W = (W \times W) \rtimes \langle \sigma \rangle. 
	\end{equation} Here, $\sigma\in W$ is the unique automorphism of order $2$ that interchanges the two half trees. Then the standard odometer is defined by the recursion relation \begin{equation}
	w = (w,1)\sigma.
	\end{equation} (\cite{Pink}, p.16)
	Taking the image of both sides of $(3.2)$ under the natural projection $W\twoheadrightarrow W_n$, we get \begin{equation}
	w_n = (w_{n-1}, 1)a_n,
	\end{equation} since $\sigma$ acts on $T_n$ by $a_n$. The result directly follows using $(3.3)$ and by induction.
\end{proof}
\begin{definition}
	We call $x_n\in W_n$ in Lemma $3.3$ the \textbf{standard n-odometer}.
\end{definition}
\begin{definition}
	Let $w_n \in W_n$ be any $n$-odometer, and $G\leq W_n$ be any subgroup of $W_n$. Then we define the set $G^{od}[w_n]$ by \begin{equation}
	G^{od}[w_n] = \{g\in G \text{ }| \text{ } w_ng \text{ is a n-odometer }\}.
	\end{equation}
\end{definition}
Next, we will prove that for any $n$-odometer $w_n\in W_n$, and $G\leq W_n$, the subset $G^{od}[w_n]$ of $G$ is a normal subgroup of $G$. Before doing that, we need to recall the notion of Frattini subgroup: For a group $X$, the Frattini subgroup of $X$, denoted by $\Phi(X)$, is defined by the intersection of all maximal subgroups of $X$. It is well-known that for a $2$-group $X$, $\Phi(X)$ is generated by squares and commutators, i.e. \begin{equation}
\Phi(X) = X^2[X,X].
\end{equation}
\begin{lemma}
	$W_n/\Phi(W_n) \cong (\mathbb{Z}/2\mathbb{Z})^n$.
\end{lemma}
\begin{proof}
	See (\cite{Harpe}, p.$215$) for a proof.
\end{proof}
\begin{lemma}
	Let $w_n\in W_n$ be any $n$-odometer. Then $W_n^{od}[w_n] = \Phi(W_n)$.
\end{lemma}
\begin{proof}
	First note that any two $n$-odometers are conjugate under $W_n$ (\cite{Pink}, Proposition $1.6.2$). Hence, if $w_ng$ is an $n$-odometer, this gives $w_ng = w_n^x$ for some $x\in W_n$, which is equivalent to say $g = [w_n^{-1}, x]\in \Phi(W_n)$. This shows that $W_n^{od}[w_n] \subset \Phi(W_n)$.\\
	
	By the definition, $W_n^{od}[w_n]$ has the same size as the set of all $n$-odometers of $W_n$. Hence, since we have already proven that $W_n^{od}[w_n] \subset \Phi(W_n)$, we will be done if we can show that $W_n$ has exactly $|\Phi(W_n)|$ many $n$-odometers. We will do induction to prove this: It is clear for $n=1$. Suppose $W_k$ has exactly $|\Phi(W_k)|$ many $k$-odometers for some $k\geq 1$. Note that for any $(k+1)$-odometer $w_{k+1}\in W_{k+1}$, $\pi_{k+1}(w_{k+1})$ is a $k$-odometer in $W_k$. Also, for any $k$-odometer $w_k \in W_k$, it is easy to see that exactly half of the elements in $\pi_{k+1}^{-1}(w_k)$ are $(k+1)$-odometers, and the other half are the products of two disjoint cycles of length $2^k$. Hence, by the induction assumption on $k$, it follows that $W_{k+1}$ has exactly $\frac{|\Phi(W_k)||\text{Ker}(\pi_{k+1})|}{2}$ many $(k+1)$-odometers. But, we also have $$\frac{|\Phi(W_k)||\text{Ker}(\pi_{k+1})|}{2} = \frac{(2^{2^{k-1}-k+1})(2^{2^{k-1}})}{2} = 2^{2^k-k} = |\Phi(W_{k+1})|,$$ which finishes the proof. Note that we used Lemma $3.6$ for the first equality.
\end{proof}
\begin{corollary}
	Let $w_n\in W_n$ be any $n$-odometer, and $G\leq W_n$. Then we have $G^{od}[w_n] \trianglelefteq G$.
\end{corollary}
\begin{proof}
	We have $G^{od}[w_n] = G\cap \Phi(W_n)$ by Lemma $3.7$, which directly gives the result. 
\end{proof}
\begin{corollary}
	For any two $n$-odometers $w_n, w_n' \in W_n$ and $G\leq W_n$, we have $G^{od}[w_n] = G^{od}[w_n']$.
\end{corollary}
\begin{proof}
	This is a direct consequence of the proof of Corollary $3.8$.
\end{proof}
\textbf{In the light of Corollary $\boldsymbol{3.9}$, for any $\boldsymbol{n}$-odometer $\boldsymbol{w_n\in W_n}$ and $\boldsymbol{G\leq W_n}$, we let $\boldsymbol{G^{od}:= G^{od}[w_n]}$, which is the notation we will use in the rest of the paper.}
\begin{lemma}
	Let $H,G\leq W_n$ be subgroups of $W_n$ such that $H \trianglelefteq G$. Then we have $H^{od} \trianglelefteq G$. 
\end{lemma}
\begin{proof}
	Recall that $\Phi(W_n) \trianglelefteq W_n$. Thus, since $H \trianglelefteq G$, for any $x\in H^{od} = H \cap \Phi(W_n)$ and $g\in G$, we have $x^g \in \Phi(W_n)$ and $x^g \in H$, which gives $x^g \in H \cap \Phi(W_n) = H^{od}$, which finishes the proof.
\end{proof}
\begin{customrmk}{}
From now on, whenever we say $x\in W_n$ (resp. $G\leq W_n$) for an element $x$ of $W_i$ (resp. subgroup $G$ of $W_i$) for some $i<n$, we do so by identifying $x$ (resp. $G$) with its image under the natural inclusion $W_i \rightarrow W_n$.
\end{customrmk}
\begin{lemma}
	Suppose $x\in W_n$ and $H\leq W_n$ both act trivially on the same half tree. Then $x^{x_n}$ commutes with each element of $H$.
\end{lemma}
\begin{proof}
	We assume without loss of generality that $x$ acts trivially on the right half tree, since the other case follows similarly. By the definition of $x_n$, we have $x_n = x_{n-1}a_n$. Then we get $x^{x_n} = (x^{a_n})^{x_{n-1}} = x^{a_n}$, because $\text{Supp}(x^{a_n}) \cap \text{Supp}(x_{n-1}) = \emptyset$, since $x_{n-1}$ acts trivially on the right half tree and $x^{a_n}$ acts trivially on the left half tree, where the latter fact is clear from the definition of $a_n$. So, we get $x^{x_n} = x^{a_n}$, which acts trivially on the left half tree. Hence, by the assumption on $H$, for any $h\in H$, we have $\text{Supp}(h) \cap \text{Supp}(x^{x_n}) = \emptyset$, which shows that $h$ commutes with $x^{x_n}$.
\end{proof}
\begin{definition}
	If $x \in W_n$ is the product of disjoint cycles of lengths $n_1, n_2, \dots, n_r$ with $n_1\leq n_2\leq \dots \leq n_r$ (including its $1$-cycles) then the vector $[n_1,n_2,\dots n_r]$ is called the \textbf{cycle type} of $x$. We denote the cycle type of $x$ by $c(x)$.
\end{definition}

To illustrate Definition $3.12$, if we take $ x = (1,3,2,4)(5,6) \in W_3$, we have $c(x) = [1,1,2,4]$.
\begin{definition}
	Let $x\in W_n$, and consider a pre-image $y\in \pi_{n+1}^{-1}(x)$ of $x$. Let $c_x = (c_1,\dots,c_k)$ be one of the cycles in the disjoint cycle decomposition of $x$, and suppose the cycle $c_y = (\alpha_1,\dots,\alpha_{2k})$ appears in the cycle decomposition of $y$, where the $2k$-tuple $(\alpha_1,\dots,\alpha_{2k})$ is a permutation of the $2k$-tuple $(2c_1-1,\dots,2c_k-1,2c_1,\dots,2c_k)$. Then we call $c_y$ a \textbf{doubling} of $c_x$. Similarly, if $d_x = (d_1,\dots,d_l)$ is one of the cycles in the disjoint cycle decomposition of $x$, and the cycle product $d_y = (\beta_1,\dots,\beta_l)(\gamma_1,\dots,\gamma_l)$ appears in the cycle decomposition of $y$, where the $l$-tuple $(\beta_1,\dots,\beta_l)$ (resp. $(\gamma_1,\dots,\gamma_l)$) is a permutation of the $l$-tuple $(2d_1-1,\dots,2d_l-1)$ (resp. $(2d_1,\dots,2d_l)$), then we call $d_y$ a \textbf{splitting} of $d_x$. In particular, if $c_y = (2c_1-1,\dots,2c_k-1,2c_1,\dots,2c_k)$, we say $c_y$ is the \textbf{standard doubling} of $c_x$, and if $d_y = (2d_1-1,\dots,2d_l-1)(2d_1,\dots,2d_l)$, we say $d_y$ is the \textbf{standard splitting} of $d_x$.
\end{definition}
\begin{example}
	Let $x=(1,2)(3,4)\in W_2$, and $y=(1,3,2,4)(5,7)(6,8)$ so that $\pi_3(y) = x$. Then the cycle $(1,3,2,4)$ is the standard doubling of $(1,2)$, and the cycle product $(5,7)(6,8)$ is the standard splitting of $(3,4)$. On the other hand, if $z = (1,4,2,3)(5,8)(6,7)\in W_3$, we again have $\pi_3(z) = x$, and $(1,4,2,3)$ is a doubling of $(1,2)$, whereas $(5,8)(6,7)$ is a splitting of $(3,4)$.
\end{example}
\begin{definition}
Let $c_1:=[n_1,\dots, n_k]$ and $c_2 = [m_1,\dots, m_l]$ be two partitions of $2^n$ such that $n_1\leq \dots \leq n_k$ and $m_1\leq \dots \leq m_l$. Then $c_1*c_2$ is a partition of $2^{n+1}$ defined by
$$c_1*c_2 = [\alpha_1,\dots,\alpha_{k+l}],$$
where the $(k+l)$-tuple $(\alpha_1,\dots,\alpha_{k+l})$ is the non-descending ordering of the $(k+l)$-tuple \\
$(n_1,\dots, n_k,m_1,\dots, m_l)$.
\end{definition}
\begin{definition}
	Let $S_1\subseteq S_2\subseteq W_n$ be two subsets of $W_n$. We define the \textbf{cycle data of} $\boldsymbol{S_1}$ \textbf{relative to} $\boldsymbol{S_2}$ to be the set of pairs $(c,q_c)$, where $c$ is a cycle type that exists in $S_1$, and $q_c = p_c\frac{|S_1|}{|S_2|}$, where $p_c$ is the proportion of the elements of $S_1$ with cycle type $c$. We denote this set by $CD(S_1,S_2)$.  
\end{definition}
\begin{example}
	Let $S_1 = \langle(1,3,2,4)\rangle\leq W_2$, and $S_2 = W_2$. Then we have $$S_1 = \{\text{id},(1,2)(3,4),(1,3,2,4),(1,4,2,3)\}.$$ This gives $$CD(S_1,S_2)=\{([1,1,1,1],\frac{1}{8}),([2,2],\frac{1}{8}),([4],\frac{1}{4})\}.$$
\end{example}
\begin{definition}
	Let $A$ and $B$ be two sets given by $$A = \{(c_1,p_1),\dots, (c_k,p_k)\},\text{ } B = \{(d_1,q_1),\dots, (d_l,q_l)\},$$ where $c_i, d_j$ are some partitions of $2^n$, and $p_i,q_j$ are such that $0\leq p_i, q_j \leq 1$ for all $i=1,\dots ,k$ and $j=1, \dots ,l$. We define the product $A\times B$ by $$A\times B = \{(c_i*d_j, p_iq_j) | (c_i, p_i) \in A, (d_j, q_j)\in B\}.$$
\end{definition}
\begin{remark}
	Of course, one may have $c_i*d_j = c_i'*d_j'$ for some $(c_i,d_j) \neq (c_i', d_j')$. In this case, we do the following: Suppose we have $(c_{i_1}*d_{j_1}, p_{i_1}q_{j_1}), (c_{i_2}*d_{j_2}, p_{i_2}q_{j_2}),\dots , (c_{i_t}*d_{j_t}, p_{i_t}q_{j_t})$ with $c:= c_{i_1}*d_{j_1} = c_{i_2}*d_{j_2} = \dots = c_{i_t}*d_{j_t}$ in our product set. Then instead of writing these pairs separately, we only write the single element $(c, \sum_{m=1}^{t} p_{i_m}q_{j_m})$.
\end{remark}
\begin{lemma}
	Let $x\in W_n$ satisfy the following two properties:\\
	\item[(i)] $x$ acts trivially on one of the half trees.
	\item[(ii)] $c(x) = [1,1,\dots,1,c_1,c_2,\dots, c_k]$ such that $\sum_{i=1}^{k}c_i = 2^{n-1}$.\\
	
	Then we have $c(a_n x) = [2c_1,2c_2,\dots, 2c_k]$.
\end{lemma}
\begin{proof}
	We assume without loss of generality that $x$ acts trivially on the right half tree. Let $$x = (p_{11},\dots,p_{1c_1})\dots (p_{k1},\dots,p_{kc_k}).$$ It follows from the definition of $a_n$ that $$a_nx = (p_{11},p_{11}+2^{n-1},\dots,p_{1c_1},p_{1c_1}+2^{n-1})\dots (p_{k1},p_{k1}+2^{n-1},\dots,p_{kc_k},p_{kc_k}+2^{n-1}),$$ which finishes the proof.
\end{proof}
\begin{customnt}{}
Let $\alpha=[n_1,\dots,n_i]$ be partition of $2^n$. We denote by $d\alpha$ the partition $[2n_1,\dots,2n_i]$ of $2^{n+1}$, where the author justifes the notation by the fact that  each coordinate is \textbf{d}oubled. Let $X_n := \{(\alpha_1,p_1),\dots,(\alpha_k,p_k)\}$ be the set of pairs, where each $\alpha_i$ is a partition of $2^n$, and $0\leq p_i\leq 1$. We define $dX_n$ by
$$dX_n = \{(d\alpha_1,p_1),\dots,(d\alpha_k,p_k)\}.$$
Finally, for any positive real number $r$ such that $rp_i\leq 1$ for all $i=1,\dots
 k$, we define the set $rX_n$ by
 $$rX_n = \{(\alpha_1,rp_1),\dots, (\alpha_k, rp_k)\}.$$
\end{customnt}
Next, we will define the the notion of a \textbf{Markov map}, which will be a bridge between the Markov process we defined in the previous section and groups $W_n$. This notion will be an important tool in our constructions in the rest of the paper:

\begin{definition}
Let $q\equiv 1(\text{mod }4)$, and $f(x)\in \mathbb{F}_q[x]$ be a quadratic polynomial. We first fix a restricted Markov process associated to $f$, in the sense of Definition $2.14$. Let $x\in W_n$, and suppose $x$ has the disjoint cycle decomposition $x = c_1c_2\dots c_k$, where the decomposition may possibly include some trivial cycles as well. We formally attach some type $T_i$ to each $c_i$ using the set of types in the Markov process of $f$. Then a \textbf{level $\boldsymbol{n}$ Markov map} $m_f^{(n)}$ associated to the restricted Markov process of $f$ is defined by the map that sends $x\in W_n$ to the product $y= d_1d_1'\dots d_kd_k'\in W_{n+1}$, where the product $d_id_i'$ for each $i$ satisfies the following:\\

\item[(i)] If $T_i$ starts with $n$, then $d_i$ is the standard doubling of $c_i$, and is given the formal type $f(T_i)$, and $d_i'$ is defined to be the identity element of $W_{n+1}$.\\

\item[(ii)] If $T_i$ starts with $s$, then $d_id_i'$ is the standard splitting of $c_i$, and $d_i$ and $d_i'$ are given the formal types $T_{i1}$ and $T_{i2}$ respectively, where $(T_{i1},T_{i2})$ is the pair of types that is dictated by the restricted Markov process.
\end{definition}

Finally, we fix some notation for some elements of $W_n$ that we will frequently use throughout the paper: $v_0:= x_2^2$, $v_1 := x_3^2$, $\dots$,$v_{n-3}:=x_{n-1}^2$. We also make the convention that $v_{-2} = v_{-1} = \text{id}\in W_n$.
\section{Warm-up case: conjugates of $x^2$}
In this section, we will answer Question $2.13$ in the affirmative for the polynomials $f_a(x) = (x+a)^2-a$, $a\in \mathbb{Z}$. Throughout, we assume that $f_a$ is irreducible. It follows from an elementary calculation that for this family we have two different models, depending on whether $a=-b^2$ for some $b\in \mathbb{Z}$ or not. Below, for each different model, we will only give the first level data that is obtained by factoring $f_a$ modulo primes of the form $4k+1$.  They are as follows:\\

	\textbf{\underline{Model 1 (for $\boldsymbol{f_a(x)}$ such that $\boldsymbol{a\neq -b^2}$ for any $\boldsymbol{b\in \mathbb{Z}}$)}}\\ 

$([n,2],\frac{1}{2}), ([n,1][n,1], \frac{1}{4}), ([s,1][s,1], \frac{1}{4})$\\

	\textbf{\underline{Model 2 (for $\boldsymbol{f_a(x)}$ such that $\boldsymbol{a = -b^2}$ for some $\boldsymbol{b\in \mathbb{Z}}$)}} \\

$([n,1][n,1], \frac{1}{2}), ([s,1][s,1], \frac{1}{2})$\\

We will first give a proof for Model $1$, and then use this proof to give a proof for Model $2$.
\subsection{Model 1}
We will construct the groups $M_n$ for each level of the Markov model using a Markov map, as given in Definition $3.21$. By direct computation, the transitions for the Markov process are as follows. In what follows, $k\geq 0$.\\

$[n,2^k]\rightarrow [n,2^{k+1}]$, $[s,2^k]\rightarrow [n,2^k][n,2^k]$ or $[s,2^k][s,2^k].$\\

We define the restricted Markov model that we will use in constructing our groups as follows:\\

$[n,2^k]\rightarrow [n,2^{k+1}]$, $[s,2^k]\rightarrow [s,2^k][s,2^k]$.\\

 We let $M_1 = \langle (1,2) \rangle$ and $M_2 = \langle (1,3,2,4) \rangle$. We attach the type $[n,4]$ to $(1,3,2,4)\in M_2$, and attach the type $[n,1]^2[s,1]^2$ to $\text{id}\in M_2$. If we apply the Markov map $m_f^{(2)}$ to $(1,3,2,4), \text{id} \in M_2$, and use the new images as generators for $M_3$, we get $M_3 = \langle (1,5,3,7,2,6,4,8), (1,2)(3,4)\rangle$, where $(1,5,3,7,2,6,4,8)$ has type $[n,8]$, $(1,2)(3,4)$ has type $[n,2]^2[s,1]^4$, and we also have $\text{id}\in M_3$ to which we attach the type $[n,1]^2[s,1]^6$. If we apply the Markov map $m_f^{(3)}$ to these three elements, we obtain $$M_4 = \langle (1,9,5,13,3,11,7,15,2,10,6,14,4,12,8,16), (1,3,2,4)(5,7,6,8), (1,2)(3,4) \rangle.$$ So, in general, for level $n$, one generator will be the standard $n$-odometer, and other $n-2$ generators will have types of the forms $[n,2^i]^2[s,1]^{2^n-2^{i+1}}$ for $i=1,2,\dots,n-2$. In other words, for any $n\geq 3$, we have \begin{equation}
 M_n = \langle x_n, v_{n-3},\dots, v_0\rangle.
 \end{equation}
\begin{customthm}{A}\label{A}
For any $n\geq 1$, let $M_n\leq W_n$ be as above. Then $M_n$ satisfies Model $1$.
\end{customthm}
The rest of this subsection will be devoted to the proof of Theorem $A$. We will first make some preparation for the proof.\\

 For $n\geq 3$, we let $V_n = \langle v_{n-3},\dots, v_0\rangle$, and $N_n:=V_n^{M_n}$ be the normal closure of $V_n$ in $M_n$. We also make the convention that $V_1$ and $V_2$ are the trivial subgroups of $W_1$ and $W_2$, respectively.
 \begin{lemma}
 Let $n\geq 3$. $M_n/N_n \cong \mathbb{Z}/{4\mathbb{Z}}$.
 \end{lemma}
 \begin{proof}
 	Clearly, we have $M_n/N_n = \langle \bar{x}_n \rangle$, where $\bar{x}_n$ is the image of $x_n$ in $M_n/N_n$. We need to show that $(x_n)^4 \in N_n$ and $(x_n)^2 \notin N_n$. The former one is true, because we have $(x_n)^4 = (v_{n-3})(v_{n-3})^{x_n}\in N_n$. To prove the latter one, we first let $X:=\langle (x_n)^2 \rangle$. Note that $2^{n-2}+1 \in X\cdot 1$, but $2^{n-2}+1 \notin N_n\cdot 1$, which shows that $X\not\subseteq N_n$, which finishes the proof. 
 \end{proof}
 \begin{lemma}
 	Let $n\geq 3$ and $P_{n-1} := \langle N_{n-1}, N_{n-1}^{x_n}\rangle \leq W_n$. Then the following is true:
 	\item[(i)] $P_{n-1} \trianglelefteq N_n$.
 	\item[(ii)] $N_n/P_{n-1} \cong \mathbb{Z}/{2\mathbb{Z}}\times \mathbb{Z}/{2\mathbb{Z}}$.
 \end{lemma}
 \begin{proof}
 	\item[(i)] We first prove that $P_{n-1} \leq N_n$. To prove this, it suffices to show $N_{n-1}\leq N_n$, since we have $P_{n-1} = \langle N_{n-1}, N_{n-1}^{x_n}\rangle$, $x_n\in M_n$, and $N_n\trianglelefteq M_n$. Note that \begin{equation}
 	N_{n-1} = V_{n-1}^{M_{n-1}} = V_{n-1}^{\langle x_{n-1}\rangle} = V_{n-1}^{\langle (x_n)^2\rangle},
 	\end{equation} where the last equality follows from the facts that $(x_n)^2 = (x_{n-1})(x_{n-1})^{x_n}$ and $(x_{n-1})^{x_n}$ commutes with each element of $V_{n-1}$ (by Lemma $3.11$). Since $V_{n-1}\leq V_n$, $(x_n)^2\in M_n$, and $N_n\trianglelefteq M_n$, $(4.2)$ directly implies $N_{n-1}\leq N_n$, hence $P_{n-1} \leq N_n$. To prove the normality, it suffices to prove that $P_{n-1}$ is normalized by each generator of $M_n$. $v_i$ already lies in $N_{n-1}\leq P_{n-1}$ for $i=0,1,\dots,n-4$, thus it suffices to show that $P_{n-1}$ is normalized by $x_n$ and $v_{n-3}$. We have \begin{equation}
 	P_{n-1}^{x_n} = \langle N_{n-1}^{x_n}, N_{n-1}^{x_n^2}\rangle = \langle N_{n-1}^{x_n}, N_{n-1}^{(x_{n-1})(x_{n-1})^{x_n}}\rangle = \langle N_{n-1}^{x_n}, N_{n-1}\rangle = P_{n-1},
 	\end{equation} where the third equality is true, because $(x_{n-1})^{x_n}$ commutes with $x_{n-1}$ and with each element of $N_{n-1}$ (by Lemma $3.11$), and $x_{n-1}$ normalizes $N_{n-1}$. Hence, $(4.3)$ shows that $P_{n-1}$ is normalized by $x_n$. To prove that $v_{n-3}$ also normalizes $P_{n-1}$, note that \begin{equation}
 	P_{n-1}^{v_{n-3}}=\langle N_{n-1}^{v_{n-3}}, (N_{n-1}^{x_n})^{v_{n-3}}\rangle = \langle N_{n-1}, N_{n-1}^{x_n}\rangle = P_{n-1},
 	\end{equation} where the second equality follows from the facts that $v_{n-3} = x_{n-1}^2$ and $x_{n-1}$ normalizes $N_{n-1}$, and also that $v_{n-3}$ commutes with each element of $N_{n-1}^{x_n}$ (by Lemma $3.11$). $(4.4)$ shows that $v_{n-3}$ normalizes $P_{n-1}$ too, hence we are done.\\
 	\item[(ii)] We have $N_n = V_n^{M_n} = V_n^{\langle x_n \rangle} = \langle V_{n-1},v_{n-3}\rangle^{\langle x_n \rangle}$. Since $V_{n-1}^{\langle x_n\rangle}\leq P_{n-1}$, and $v_{n-3}^{x_n^2} = v_{n-3}^{(x_{n-1})(x_{n-1}^{x_n})} = v_{n-3}$ (recall that $v_{n-3} = x_{n-1}^2$), we get $N_n/P_{n-1} = \langle \bar{v}_{n-3}, \bar{v}_{n-3}^{x_n}\rangle$. Note that $$v_{n-3}^2 = x_{n-1}^4 = (x_{n-2}^2)(x_{n-2}^2)^{x_{n-1}} = (v_{n-4})(v_{n-4})^{x_{n-1}}\in P_{n-1}.$$ This implies $(v_{n-3}^2)^{x_n}\in P_{n-1}$ as well. Arguing with the orbits as in the proof of Lemma $4.1$, it is clear that $v_{n-3}, v_{n-3}^{x_n}, v_{n-3}v_{n-3}^{x_n} \notin P_{n-1}$. Noting that $v_{n-3}$ commutes with $v_{n-3}^{x_n}$ (by Lemma $3.11$), the proof is completed.
 \end{proof}
 \begin{lemma}
 	$N_n\leq \Phi(W_n)$ for all $n\geq 1$.
 \end{lemma}
 \begin{proof}
 	Since $\Phi(W_n)=W_n^2[W_n,W_n]$, we have $v_i = x_{i+2}^2\in \Phi(W_n)$ for all $i\in \{0,1,\dots,n-3\}$. Because $N_n$ is generated by some conjugates of $v_i$, and $\Phi(W_n)\trianglelefteq W_n$, we get $N_n \leq \Phi(W_n)$.
 \end{proof}
 We are now ready to prove Theorem $A$.\\
 
 By the proof of Lemma $4.1$, we have \begin{equation}
 M_n = \coprod_{i=0}^{3} x_n^i N_n.
 \end{equation}
 We introduce some notation for Model $1$:
 \begin{align*}
 A_1^{(n)}&:= \text{the set of cycle data obtained by applying the Markov process } n-1 \text{ times to }([n,2],\frac{1}{2}). \\
 A_2^{(n)}&:= \text{the set of cycle data obtained by applying the Markov process } n-1 \text{ times to }([n,1][n,1],\frac{1}{4}). \\
 A_3^{(n)}&:= \text{the set of cycle data obtained by applying the Markov process } n-1 \text{ times to }([s,1][s,1],\frac{1}{4}). 
 \end{align*}
 \begin{proposition}
 	For any $n\geq 1$, we have the following equalities:\\
 	\item[(i)] $A_1^{(n)} = \text{CD}(x_nN_n\sqcup x_n^3N_n, M_n).$
 	\item[(ii)] $A_2^{(n)} = \text{CD}(x_n^2N_n, M_n).$
 	\item[(iii)] $A_3^{(n)} = \text{CD}(N_n, M_n).$
 \end{proposition}
\begin{customrmk}{}\label{A}
\textbf{Note that proving Proposition} $\boldsymbol{4.4}$ \textbf{is enough for establishing Theorem A, because it shows that the cycle data of } $\boldsymbol{M_n}$ \textbf{match the cycle data of the} $\boldsymbol{n}$\textbf{th level of the even Markov model.}
\end{customrmk}
 \begin{proof}[Proof of Proposition $4.4$]
 	We first note that all the parts of Proposition $4.4$ trivially hold for $n=1,2$. Throughout, we assume that $n\geq 3$. We will prove each part separately.\\
 	\item[\textbf{Proof of (i).}] Note that by the Markov process, $A_1^{(n)} = \{([2^n],\frac{1}{2})\}$.\\ 
 	
 	Using Lemma $3.7$ and Lemma $4.3$, we have that all the elements in $x_nN_n$ and $x_n^3N_n$ are $n$-odometers, and they correspond to the half of the elements of $M_n$, which shows that $A_1^{(n)} = \text{CD}(x_nN_n\sqcup x_n^3N_n,M_n).$
 	\item[\textbf{Proof of (ii).}] By the Markov process, we have $A_2^{(n)} = \{([2^{n-1},2^{n-1}],\frac{1}{4})\}$.\\
 	
 	By the proof of Lemma $4.2$, we have \begin{equation}
 	N_n = \coprod_{i,j \in \{0,1\}} (v_{n-3})^i(v_{n-3}^{x_n})^jP_{n-1}.
 	\end{equation}
 	Since $P_{n-1} = \langle N_{n-1}, N_{n-1}^{x_n}\rangle$, and $N_{n-1}$ acts independently from $N_{n-1}^{x_n}$ on the tree (by the proof of Lemma $3.11$), we can write $P_{n-1} = N_{n-1}\times N_{n-1}^{x_n}$. Thus, noting that $v_{n-3} = x_{n-1}^2$, $x_n^2 = x_{n-1}x_{n-1}^{x_n}$ and using $(4.6)$, we get
 	\begin{equation}
 	x_n^2N_n = \coprod_{i,j \in \{0,1\}} (x_{n-1}^{2i+1}N_{n-1})\times (x_{n-1}^{2j+1}N_{n-1})^{x_n}.
 	\end{equation}
 	Since $x_{n-1}^{2i+1}$ and $x_{n-1}^{2j+1}$ are both $(n-1)$-odometers, and $N_{n-1}\leq \Phi(W_{n-1})$ (by Lemma $4.3$), using Lemma $3.7$, $(4.7)$ implies that all the elements of $x_n^2N_n$ have cycle type $[2^{n-1},2^{n-1}]$, which correspond to $\frac{1}{4}$ of $M_n$, which finishes the proof.\\
 	
 	\item[\textbf{Proof of (iii).}] We first need the following lemma:\\
 	\begin{lemma}
 		We have the following equality:
 		$$A_3^{(n)} = (A_2^{(n-1)}\sqcup A_3^{(n-1)})\times (A_2^{(n-1)}\sqcup A_3^{(n-1)}).$$
 	\end{lemma}
 	\begin{proof}[Proof of Lemma $4.5$]
 		Applying the Markov process to the datum $([s,1][s,1],\frac{1}{4})$, we have the following data in the second level:
 		$$B_1:= ([n,1][n,1][s,1][s,1],\frac{1}{16}).$$
 		$$B_2:= ([s,1][s,1][n,1][n,1],\frac{1}{16}).$$
 		$$B_3:= ([n,1][n,1][n,1][n,1],\frac{1}{16}).$$
 		$$B_4:= ([s,1][s,1][s,1][s,1],\frac{1}{16}).$$
 		Consider the datum $B_1$. By how we define the Markov process, if we apply the Markov process to this datum $n-2$ times, the partitions of $2^{n-1}$ arising from the part $[n,1][n,1]$ will be independent from the partitions of $2^{n-1}$ arising from the part $[s,1][s,1]$. Thus, since $A_2 = ([n,1][n,1],\frac{1}{4})$ and $A_3 = ([s,1][s,1],\frac{1}{4})$, it follows that $$B_1^{(n-1)} = A_2^{(n-1)}\times A_3^{(n-1)}.$$ Doing the same thing for $B_2, B_3$, and $B_4$, we obtain the equalities
 		$$B_2^{(n-1)} = A_3^{(n-1)}\times A_2^{(n-1)}.$$
 		$$B_3^{(n-1)} = A_2^{(n-1)}\times A_2^{(n-1)}.$$
 		$$B_4^{(n-1)} = A_3^{(n-1)}\times A_3^{(n-1)}.$$
 		Noting that $A_3^{(n)} = B_1^{(n-1)}\sqcup B_2^{(n-1)} \sqcup B_3^{(n-1)} \sqcup B_4^{(n-1)}$, the result directly follows.
 	\end{proof}
 	We finish the proof of $(\text{iii})$: We will do induction. Suppose $A_3^{(k)} = \text{CD}(N_k,M_k)$ for some $k\geq 2$. Using $(4.6)$ and the sentence following it, we have
 	\begin{equation}
 	N_{k+1} = \coprod_{i,j \in \{0,1\}} (v_{k-2})^i(v_{k-2}^{x_{k+1}})^jN_{k}\times N_{k}^{x_{k+1}}.
 	\end{equation}
 	Since we have $v_{k-2} = x_k^2$, $(4.8)$ gives
 	\begin{equation}
 	N_{k+1} = \coprod_{i,j \in \{0,1\}} ((x_k^2)^iN_k)\times ((x_k^2)^jN_k)^{x_{k+1}}.
 	\end{equation}
 	\begin{claim}
 		CD$(N_k\times N_k^{x_{k+1}}, M_{k+1}) = \text{CD}(N_k,M_k)\times \text{CD}(N_k, M_k)$.
 	\end{claim}
 	\begin{proof}[Proof of Claim $4.6$]
 		Take a cycle type $c$ that exists in $N_k\times N_k^{x_{k+1}}$. Suppose we have the pairs of cycle types $(\alpha_1,\beta_1),\dots, (\alpha_i,\beta_i)$ in $N_k\times N_k^{x_{k+1}}$ with $c = \alpha_1*\beta_1 = \dots \alpha_i*\beta_i$. Assume that the proportion of the cycle type $\alpha_i$ (resp. $\beta_i$) in $N_k$ is $p_i$ (resp. $q_i$). By the definition, the density corresponding to the cycle type $c$ in CD$(N_k\times N_k^{x_{k+1}}, M_{k+1})$ is $$\frac{|N_k\times N_k^{x_{k+1}}|}{|M_{k+1}|}(\sum_{l=1}^{i}p_lq_l) = \frac{|N_k|^2}{|M_{k+1}|}(\sum_{l=1}^{i}p_lq_l) = \frac{|N_k|^2}{4|N_{k+1}|}(\sum_{l=1}^{i}p_lq_l)= \frac{1}{16}(\sum_{l=1}^{i}p_lq_l),$$ where the third and forth equalities follow from Lemma $4.1$ and Lemma $4.2$, respectively. On the other hand, if we calculate the corresponding density for the cycle type in the right-hand side, we obtain
 		$$\sum_{l=1}^{i}(p_i\frac{|N_k|}{|M_k|})(q_i\frac{|N_k|}{|M_k|}) = \frac{|N_k|^2}{|M_k|^2}(\sum_{l=1}^{i}p_lq_l) = \frac{1}{16}(\sum_{l=1}^{i}p_lq_l),$$
 		where the last equality follows from Lemma $4.1$. Hence, we obtained that the proportions of $c$ for both sides match. We can do the other direction similarly by taking a cycle type $c$ that exists in $\text{CD}(N_k,M_k)\times \text{CD}(N_k, M_k)$ as well, so we are done.
 	\end{proof}
 	Using Claim $4.6$ and the induction assumption, we have
 	\begin{equation}
 	\text{CD}(N_k\times N_k^{x_{k+1}},M_{k+1}) = A_3^{(k)}\times A_3^{(k)}.
 	\end{equation}
 	Similarly, by the induction assumption and the proof of $\text{(ii)}$, we have
 	\begin{equation}
 	\text{CD}((x_k^2N_k)\times (N_k)^{x_{k+1}},M_{k+1}) = A_2^{(k)}\times A_3^{(k)}
 	\end{equation}
 	and
 	\begin{equation}
 	\text{CD}((N_k)\times (x_k^2N_k)^{x_{k+1}},M_{k+1}) = A_3^{(k)}\times A_2^{(k)}.
 	\end{equation}
 	Similarly, using the proof of $(\text{ii})$, we also have
 	\begin{equation}
 	\text{CD}((x_k^2N_k)\times (x_k^2N_k)^{x_{k+1}},M_{k+1}) = A_2^{(k)}\times A_2^{(k)}.
 	\end{equation}
 	Combining $(4.10)-(4.13)$, and using Lemma $4.5$, the proof of (iii) directly follows.
 \end{proof}
\subsection{Model 2}
\begin{corollary}
Set $M_n(2)=\langle x_n^2, N_n \rangle$. Then $M_n(2)$ satisfies Model $2$ for all $n\geq 1$.
\end{corollary}
\begin{proof}
We have $M_n(2) = N_n \sqcup x_n^2N_n$. Then the result directly follows from the proofs of parts $\text{(ii)}$ and $\text{(iii)}$ of Proposition $4.4$.
\end{proof}
\subsection{Hausdorff Dimensions}
By the construction, both $\{M_n\}_{n\geq1}$ and $\{M_n(2)\}_{n\geq1}$ are inverse systems of groups via the natural restriction maps $M_{n+1}\twoheadrightarrow M_n$ and $M_{n+1}(2)\twoheadrightarrow M_n(2)$. Call their inverse limits $M\leq W$ and $M(2)\leq W$, respectively. Recall that the Hausdorff dimension of a subgroup $G\leq W$ is given by $$\mathcal{H}(G) = \lim_{n\rightarrow \infty} \frac{\log_2(|G_n|)}{\log_2 (|W_n|)},$$ where $G_n$ is the image of $G$ in $W_n$. We have the following corollary to the main theorems of this section:
\begin{corollary}
$\mathcal{H}(M) = \mathcal{H}(M(2)) = \frac{1}{2}$.
\end{corollary}
\begin{proof}
Using Corollary $4.7$, $M$ and $M(2)$ clearly have same Hausdorff dimensions. We will calculate $|M_n|$ for all $n$, which will give the result.
\begin{claim}
$|M_n| = 2^{2^{n-1}}$ for all $n\geq 1$.
\end{claim}
\begin{proof}[Proof of Claim $4.9$]
We will do induction. The claim is trivially true for $n=1,2$. Suppose the claim is true for $n=k\geq 2$. Using Lemma $4.1$ and Lemma $4.2$, we have $|M_{k+1}| = 4|N_{k+1}|$, and $|N_{k+1}|=4|N_{k}|^2$. Since $|M_k| = 4|N_k|$ again by Lemma $4.1$, we get
$$|M_{k+1}| = 4|N_{k+1}| = 16|N_k|^2 = |M_k|^2 = 2^{2^k},$$
which finishes the proof of the claim.
\end{proof}
The result immediately follows from Claim $4.9$.
\end{proof}
\section{Conjugates of $x^2-2$}
In this section, we will answer Question $2.13$ in the affirmative for the polynomials of the form $g_a(x) = (x+a)^2-a-2$, where $a = 2\pm b^2$ for some $b\in \mathbb{Z}$. We again assume that $g_a$ is irreducible. It follows from an elementary calculation that for this family of polynomials we have a unique Markov model.\\

Our proof method in this section is rather unusual: We first cook up a Markov model, called Model $3$, which is not directly given by the polynomials $g_a$. We construct permutation groups satisfying this model, and then we deduce a positive answer to Question $2.13$ for the family of polynomials above as a consequence of this work. We define the first level data for Model $3$ as follows:\\ 

\textbf{\underline{Model 3}}\\ 

$([nn,2],\frac{1}{4}),([ss,2],\frac{1}{4}) ,([nn,1][nn,1], \frac{1}{8}), ([ss,1][ss,1], \frac{1}{8}), ([ns,1][sn,1], \frac{1}{4})$\\

Also, it follows from an elementary calculation that if $a = 2\pm b^2$ for some $b\in \mathbb{Z}$, then the first level data for the even Markov model of $g_a$ is as follows:\\ 

\textbf{\underline{Model 4}} \\

$([ss,2],\frac{1}{2}), ([nn,1][nn,1], \frac{1}{4}), ([ss,1][ss,1], \frac{1}{4})$\\

\subsection{Model 3}
We will construct the groups $M_n$ for each level of the Markov model using a Markov map, as given in Definition $3.21$. We define the transitions for the Markov process as follows. In what follows, $k\geq 0$.\\

$[nn,2^k]\rightarrow [nn,2^{k+1}]$, $[ns,2^k]\rightarrow [ss,2^{k+1}]$, $[sn,2^k]\rightarrow [nn,2^k][ss,2^k]$,\\ 

$[ss,2^k]\rightarrow [nn,2^k][nn,2^k]$ or $[ss,2^k][ss,2^k]$.\\

We define the restricted Markov model that we will use in constructing our groups as follows:\\

$[nn,2^k]\rightarrow [nn,2^{k+1}]$, $[ns,2^k]\rightarrow [ss,2^{k+1}]$, $[sn,2^k]\rightarrow [nn,2^k][ss,2^k]$, $[ss,2^k]\rightarrow [ss,2^k][ss,2^k]$.\\

We let $M_1=\langle (1,2)\rangle$ and $M_2 = \langle (1,3,2,4),(1,2)\rangle$. We attach the type $[nn,4]$ to $(1,3,2,4)\in M_2$, the type $[ss,2][nn,1][ss,1]$ to $(1,2)\in M_2$, and the type $[nn,1]^2[ss,1]^2$ to $\text{id}\in M_2$. Applying the Markov map $m_g^{(2)}$ to these three elements and using the images as the generators of $M_3$, we get $$M_3 = \langle (1,5,3,7,2,6,4,8), (1,3)(2,4)(5,6), (1,2)(3,4)\rangle,$$ \\
where $(1,5,3,7,2,6,4,8)\in M_3$ has type $[nn,8]$, $(1,3)(2,4)(5,6)\in M_3$ has type $[ss,2]^2[nn,2][ss,1]^2$, and $(1,2)(3,4)\in M_3$ has type $[nn,2]^2[ss,1]^4$.
If we continue this way, for any $n\geq 3$, we obtain \begin{equation}
M_n = \langle x_n, m_n, v_{n-3},\dots, v_0\rangle,
\end{equation}
where $x_n\in M_n$ has type $[nn,2^n]$, $m_n = a_{n-1}x_{n-2}^{x_n}$ has type $[ss,2]^{2^{n-2}}[nn,2^{n-2}][ss,1]^{2^{n-2}}$, and $v_i$ has type $[nn,2^{i+1}]^2[ss,1]^{2^n-2^{i+2}}$ for $i=0,\dots, n-3$. We also set $m_1 = \text{id}\in W_1$, $m_2 = (1,2)\in W_2$.
\begin{customthm}{B}\label{B}
	For any $n\geq 1$, let $M_n\leq W_n$ be as above. Then $M_n$ satisfies Model $3$.
\end{customthm}
The rest of this subsection will be devoted to the proof of Theorem $B$. We will first make some preparation for the proof.\\
\begin{lemma}
	For any $n\geq 2$, we have the following identity:
	$$x_n^2 = v_{n-4}[x_n,m_n]^{-1}.$$
\end{lemma}
\begin{proof}
We will prove the lemma by showing that $x_n^2[x_n,m_n] = x_{n-2}^2$. Note that we have $x_n^2 = x_{n-1}x_{n-1}^{x_n}$, $x_{n-1} = x_{n-2}a_{n-1}$, and $m_n = a_{n-1}x_{n-2}^{x_n}$. We will use these three identities in the computation. We have
\begin{align*}
x_n^2[x_n,m_n] &= x_n^3m_nx_n^{-1}m_n^{-1}\\
&= x_n^3a_{n-1}x_nx_{n-2}x_n^{-1}x_n^{-1}a_{n-1}x_nx_{n-2}^{-1}x_n^{-1}\\
&= x_n^3a_{n-1}x_nx_{n-2}x_n^{-1}x_{n-2}^{-1}x_n^{-1}a_{n-1} \text{ } (\text{using Lemma $3.11$})\\
&= x_n^4x_{n-2}x_n^{-1}a_{n-1}x_{n-2}^{-1}x_n^{-1}a_{n-1}\text{ } (\text{using Lemma $3.11$})\\
&= x_{n-1}^2(x_{n-1}^2)^{x_n}x_{n-2}x_n^{-1}x_{n-1}^{-1}x_n^{-1}a_{n-1}\\
&= x_{n-1}^2(x_{n-1}^2)^{x_n}x_{n-2}x_n^{-2}x_nx_{n-1}^{-1}x_n^{-1}a_{n-1}\\
&= x_{n-1}^2(x_{n-1}^2)^{x_n}x_{n-2}x_{n-1}^{-1}(x_{n-1}^{-1})^{x_n}(x_{n-1}^{-1})^{x_n}a_{n-1}\\
&= x_{n-1}^2x_{n-2}x_{n-1}^{-2}x_{n-2}\text{ } (\text{using Lemma $3.11$})\\
&= x_{n-2}^2,
\end{align*}
which completes the proof.
\end{proof}
Let $H_n = \langle m_n, v_{n-3},\dots,v_0\rangle$, and $N_n = H_n^{M_n}$. Also set $V_n = \langle v_{n-3},\dots, v_0\rangle$, and $U_n = V_n^{\langle x_n\rangle}$. We also make the convention that $V_1$ and $V_2$ are the trivial subgroups of $W_1$ and $W_2$, respectively.
\begin{lemma}
For any $n\geq 2$, we have $M_n'\leq N_n^{od}$.
\end{lemma}
\begin{proof}
It suffices to show that $M_n/N_n^{od}$ is abelian. Note that $V_n$ is generated by the set $\{v_0,v_1,\dots,v_{n-3}\}$ and $v_i = x_{i+2}^2 \in \Phi(W_n)$, hence $V_n\leq N_n \cap \Phi(W_n) = N_n^{od}$. Thus, we get $M_n/N_n^{od} = \langle \bar{x}_n, \bar{m}_n\rangle$. Noting that $[x_n,m_n] = m_n^{x_n}m_n^{-1}\in N_n$ (since $m_n\in N_n$, and $N_n\trianglelefteq M_n$), and $[x_n,m_n]\in \Phi(W_n)$, we get $[x_n,m_n]\in N_n\cap \Phi(W_n) = N_n^{od}$, which shows that $M_n/N_n^{od}$ is abelian.
\end{proof}
\begin{lemma}
For any $n\geq 2$, we have $U_n\leq N_n^{od}$.
\end{lemma}
\begin{proof}
By the proof of Lemma $5.2$, we have $V_n\leq N_n^{od}$. Since $U_n = V_n^{\langle x_n\rangle}$, and $N_n^{od}\trianglelefteq M_n$, this directly implies $U_n\leq N_n^{od}$.
\end{proof}
\begin{lemma}
For any $n\geq 2$, we have $N_n^{od} = M_n^{od} = \langle M_n', U_n\rangle$.
\end{lemma}
\begin{proof}
For simplicity, we let $X_n:=\langle M_n', U_n\rangle$. Note that $X_n\leq N_n^{od}$ (by Lemma $5.2$ and Lemma $5.3$), and $N_n^{od}\leq M_n^{od}$. Hence, to prove the statement, it suffices to show $M_n^{od}\leq X_n$. Note that $X_n\trianglelefteq M_n$, because $X_n$ contains $M_n'$. We have $M_n/X_n = \langle \bar{x}_n, \bar{m}_n\rangle$. Observe that $m_n^2 = (a_{n-1}x_{n-2}^{x_n})^2 = v_{n-4}^{x_n}\in U_n\leq X_n$, and also that $x_n^2 \in X_n$ (by Lemma $5.1$). It is also clear that $x_n,m_n,x_nm_n\notin X_n\leq M_n^{od}$, which shows that $[M_n:X_n] = 4$. Since $x_n,m_n,x_nm_n\notin M_n^{od}$, and $x_n^2, m_n^2, v_{n-3},\dots,v_0 \in M_n^{od}$, we also have $[M_n : M_n^{od}] = 4$, which shows that $M_n^{od} = X_n = N_n^{od}$, as desired.
\end{proof}
\begin{proposition}
For any $n\geq 3$, we have $N_{n-1}^{od}\times (N_{n-1}^{od})^{x_n} \trianglelefteq N_n^{od}.$	
\end{proposition}
\begin{proof}
We will first prove that $N_{n-1}^{od}\times (N_{n-1}^{od})^{x_n}\leq N_n^{od}$. To prove this, since $N_n^{od}$ is normalized by $x_n$, it suffices to show $N_{n-1}^{od}\leq N_n^{od}$. Recall that, by Lemma $5.4$, we have $N_n^{od} = \langle M_n', U_n\rangle$, and $N_{n-1}^{od} = \langle M_{n-1}', U_{n-1}\rangle$. So, it suffices to show $U_{n-1}\leq N_n^{od}$ and $M_{n-1}'\leq N_n^{od}$.\\

\item Since $U_{n-1} = V_{n-1}^{x_{n-1}}$, using the last equality in $(4.2)$, we already have $U_{n-1}\leq U_n$.\\

\item To prove $M_{n-1}'\leq N_n^{od}$, note that $M_{n-1}'$ is given by

$$M_{n-1}' = \langle \{[x_{n-1},m_{n-1}],[x_{n-1},u_i],[m_{n-1},v_i],[v_i,v_j] | i,j=0,1,\dots,n-4\}\rangle^{M_{n-1}}.$$
We will prove the stronger statement that $M_{n-1}'\leq U_n$, as follows:\\

\item[(i)] Using Lemma $5.1$, we have $[x_{n-1},m_{n-1}] = x_{n-1}^{-2}v_{n-5}=v_{n-3}^{-1}v_{n-5}\in U_n$.\\
\item[(ii)] $[x_{n-1},v_i] = v_i^{x_{n-1}}v_i^{-1}\in U_{n}$, since $v_i\in V_{n-1}\leq U_{n-1}\leq U_n$, and $v_i^{x_{n-1}}\in V_{n-1}^{\langle x_{n-1}\rangle} = U_{n-1}\leq U_n$.\\
\item[(iii)] Recall that $m_{n-1} = a_{n-2}x_{n-2}^{x_n}$. Hence, we get
$$[m_{n-1},v_i] = a_{n-1}x_{n-2}^{x_n}v_ia_{n-1}(x_{n-2}^{x_n})^{-1}v_i^{-1} = a_{n-1}v_ia_{n-1}v_i^{-1}\in U_{n-1}\leq U_n,$$ where the last equality is true because $x_{n-2}^{x_n}$ commutes with $a_{n-1}$ and $v_i$ (by Lemma $3.11$), and the inclusion is true because $a_{n-1}$ normalizes $U_{n-1}$ (since $x_{n-1}$ normalizes $U_{n-1}$ and $U_{n-1}^{x_{n-1}} = U_{n-1}^{a_{n-1}})$, and $v_i\in U_{n-1}\leq U_n$.\\

We also need to show that all the conjugates of these generators under $M_{n-1}$ also lie in $U_n$. From above, $[x_{n-1},v_i]$ and $[m_{n-1},v_i]$ for $i=0,1,\dots,n-4$ already lie in $U_{n-1}$, hence their conjugates under $M_{n-1}$ lie in $U_{n-1}\leq U_n$ (since $U_{n-1}\trianglelefteq M_{n-1}$). Thus, it suffices to show that the conjugates of $[x_{n-1},m_{n-1}] = v_{n-3}^{-1}v_{n-5}$ under $M_{n-1}$ lie in $U_n$. Since the generators $v_0,\dots,v_{n-4}$ of $M_{n-1}$ already lie in $U_n$, it is enough to show that $(v_{n-3}^{-1}v_{n-5})^{x_{n-1}}\in U_n$ and $(v_{n-3}^{-1}v_{n-5})^{m_{n-1}}\in U_n$. The former inclusion is true, because $v_{n-3}^{-1}v_{n-5}\in U_n$, $(v_{n-3}^{-1}v_{n-5})^{x_{n-1}} = (v_{n-3}^{-1}v_{n-5})^{x_n^2}$, and $U_n$ is normalized by $x_n^2$ (since $U_n$ is normal in $M_n$). For the latter one, recall that $m_{n-1} = a_{n-2}x_{n-3}^{x_{n-1}}$, which clearly normalizes $U_n$, which directly implies $(v_{n-3}^{-1}v_{n-5})^{m_{n-1}}\in U_n$. Hence, we have proven that $N_{n-1}^{od}\times (N_{n-1}^{od})^{x_n}\leq N_n^{od}$.\\

To prove the normality, note that from the proof above, it is easy to see that $U_n\leq N_{n-1}^{od}\times (N_{n-1}^{od})^{x_n}$ as well, which gives $N_{n-1}^{od}\times (N_{n-1}^{od} )^{x_n}= U_n$. But, $U_n\trianglelefteq N_n^{od}$ already holds (since $U_n\trianglelefteq M_n$), which establishes $N_{n-1}^{od}\times (N_{n-1}^{od})^{x_n} \trianglelefteq N_n^{od}.$
\end{proof}
From now on, for simplicity, we set $P_{n-1}=N_{n-1}^{od}\times(N_{n-1}^{od})^{x_n}$ (although we may still sometimes use the original expression).
\begin{proposition}
For any $n\geq 3$, we have $N_n^{od}/P_{n-1}\cong \mathbb{Z}/2\mathbb{Z}$.
\end{proposition}
\begin{proof}
By the proof of Proposition $5.5$, we have $P_{n-1}=U_n$, thus we need to show that $N_n^{od}/U_n\cong \mathbb{Z}/2\mathbb{Z}$. Recall that, by Lemma $5.4$, $N_n^{od}=\langle M_n', U_n\rangle$. Thus, we have $N_n^{od}/U_n = \langle \overline{M_n'}\rangle$. Note that $[x_n,v_i], [m_n,v_i], [v_i,v_j]$ for $i,j\in \{0,\dots,n-3\}$ and their conjugates under $M_n$ all lie in $U_n$, because $U_n\trianglelefteq M_n$. Hence, we get $N_n^{od}/U_n = \overline{\langle[x_n,m_n]\rangle^{M_n}}$. By Lemma $5.1$, we have $[x_n,m_n]=x_n^{-2}v_{n-4}\in x_n^{-2}U_n$. If we can show that all the conjugates of $[x_n,m_n]$ under $M_n$ also lie in $x_n^{-2}U_n$ and that $x_n^2\notin U_n$, then we will get $N_n^{od}/P_{n-1}=N_n^{od}/U_n=\langle \overline{x_n^{-2}}\rangle\cong \mathbb{Z}/2\mathbb{Z}$, where the isomorphism is because of the fact that $x_n^4 = v_{n-3}v_{n-3}^{x_n}\in U_n$. The fact that $x_n^2\notin U_n$ follows, because the action of $U_n$ has $4$ orbits, and the action of $x_n^2$ has $2$ orbits. So, it remains to prove that all the conjugates of $[x_n,m_n]$ under $M_n$ lie in $x_n^{-2}U_n$.\\

\item[(i)] $[x_n,m_n]^{v_i}U_n=(x_n^{-2}v_{n-4})^{v_i}U_n=v_ix_n^{-2}v_{n-4}v_i^{-1}U_n=v_ix_n^{-2}U_n=v_iU_nx_n^{-2}=x_n^{-2}U_n$, where the fourth and fifth equalities hold because $x_n^{-2}$ normalizes $U_n$ and $v_i\in U_n$.\\
\item[(ii)] $[x_n,m_n]^{x_n}U_n = (x_n^{-2}v_{n-4})^{x_n}U_n = x_n^{-1}v_{n-4}x_n^{-1}U_n = x_n^{-1}U_nx_n^{-1}=x_n^{-2}U_n$.\\
\item[(iii)] 

\begin{align*}
[x_n,m_n]^{m_n}U_n &= (x_n^{-2}v_{n-4})^{m_n}U_n\\
&= m_nx_n^{-2}v_{n-4}m_n^{-1}U_n\\
&= m_nx_n^{-2}m_n^{-1}U_n \text{ } (\text{since } m_n \text{ normalizes } U_n \text{ and } v_{n-4}\in U_n)\\
&= x_n^{-2}(x_n^2m_nx_n^{-2}m_n^{-1})U_n\\
&= x_n^{-2}(x_{n-1}x_{n-1}^{x_n}a_{n-1}x_{n-2}^{x_n}x_{n-1}^{-1}(x_{n-1}^{-1})^{x_n}a_{n-1}(x_{n-2}^{-1})^{x_n})U_n\\
&= x_n^{-2}([x_{n-1},a_{n-1}][x_{n-1},x_{n-2}]^{x_n})U_n \text{ } (\text{using Lemma $3.11$})\\
&= x_n^{-2}([x_{n-1},a_{n-1}](x_{n-1}^{-2})^{x_n})U_n\\
&= x_n^{-2}[x_{n-1},a_{n-1}]U_n\\
&= x_n^{-2}v_{n-3}^{x_{n-2}}U_n \text{ } (\text{using } x_{n-1}=x_{n-2}a_{n-1})\\
&= x_n^{-2}U_n.
\end{align*}
\end{proof}
From the proofs of Lemma $5.4$ and Proposition $5.6$, we get
\begin{equation}
M_n = \coprod_{i,j \in \{0,1\}} x_n^im_n^jN_n^{od}
\end{equation} and
\begin{equation}
N_n = \coprod_{i\in \{0,1\}} x_n^{2i}P_{n-1}.
\end{equation}
Hence, if we combine $(5.2)$ and $(5.3)$, we get
\begin{equation}
M_n = \coprod_{i,j,k \in \{0,1\}} x_n^im_n^ix_n^{2k}P_{n-1}.
\end{equation}
We introduce the following notation for Model $3$:\\
\begin{align*}
A_1^{(n)}&:= \text{the set of cycle data obtained by applying the Markov process } n-1 \text{ times to }([nn,2],\frac{1}{4}). \\
A_2^{(n)}&:= \text{the set of cycle data obtained by applying the Markov process } n-1 \text{ times to }([ss,2],\frac{1}{4}). \\
A_3^{(n)}&:= \text{the set of cycle data obtained by applying the Markov process } n-1 \text{ times to }([nn,1][nn,1],\frac{1}{8}). \\
A_4^{(n)}&:= \text{the set of cycle data obtained by applying the Markov process } n-1 \text{ times to }([ss,1][ss,1],\frac{1}{8}). \\
A_5^{(n)}&:= \text{the set of cycle data obtained by applying the Markov process } n-1 \text{ times to }([ns,1][sn,1],\frac{1}{4}). \\
\end{align*}
\begin{proposition}
We have the following equalities:\\
\item[(1)] $A_1^{(n)} = \text{CD}(x_nN_n^{od},M_n).$
\item[(2)] $A_2^{(n)} = \text{CD}(x_nm_nN_n^{od},M_n).$
\item[(3)] $A_3^{(n)} = \text{CD}(x_n^2P_{n-1},M_n).$
\item[(4)] $A_4^{(n)} = \text{CD}(P_{n-1},M_n).$
\item[(5)] $A_5^{(n)} = \text{CD}(m_nN_n^{od},M_n).$
\end{proposition}
\begin{customrmk}{}\label{A}
	\textbf{Note that proving Proposition} $\boldsymbol{5.7}$ \textbf{is enough for establishing Theorem B, because it shows that the cycle data of } $\boldsymbol{M_n}$ \textbf{match the cycle data of the} $\boldsymbol{n}$\textbf{th level of the even Markov model.}
\end{customrmk}
\begin{proof}[Proof of Proposition $5.7$]
We will first give direct proofs for $(1)$ and $(3)$, and then we will prove $(2),(4)$ and $(5)$ using induction.
\item[\textbf{Proof of (1).}]Note that $A_1^{(n)} = \{([2^n],\frac{1}{4})\}$. By the definition of $N_n^{od}$, all the permutations in $x_nN_n^{od}$ are $n$-odometers, which correspond to $\frac{1}{4}$ of the elements of $M_n$, which shows that $A_1^{(n)} = \text{CD}(x_nN_n^{od},M_n).$\\

\item[\textbf{Proof of (3).}] Note that $A_3^{(n)} = \{([2^{n-1},2^{n-1}],\frac{1}{8})\}$. We also have $$x_n^2P_{n-1} = (x_{n-1}x_{n-1}^{x_n})N_{n-1}^{od}\times (N_{n-1}^{od})^{x_n} = (x_{n-1}N_{n-1}^{od})\times (x_{n-1}N_{n-1}^{od})^{x_n}.$$ Hence, by the definition of $N_{n-1}^{od}$, all the permutations in $x_n^2P_{n-1}$ have cycle type of the form $[2^{n-1},2^{n-1}]$, and they correspond to $\frac{1}{8}$ of the elements of $M_n$, which shows that $A_3^{(n)} = \text{CD}(x_n^2P_{n-1},M_n).$\\

As promised at the beginning, we will now prove the remaining parts using induction. All the statements of Proposition $5.7$ are true for $n=1,2,3$ by direct computation. Suppose that they are true for some $n=k\geq 3$. We will prove each of remaining statements for $n=k+1$.\\

\item[\textbf{Proof of (2).}] Note that $x_{k+1}m_{k+1}N_{k+1}^{od} = x_{k+1}m_{k+1}P_{k}\sqcup x_{k+1}m_{k+1}x_{k+1}^2P_{k}.$
\begin{lemma}
$x_{k+1}m_{k+1}P_{k}$ is conjugate to $x_{k+1}m_{k+1}x_{k+1}^2P_{k}$ under $W_{k+1}$.
\end{lemma}
\begin{proof}
Using $P_{k} = N_{k}^{od}\times (N_{k}^{od})^{x_{k+1}}$ and $x_{k+1}^2 = x_kx_k^{x_{k+1}}$, the statement of the lemma becomes
\begin{equation}
x_{k+1}m_{k+1}N_{k}^{od}x_{k+1}N_{k}^{od}x_{k+1}^{-1} \sim x_{k+1}m_{k+1}x_{k}N_{k}^{od}x_{k+1}x_{k}N_{k}^{od}x_{k+1}^{-1}
\end{equation}
\begin{equation}
\iff m_{k+1}N_{k}^{od}x_{k+1}N_{k}^{od} \sim m_{k+1}x_{k}N_{k}^{od}x_{k+1}x_{k}N_{k}^{od}
\end{equation}
\begin{equation}
\iff m_{k+1}N_{k}^{od}x_{k+1}N_{k}^{od} \sim m_{k+1}x_{k}N_{k}^{od}x_{k+1}N_{k}^{od}x_{k}^{-1} \text{ }(\text{since }x_k^2 \in N_k^{od})
\end{equation}
So, it suffices to show that $m_{k+1}x_{k}N_{k}^{od} = x_{k}m_{k+1}N_{k}^{od}$. Recall that $m_{k+1} = a_{k}x_{k-1}^{x_{k+1}}$. Therefore, since $x_{k-1}^{x_{k+1}}$ commutes with $x_{k}$, it is enough to prove $[a_{k},x_{k}]\in N_{k}^{od}$, which is true, because we have
\begin{align*}
[a_{k},x_{k}] &= a_{k}x_{k}a_{k}x_{k}^{-1}\\
&= a_{k}(x_{k-1}a_{k})a_{k}(x_{k-1}a_{k})^{-1}\\
&= x_{k-1}^{a_{k}}x_{k-1}^{-1}\\
&= x_{k-1}^{x_{k}}x_{k-1}x_{k-1}^{-2}\\
&= x_{k}^2x_{k-1}^{-2}\in N_k^{od}.
\end{align*} 
\end{proof}
Hence, in the light of Lemma $5.8$, we have $\text{CD}(x_{k+1}m_{k+1}N_{k+1}^{od},M_{k+1}) = 2\text{CD}(x_{k+1}m_{k+1}P_{k},M_{k+1})$. We will focus on the coset $x_{k+1}m_{k+1}P_{k}$.\\

Recall that $x_{k+1} = x_{k}a_{k+1}$ and $m_{k+1} = a_{k}(x_{k-1})^{x_{k+1}}$. We also have $x_{k}^2 = x_{k-1}x_{k-1}^{x_{k}}$. Hence, we obtain
\begin{align*}
x_{k+1}m_{k+1}P_{k}&= x_{k+1}m_{k+1}N_{k}^{od}\times (N_{k}^{od})^{x_{k+1}}\\
&= x_{k}a_{k+1}a_{k}(x_{{k-1}})^{x_{k+1}}N_{k}^{od}\times (N_{k}^{od})^{x_{k+1}}\\
&= x_{k}a_{k+1}a_{k}(x_{k}a_{k})^{x_{k+1}}N_{k}^{od}\times (N_{k}^{od})^{x_{k+1}}\\
&= x_{k}a_{k+1}a_{k}N_{k}^{od}\times (x_{k}a_{k}N_{k}^{od})^{x_{k+1}}\text{ } (\text{using Lemma $3.11$})\\
&= a_{k+1}(a_{k+1}x_{k}a_{k+1})a_{k}N_{k}^{od}\times (x_{k}a_{k}N_{k}^{od})^{x_{k+1}}\\
&=a_{k+1}(a_{k}N_{k}^{od})\times (x_{k}^2a_{k}N_{k}^{od})^{x_{k+1}}\text{ } (\text{using Lemma $3.11$})\\
&=a_{k+1}(a_{k}N_{k}^{od})\times (a_{k}N_{k}^{od})^{x_{k+1}},\\
\end{align*}
where the last equality holds because $a_k$ normalizes $N_k^{od}$ and $x_k^2\in N_k^{od}$. This is conjugate to $a_{k+1}N_{k}^{od}\times (N_{k}^{od})^{x_{k+1}}$, where the conjugating element is $a_{k}a_{k}^{a_{k+1}}$. Thus, since cycle structures do not change by conjugating, without loss of generality we can look at the coset $a_{k+1}N_{k}^{od}\times (N_{k}^{od})^{x_{k+1}}$.\\

Recall that we have $$N_{k}^{od} = N_{k-1}^{od}\times (N_{k-1}^{od})^{x_{k}} \sqcup x_{k}^2N_{k-1}^{od}\times (N_{k-1}^{od})^{x_{k}}.$$ Using this, we obtain \begin{equation}
a_{k+1}N_{k}^{od}\times (N_{k}^{od})^{x_{k+1}} = \coprod_{i,j \in \{0,1\}} a_{k+1} X_{i,j}^{(k+1)},
\end{equation} where we let \begin{equation}
X_{i,j}^{(k+1)} = (x_{k}^{2i}N_{k-1}^{od}\times (N_{k-1}^{od})^{x_{k}})\times (x_{k}^{2j}N_{k-1}^{od}\times (N_{k-1}^{od})^{x_{k}})^{x_{k+1}}.
\end{equation}
We need the following lemma:
\begin{lemma}
For any $n\geq 2$, we have the following equality.\\
$$A_2^{(n)} = dA_3^{(n-1)}\sqcup dA_4^{(n-1)}.$$
\end{lemma}
\begin{proof}
Applying the Markov process to the datum $([ss,2],\frac{1}{4})$, we obtain the following data in the second level:
$$B_1:= ([nn,2][nn,2],\frac{1}{8}).$$
$$B_2:= ([ss,2][ss,2],\frac{1}{8}).$$
We have $$A_2^{(n)} = B_1^{(n-1)}\sqcup B_2^{(n-1)}.$$ Consider the datum $B_1$. If we apply the Markov process to this datum $n-1$ times, the partitions of $2^n$ we get will be doublings of the partitions of $2^{n-1}$ that we get by applying the Markov process to the datum $[nn,1][nn,1]$ $n-1$ times, which implies $B_1^{(n-1)} = dA_3^{(n-1)}.$ If we argue in the same way for $B_2$ (using the first level datum $[ss,1][ss,1]$), the result directly follows.
\end{proof}
\begin{lemma}
All the elements of the cosets $a_{k+1}X_{1,0}^{(k+1)}$ and $a_{k+1}X_{0,1}^{(k+1)}$ have cycle type of the form $[2^{k},2^{k}]$. Thus, we have $\frac{1}{2}[dA_3^{(k)}] = CD(a_{k+1}X_{1,0}^{(k+1)}\sqcup a_{k+1}X_{0,1}^{(k+1)},M_{k+1})$.
\end{lemma}
\begin{proof}
By the symmetry, it is enough to look at $X_{1,0}^{(k+1)}$. Using $x_{k}^2 = x_{k-1}x_{k-1}^{x_{k}}$, we obtain
\begin{equation}
X_{1,0}^{(k+1)} = (x_{k-1}N_{k-1}^{od})\times(x_{k-1}N_{k-1}^{od})^{x_{k}}\times(N_{k-1}^{od})^{x_{k+1}}\times((N_{k-1}^{od})^{x_{k}})^{x_{k+1}}.
\end{equation} Hence, we get
\begin{equation}
a_{k+1}X_{1,0}^{(k+1)} = a_{k+1}[(x_{k-1}N_{k-1}^{od})\times(x_{k-1}N_{k-1}^{od})^{x_{k}}\times(N_{k-1}^{od})^{x_{k+1}}\times((N_{k-1}^{od})^{x_{k}})^{x_{k+1}}].
\end{equation}
It is easy to show that $a_{k+1}(x_{k-1}N_{k-1}^{od})(x_{k-1}N_{k-1}^{od})^{x_{k}}n_1^{x_{k+1}}(n_2^{x_{k}})^{x_{k+1}}$ is conjugate to\\ $a_{k+1}(x_{k-1}N_{k-1}^{od})(x_{k-1}N_{k-1}^{od})^{x_{k}}$ for all $n_1,n_2\in N_{k-1}^{od}$. By Lemma $3.20$, all the elements of\\ $a_{k+1}(x_{k-1}N_{k-1}^{od})(x_{k-1}N_{k-1}^{od})^{x_{k}}$ have cycle type $[2^{k},2^{k}]$. Hence all the elements of $a_{k+1}X_{1,0}^{(k+1)}\sqcup a_{k+1}X_{0,1}^{(k+1)}$ have cycle type $[2^{k},2^{k}]$, which correspond to $\frac{1}{16}$ of $M_{k+1}$. Noting that $dA_3^{(k)} = \{([2^{k},2^{k}],\frac{1}{8})\}$, the proof is complete.
\end{proof}
\begin{lemma}
For any $n\geq 3$, we have the following equality:$$\frac{1}{2}[dA_4^{n-1}] = (dA_3^{(n-2)}\sqcup dA_4^{(n-2)})\times (dA_3^{(n-2)}\sqcup dA_4^{(n-2)}).$$
\end{lemma}
\begin{proof}
Recall that $dA_4^{(n-1)} = B_2^{(n-1)}$, where $B_2 = ([ss,2][ss,2],\frac{1}{8})$. Applying the Markov process to the datum $B_2$, we get the following data in the third level:\\
$$C_1:= ([nn,2][nn,2][nn,2][nn,2],\frac{1}{32}).$$
$$C_2:= ([ss,2][ss,2][ss,2][ss,2],\frac{1}{32}).$$
$$C_3:= ([nn,2][nn,2][ss,2][ss,2],\frac{1}{32}).$$
$$C_4:= ([ss,2][ss,2][nn,2][nn,2],\frac{1}{32}).$$
We have \begin{equation}
B_2^{(n-1)} = C_1^{(n-2)}\sqcup C_2^{(n-2)} \sqcup C_3^{(n-2)}\sqcup C_4^{(n-2)}.
\end{equation}
Similar to the proof of Lemma $4.5$, we obtain 
\begin{equation}
\frac{1}{2}[C_1^{(n-2)}] = dA_3^{(n-2)}\times dA_3^{(n-2)},
\end{equation}
\begin{equation}
\frac{1}{2}[C_2^{(n-2)}] = dA_4^{(n-2)}\times dA_4^{(n-2)},
\end{equation}
\begin{equation}
\frac{1}{2}[C_3^{(n-2)}] = dA_3^{(n-2)}\times dA_4^{(n-2)},
\end{equation}
\begin{equation}
\frac{1}{2}[C_4^{(n-2)}] = dA_4^{(n-2)}\times dA_3^{(n-2)}.
\end{equation}
Combining $(5.12)-(5.16)$, and expanding the expression in Lemma $5.11$, the result directly follows.
\end{proof}
\begin{lemma}
$\frac{1}{2}[dA_4^{(k)}] = CD(a_{k+1}X_{0,0}^{(k+1)}\sqcup a_{k+1}X_{1,1}^{(k+1)},M_{k+1})$.
\end{lemma}
\begin{proof}
It is easy to check that $a_{k+1}X_{0,0}^{(k+1)}$ and $a_{k+1}X_{1,1}^{(k+1)}$ are conjugate under $W_{k+1}$, namely $a_{k+1}X_{0,0}^{(k+1)} = (a_{k+1}X_{1,1}^{(k+1)})^{x_{k+1}^2}$, hence it is enough to understand the coset $a_{k+1}X_{0,0}^{(k+1)}$. We have
\begin{equation}
a_nX_{0,0}^{(k+1)} = a_n (N_{n-2}^{od})\times (N_{n-1}^{od})^{x_{n-1}}\times (N_{n-2}^{od})^{x_n}\times ((N_{n-2}^{od})^{x_{n-1}})^{x_n}.
\end{equation}
For any $n_1,n_2 \in N_{k-1}^{od}$, it is straightforward to check that $a_{k+1} (N_{k-1}^{od})(N_{k-1}^{od})^{x_{k}}$ is conjugate to $a_{k+1} (N_{k-1}^{od})(N_{k-1}^{od})^{x_{k}}n_1n_2^{x_{k+1}}$ under $W_{k+1}$ (where, the conjugating element is $(n_1n_2^{x_{k}})^{a_{k+1}}$), hence we get 
\begin{equation}
\text{CD}(a_{k+1}X_{0,0}^{(k+1)},M_{k+1}) = |P_{k-1}|\text{CD}(a_{k+1}(N_{k-1}^{od})\times (N_{k-1}^{od})^{x_{k}},M_{k+1}).
\end{equation}
Therefore, we have
\begin{equation}
\text{CD}(a_{k+1}X_{0,0}^{(k+1)}\sqcup a_{k+1}X_{1,1}^{(k+1)},M_{k+1}) = |N_{k}^{od}|\text{CD}(a_{k+1}(N_{k-1}^{od})\times (N_{k-1}^{od})^{x_{k}},M_{k+1}),
\end{equation}
since $2|P_{k-1}| = |N_{k}^{od}|$.
By the induction assumption, we already have $\text{CD}(N_{k-1}^{od},M_{k-1}) = A_3^{(k-1)}\sqcup A_4^{(k-1)}$. Therefore, by Lemma $5.11$, it suffices to show that
\begin{equation}
|N_{k}^{od}|\text{CD}(a_{k+1}(N_{k-1}^{od})\times (N_{k-1}^{od})^{x_{k}},M_{k+1}) = d[\text{CD}(N_{k-1}^{od},M_{k-1})]\times d[\text{CD}(N_{k-1}^{od},M_{k-1})].
\end{equation}
To show this, first take a cycle type $c$ that exists in $a_{k+1}(N_{k-1}^{od})\times (N_{k-1}^{od})^{x_{k}}$. By Lemma $3.20$, we have $c = (d\alpha_1)*(d\beta_1) = \dots = (d\alpha_i)*(d\beta_i)$, where $\alpha_l,\beta_l$ are cyle types in $N_{k-1}^{od}$ and $(N_{k-1}^{od})^{x_{k}}$ with corresponding densities $p_l$ and $q_l$ respectively, for $l=1,\dots,i$. Hence, the corresponding density for $c$ in  $|N_{k}^{od}|\text{CD}(a_{k+1}(N_{k-1}^{od})\times (N_{k-1}^{od})^{x_{k}},M_{k+1})$ will be \begin{equation}
|N_{k}^{od}|(\sum_{l=1}^{i} p_lq_l\frac{|N_{k-1}^{od}|^2}{|M_{k+1}|}) = |N_{k}^{od}|\frac{|N_{k-1}^{od}|^2}{|M_{k+1}|}(\sum_{l=1}^{i} p_lq_l) = \frac{|N_{k}^{od}|^2}{2|M_{k+1}|}(\sum_{l=1}^{i} p_lq_l) = \frac{1}{16}(\sum_{l=1}^{i} p_lq_l),
\end{equation}
where the third equality follows from Proposition $5.6$, and the last equality follows from the proof of Lemma $5.4$. Again using Lemma $3.20$, the cycle type $c$ will definitely exist on the right-hand side of $(5.20)$ as well. To calculate the corresponding density in the right-hand side of $(5.20)$, note that we again have $c = (d\alpha_1)*(d\beta_1) = \dots =(d\alpha_i)*(d\beta_i)$, where $\alpha_l,\beta_l$ are cycle types in $N_{k-1}^{od}$ with corresponding densities $p_i,q_i$ respectively for $l=1,\dots,i$. Hence, the corresponding density for $c$ will be
\begin{equation}
\sum_{l=1}^{i} (p_l\frac{|N_{k-1}^{od}|}{M_{k-1}})(q_l\frac{|N_{k-1}^{od}|}{M_{k-1}}) = (\frac{|N_{k-1}^{od}|}{|M_{k-1}|})^2(\sum_{l=1}^{i} p_lq_l) = \frac{1}{16}(\sum_{l=1}^{i} p_lq_l),
\end{equation}
where the last equality follows from the proof of Lemma $5.4$. Hence, for a cycle type $c$ in the left-hand side of $(5.20)$, we have proven that the corresponding densities for $c$ in both sides of $(5.20)$ are same. We can similarly do the other direction by taking a cycle type $c$ in the right-hand side too, which finishes the proof.
\end{proof}
Combining Lemma $5.9$, Lemma $5.10$ and Lemma $5.12$, we get
$$\text{CD}(x_{k+1}m_{k+1}P_k,M_{k+1}) = \frac{1}{2}A_2^{(k+1)}.$$
Now recalling that $\text{CD}(x_{k+1}m_{k+1}N_{k+1}^{od},M_{k+1}) = 2\text{CD}(x_{k+1}m_{k+1}P_{k},M_{k+1})$, we get
$$\text{CD}(x_{k+1}m_{k+1}N_{k+1}^{od},M_{k+1}) = A_2^{(k+1)},$$
 which finishes the proof of (2).\\

\item[\textbf{Proof of (4).}] We need the following lemma:
\begin{lemma}
For any $n\geq 2$, we have the following equality:\\
$$\frac{1}{2}[A_4^{(n)}] = (A_3^{(n-1)}\sqcup A_4^{(n-1)})\times (A_3^{(n-1)}\sqcup A_4^{(n-1)}).$$
\end{lemma}
\begin{proof}
	Applying the Markov process to the datum $([ss,1][ss,1],\frac{1}{8})$, we have the following data in the second level:
	$$B_1:= ([nn,1][nn,1][ss,1][ss,1],\frac{1}{32}).$$
	$$B_2:= ([ss,1][ss,1][nn,1][nn,1],\frac{1}{32}).$$
	$$B_3:= ([nn,1][nn,1][nn,1][nn,1],\frac{1}{32}).$$
	$$B_4:= ([ss,1][ss,1][ss,1][ss,1],\frac{1}{32}).$$
The result now follows similarly to the proof of Lemma $4.5$.
\end{proof}
Recall that we have $$P_{k} = N_{k}^{od}\times (N_{k}^{od})^{x_{k+1}}.$$
By the induction assumption, we already have CD$(N_{k}^{od},M_{k}) = A_3^{(k)}\sqcup A_4^{(k)}.$ Since $(N_{k}^{od})^{x_{k+1}}$ also has the same cycle data as $N_{k}^{od}$, arguing similarly to the proof of Claim $4.6$, we obtain
$$\frac{1}{2}\text{CD}(P_k,M_{k+1}) = \text{CD}(N_k^{od},M_k)\times \text{CD}(N_k^{od},M_k),$$
which, using Lemma $5.13$, finishes the proof.\\

\item[\textbf{Proof of (5).}] Note that $m_{k+1}N_{k+1}^{od} = m_{k+1}P_k\sqcup m_{k+1}x_{k+1}^2P_k$.
\begin{lemma}
The cycle structure of the coset $m_{k+1}P_k$ is same as the cycle structure of the coset $m_{k+1}x_{k+1}^2P_k$.
\end{lemma}
\begin{proof}
Recall that $m_{k+1} = a_kx_{k-1}^{a{k+1}}$. Using this and Lemma $3.11$, we can write
$$m_{k+1}P_k = m_{k+1}N_k^{od}\times(N_k^{od})^{x_{k+1}} = (a_kN_k^{od})\times (x_{n-1}N_k^{od})^{x_{k+1}}.$$
Recall also that $x_{k+1}^2 = x_kx_k^{x_{k+1}}$. Hence, using Lemma $3.11$, we can obtain
\begin{align*}
m_{k+1}x_{k+1}^2P_k&= m_{k+1}x_{k+1}^2N_k^{od}\times (N_k^{od})^{x_{k+1}}\\
&=(a_kx_kN_k^{od})\times (x_{k-1}x_kN_k^{od})^{x_{k+1}}\\
&= (x_{k-1}N_k^{od})^{a_k}\times (a_kN_k^{od})^{x_{k+1}},
\end{align*}
where the last equality follows because $x_k = x_{k-1}a_k$, $a_k$ normalizes $N_k^{od}$, and $x_{k-1}^2\in N_k^{od}$. Combining this equality with the equality above clearly shows that the cycle structure of the coset $m_{k+1}P_k$ is same as the cycle structure of the coset $m_{k+1}x_{k+1}^2P_k$, as desired.
\end{proof}
Hence, in the light of Lemma $5.14$, we have $\text{CD}(m_{k+1}N_{k+1}^{od},M_{k+1}) = 2\text{CD}(m_{k+1}P_k, M_{k+1})$. We will focus on the coset $m_{k+1}P_k$.
Using the notation in $(5.9)$, we have 
\begin{equation}
m_{k+1}P_k = \coprod_{i,j \in \{0,1\}} m_{k+1} X_{i,j}^{(k+1)}.
\end{equation}
Recall that $X_{i,j}^{(k+1)}$ is given by
$$X_{i,j}^{(k+1)} = (x_{k}^{2i}N_{k-1}^{od}\times (N_{k-1}^{od})^{x_{k}})\times (x_{k}^{2j}N_{k-1}^{od}\times (N_{k-1}^{od})^{x_{k}})^{x_{k+1}}.$$ Using the identities $m_{k+1} = a_{k}(x_{k-1})^{x_{k+1}}$ and $x_{k}^2 = x_{k-1}x_{k-1}^{x_{k}}$, and organizing the terms (when needed), we obtain

$$m_{k+1}X_{0,0}^{(k+1)} = a_{k}[(N_{k-1}^{od})\times (N_{k-1}^{od})^{x_{k}}\times (x_{k-1}N_{k-1}^{od})^{x_{k+1}}\times ((N_{k-1}^{od})^{x_{k}})^{x_{k+1}}].$$

$$m_{k+1}X_{1,0}^{(k+1)} = a_{k}[(x_{k-1}N_{k-1}^{od})\times (x_{k-1}N_{k-1}^{od})^{x_{k}}\times (x_{k-1}N_{k-1}^{od})^{x_{k+1}}\times ((N_{k-1}^{od})^{x_{k}})^{x_{k+1}}].$$

$$m_{k+1}X_{0,1}^{(k+1)} = a_{k}[(N_{k-1}^{od})\times (N_{k-1}^{od})^{x_{k}}\times (N_{k-1}^{od})^{x_{k+1}}\times ((x_{k-1}N_{k-1}^{od})^{x_{k}})^{x_{k+1}}].$$

$$m_{k+1}X_{1,1}^{(k+1)} = a_{k}[(x_{k-1}N_{k-1}^{od})\times (x_{k-1}N_{k-1}^{od})^{x_{k}}\times (N_{k-1}^{od})^{x_{k+1}}\times (x_{k-1}(N_{k-1}^{od})^{x_{k}})^{x_{k+1}}].$$
By the symmetry, we clearly have $$\text{CD}(m_{k+1}X_{0,0}^{(k+1)},M_{k+1}) = \text{CD}(m_{k+1}X_{0,1}^{(k+1)},M_{k+1}),$$ and $$CD(m_{k+1}X_{1,0}^{(k+1)},M_{k+1}) = \text{CD}(m_{k+1}X_{1,1}^{(k+1)},M_{k+1}).$$
\begin{lemma}
We have the following equalities:
$$\text{CD}(m_{k+1}X_{0,0}^{(k+1)},M_{k+1}) = \text{CD}(m_{k+1}X_{0,1}^{(k+1)},M_{k+1}) = \text{CD}(m_{k+1}X_{1,0}^{(k+1)},M_{k+1}) = \text{CD}(m_{k+1}X_{1,1}^{(k+1)},M_{k+1}).$$
\end{lemma}
\begin{proof}
It suffices to prove (say) $\text{CD}(m_{k+1}X_{0,0}^{(k+1)},M_{k+1}) = \text{CD}(m_{k+1}X_{1,0}^{(k+1)},M_{k+1})$. Note that the actions of $m_{k+1}X_{0,0}^{(k+1)}$ and $m_{k+1}X_{1,0}^{(k+1)}$ on the right half of the tree are already identical, so it is enough to show that their actions on the left side of the tree are conjugate to each other under $W_{k+1}$,
\begin{align*}
&\iff a_{k}N_{k-1}^{od}(N_{k-1}^{od})^{x_{k}} \sim a_{k}(x_{k-1}N_{k-1}^{od})(x_{k-1}N_{k-1}^{od})^{x_{k}}\\
&\iff a_{k}N_{k-1}^{od}(N_{k-1}^{od})^{x_{k}} \sim a_{k}N_{k-1}^{od}x_{k-1}x_{k-1}^{x_{k}}(N_{k-1}^{od})^{x_{k}}\\
&\iff a_{k}N_{k-1}^{od}a_{k}N_{k-1}^{od}a_{k} \sim  a_{k}N_{k-1}^{od}x_{k-1}a_{k}x_{k-1}N_{k-1}^{od}a_{k}\\
&\iff N_{k-1}^{od}a_{k}N_{k-1}^{od} \sim N_{k-1}^{od}x_{k-1}a_{k}x_{k-1}N_{k-1}^{od}\\
&\iff N_{k-1}^{od}a_{k}N_{k-1}^{od} \sim x_{k-1}N_{k-1}^{od}a_{k}N_{k-1}^{od}x_{k-1},
\end{align*}
which is true because $x_{k-1}^2\in N_{k-1}^{od}$ by the proof of Lemma $5.4$.
\end{proof}
\begin{lemma}
For any $n\geq 3$, we have $\frac{1}{16}[A_5^{(n)}] = A_1^{(n-2)}\times (dA_3^{(n-2)}\sqcup dA_4^{(n-2)})\times (A_3^{(n-2)}\sqcup A_4^{(n-2)}).$
\end{lemma}
\begin{proof}
Applying the Markov process to the datum $A_5 = ([ns,1][sn,1],\frac{1}{4})$, we have the following data in the second level:
$$B_1 = ([ss,2][nn,1][ss,1],\frac{1}{4}).$$ If we apply the Markov process once more, we get the following data in the third level:
$$C_1 = ([nn,2][nn,2][nn,2][nn,1][nn,1], \frac{1}{16}).$$
$$C_2 = ([nn,2][nn,2][nn,2][ss,1][ss,1], \frac{1}{16}).$$
$$C_3 = ([ss,2][ss,2][nn,2][nn,1][nn,1], \frac{1}{16}).$$
$$C_4 = ([ss,2][ss,2][nn,2][ss,1][ss,1], \frac{1}{16}).$$
\end{proof}
Hence, we get $$A_5^{(n)} = B_1^{(n-1)} = C_1^{(n-2)}\sqcup C_2^{(n-2)}\sqcup C_3^{(n-2)}\sqcup C_4^{(n-2)}.$$
Similar to the proof of Lemma $4.5$, we obtain the following:
\begin{equation}
\frac{1}{16}[C_1^{(n-2)}] = dA_3^{(n-2)}\times A_1^{(n-2)}\times A_3^{(n-2)}.
\end{equation}
\begin{equation}
\frac{1}{16}[C_2^{(n-2)}] = dA_3^{(n-2)}\times A_1^{(n-2)}\times A_4^{(n-2)}.
\end{equation}
\begin{equation}
\frac{1}{16}[C_3^{(n-2)}] = dA_4^{(n-2)}\times A_1^{(n-2)}\times A_3^{(n-2)}.
\end{equation}
\begin{equation}
\frac{1}{16}[C_4^{(n-2)}] = dA_4^{(n-2)}\times A_1^{(n-2)}\times A_4^{(n-2)}.
\end{equation}
Combining and organizing these, Lemma $5.16$ directly follows.

\begin{lemma}
$\frac{1}{2}\text{CD}(m_{k+1}X_{0,0}^{(k+1)},M_{k+1}) = A_1^{(k-1)}\times (dA_3^{(k-1)}\sqcup dA_4^{(k-1)})\times (A_3^{(k-1)}\sqcup A_4^{(k-1)}).$
\end{lemma}
\textbf{First note that using Lemma} $\boldsymbol{5.15}$ \textbf{and Lemma} $\boldsymbol{5.16}$\textbf{, and also recalling that}\\ $\boldsymbol{\textbf{CD}(m_{k+1}N_{k+1}^{od},M_{k+1}) = 2\text{CD}(m_{k+1}P_k, M_{k+1})}$\textbf{, proving Lemma} $\boldsymbol{5.17}$ \textbf{will finish the proof of }$\boldsymbol{(5)}$\textbf{. We now prove Lemma} $\boldsymbol{5.17:}$\\
\begin{proof}[Proof of Lemma $5.17$]
Recall that we have
$$m_{k+1}X_{0,0}^{(k+1)} = a_{k}[(N_{k-1}^{od})\times (N_{k-1}^{od})^{x_{k}}\times (x_{k-1}N_{k-1}^{od})^{x_{k+1}}\times ((N_{k-1}^{od})^{x_{k}})^{x_{k+1}}].$$
Arguing similarly to the proof of Lemma $5.12$, we get
\begin{equation}
\frac{1}{4}\text{CD}(m_{k+1}X_{0,0}^{(k+1)},M_{k+1}) = \text{CD}(a_{k}(N_{k-1}^{od})\times (N_{k-1}^{od})^{x_{k}},M_{k})\times \text{CD}(x_{k-1}N_{k-1}^{od},M_{k-1})\times \text{CD}(N_{k-1}^{od},M_{k-1}).
\end{equation}
Note that $a_{k}(N_{k-1}^{od})\sim a_{k}N_{k-1}^{od}n^{x_{k}}$ for all $n\in N_{k-1}^{od}$, and by the induction assumption $\text{CD}(N_{k-1}^{od},M_{k-1}) = A_3^{(k-1)}\sqcup A_4^{(k-1)}.$ Using Lemma $3.20$, and arguing similarly to the proof of Lemma $5.12$, it follows that 
\begin{equation}
\text{CD}(a_{k}(N_{k-1}^{od})\times (N_{k-1}^{od})^{x_{k}},M_k) = \frac{1}{2}[dA_3^{(k-1)}\sqcup dA_4^{(k-1)}].
\end{equation}
Again by the induction, we also have
\begin{equation}
\text{CD}(N_{k-1}^{od},M_{k-1}) = A_3^{(k-1)}\sqcup A_4^{(k-1)}.
\end{equation}
We also clearly have
\begin{equation}
\text{CD}(x_{k-1}N_{k-1}^{od},M_{k-1}) = A_1^{(k-1)}.
\end{equation}
Using $(5.28)$, and taking the product of $(5.29), (5.30)$ and $(5.31)$, the desired result follows.\\
\end{proof}
Hence, the proof for Model $3$ is completed.
\end{proof}
\subsection{Model 4}
We first recall the first level data for Model $4$:\\

\textbf{\underline{Model 4}} \\

$([ss,2],\frac{1}{2}), ([nn,1][nn,1], \frac{1}{4}), ([ss,1][ss,1], \frac{1}{4})$.\\

Using Lemma $2.6$, it follows that we have the following transitions in this case. In what follows, $k\geq 0$.\\

$[nn,2^k]\rightarrow [nn,2^{k+1}]$, $[ns,2^k]\rightarrow [ss,2^{k+1}]$, $[sn,2^k]\rightarrow [ns,2^k][sn,2^k]$,\\ 

$[ss,2^k]\rightarrow [nn,2^k][nn,2^k]$ or $[ss,2^k][ss,2^k]$.\\

Observe that the only difference between the transitions for Model $3$ and the transitions for Model $4$ is that $[sn,2^k]$ leads to $[nn,2^k][ss,2^k]$ in Model $3$, whereas $[sn,2^k]$ leads to $[ns,2^k][sn,2^k]$ in Model $3$. However, applying the Markov process for Model $4$ iteratively, it is clear that the type $sn$ never appears, which shows that two set of transitions will give the same level $n$ data for all $n$, when applied to Model $4$. Therefore, we have the following corollary, similar to the previous section:
\begin{corollary}
Set $M_n(2) = \langle x_nm_n, N_n^{od}\rangle$. Then $M_n(2)$ satisfies Model $4$ for all $n\geq 1$.
\end{corollary}
\begin{proof}
Since $(x_nm_n)^2\in M_n^{od} = N_n^{od}$ (by Lemma $5.4$), we have $M_n(2) = N_n^{od}\sqcup x_nm_nN_n^{od}$. Then the result directly follows from the proofs of parts $(2)$, $(3)$ and $(4)$ of Proposition $5.7$.
\end{proof}
\subsection{Hausdorff Dimensions}
We again denote by $M$ and $M(2)$ the inverse limits of $\{M_n\}_{n\geq 1}$ and $\{M_n(2)\}_{n\geq 1}$, respectively. Using Corollary $5.18$, $M$ and $M(2)$ clearly have same Hausdorff dimensions.
\begin{corollary}
$\mathcal{H}(M) = \mathcal{H}(M(2)) = \frac{1}{2}$
\end{corollary}
\begin{proof}
We will prove the corollary by calculating $|M_n|$ for all $n$. We have $|M_1| = 2$.
\begin{claim}
For $n\geq 2$, we have $|M_n| = 2^{2^{n-1}+1}$.
\end{claim}
\begin{proof}[Proof of Claim $5.20$]
The claim is trivially true for $n=2$. Suppose it is true for $n=k\geq 2$. By the proof of Lemma $5.4$ and Proposition $5.6$, we have $|M_{k+1}| = 4|N_{k+1}^{od}|$ and $|N_{k+1}^{od}| = 2|N_k^{od}|^2$. Since we also have $|M_k| = 4|N_k^{od}|$ again by the proof of Lemma $5.4$, we get
$$|M_{k+1}| = 4|N_{k+1}^{od}| = 8|N_k^{od}|^2 = \frac{1}{2}|M_k|^2 = 2^{2^k+1},$$
which completes the proof.
\end{proof}
The result is immediate from Claim $5.20$.
\end{proof}
\section{Conjugates of $x^2-1$}
In this section, we will answer Question $2.13$ in the affirmative for the polynomials of the form $h_a(x) = (x+a)^2-a-1$. We again assume that $h_a$ is irreducible. It follows from an elementary calculation that for this family we again have two different models, depending on whether $a=\pm b^2$ for some $b\in \mathbb{Z}$ or not. Below, for each different model, we only give the first level data that is obtained by factoring $h_a$ modulo primes of the form $4k+1$. They are as follows:\\

\textbf{\underline{Model 5 (for $\boldsymbol{h_a(x)}$ such that $\boldsymbol{a\neq \pm b^2}$ for any $\boldsymbol{b\in \mathbb{Z}}$)}}\\ 

$([nn,2],\frac{1}{4}),([sn,2],\frac{1}{4}) ,([nn,1][sn,1], \frac{1}{8}), ([ns,1][ss,1], \frac{1}{8})$,\\ 

$([nn,1][nn,1], \frac{1}{16}), ([ns,1][ns,1], \frac{1}{16}), ([sn,1][sn,1], \frac{1}{16}), ([ss,1][ss,1], \frac{1}{16}).$\\

\textbf{\underline{Model 6 (for $\boldsymbol{h_a(x)}$ such that $\boldsymbol{a = \pm b^2}$ for some $\boldsymbol{b\in \mathbb{Z}}$)}} \\

$([sn,2],\frac{1}{2}), ([nn,1][nn,1], \frac{1}{8}), ([ns,1][ns,1], \frac{1}{8}), ([sn,1][sn,1], \frac{1}{8}), ([ss,1][ss,1], \frac{1}{8}).$\\

The strategy will be similar to the previous section: We will first give a proof for Model $5$, and then use this proof to give a proof for Model $6$.
\subsection{Model 5}
We will construct the groups $M_n$ for each level of the Markov model using a Markov map, as given in Definition $3.21$. By direct computation, the transitions for the Markov process are as follows. In what follows, $k\geq 0$.\\

$[nn,2^k]\rightarrow [nn,2^{k+1}]$, $[ns,2^k]\rightarrow [sn,2^{k+1}]$, $[sn,2^k]\rightarrow [nn,2^k][sn,2^k]$,\\ 

$[ss,2^k]\rightarrow [nn,2^k][nn,2^k]$ or $[ns,2^k][ns,2^k]$ or $[sn,2^k][sn,2^k]$ or $[ss,2^k][ss,2^k]$.\\

We define the restricted Markov model that we will use in constructing our groups as follows:\\

$[nn,2^k]\rightarrow [nn,2^{k+1}]$, $[ns,2^k]\rightarrow [sn,2^{k+1}]$, $[sn,2^k]\rightarrow [nn,2^k][sn,2^k]$, $[ss,2^k]\rightarrow [ss,2^k][ss,2^k]$.\\

We let $M_1=\langle (1,2)\rangle$ and $M_2 = \langle (1,3,2,4),(1,2)\rangle$. We attach the type $[nn,4]$ to $(1,3,2,4)\in M_2$, the type $[sn,2][ss,1]^2$ to $(1,2)\in M_2$, and the type $[nn,1]^2[ss,1]^2$ to $\text{id}\in M_2$. Applying the Markov map $m_h^{(2)}$ to these three elements and using the images as the generators of $M_3$, we get $$M_3 = \langle (1,5,3,7,2,6,4,8), (1,3)(2,4)(5,6), (1,2)(3,4)\rangle,$$ \\
where $(1,5,3,7,2,6,4,8)\in M_3$ has type $[nn,8]$, $(1,3)(2,4)\in M_3$ has type $[nn,2][sn,2][ss,1]^4$, and $(1,2)(3,4)\in M_3$ has type $[nn,2]^2[ss,1]^4$.
If we continue this way, for any $n\geq 3$, we obtain \begin{equation}
M_n = \langle x_n, m_n, v_{n-3},\dots, v_0\rangle,
\end{equation}
where $x_n\in M_n$ has type $[nn,2^n]$, $m_n$ has type $[nn,2^{n-2}]\dots [nn,2][sn,2][ss,1]^{2^{n-1}}$, and $v_i$ has type\\ $[nn,2^{i+1}]^2[ss,1]^{2^n-2^{i+2}}$ for $i=0,\dots, n-3$. We also set $m_1 = \text{id}\in W_1$, $m_2 = (1,2)\in W_2$.\\

\begin{customthm}{C}\label{C}
	For any $n\geq 1$, let $M_n\leq W_n$ be as above. Then $M_n$ satisfies Model $5$.
\end{customthm}
The rest of this subsection will be devoted to the proof of Theorem $C$. We will first make some preparation for the proof.

\begin{lemma}
$m_{n+1} = x_n^2m_nx_n^{-1}$ for all $n\geq 1.$
\end{lemma}
\begin{proof}
We will prove the lemma by showing that $x_{n}^{-1}m_{n+1} = m_n^{x_n}$. The statement is trivially true for $n=1,2$. By applying the Markov map successively, and using the same idea as in the proof of Lemma $3.20$, we can write $m_3 = a_2$, $m_4 = (1,2)a_3$, $m_5 = (1,3,2,4)(5,6)a_4$, $m_6 = (1,5,3,7,2,6,4,8)(9,11,10,12)(13,14)a5,\dots$. It follows that for all $n\geq 3$, there exists $\alpha_n\in W_n$ such that $x_{n-3}\alpha_na_{n-1} = m_n$ and $\text{Supp}(x_{i-1})\cap \text{Supp}(\alpha_i) = \emptyset$. Moreover, it is clear from the successive Markov maps that we have $(x_{n-3}\alpha_n)^{a_{n-1}} = \alpha_{n+1}$ for all $n\geq 3$.\\

Note that since $m_n$ acts trivially on the right half tree, we have $m_n^{x_n} = m_n^{a_n}$. Hence, we have
\begin{align*}
x_n^{-1}m_{n+1} &= m_n^{x_n}&\iff \\
x_n^{-1}m_{n+1} &= a_nm_na_n&\iff \\
a_nx_n^{-1}m_{n+1}a_n &= m_n&\iff \\
x_{n-1}^{-1}m_{n+1}a_n &= m_n&\iff \\
m_{n+1}a_n &= x_{n-1}m_n &\iff \\
x_{n-2}\alpha_{n+1} &= x_{n-2}a_{n-1}m_n\\
\alpha_{n+1} &= a_{n-1}m_n &\iff\\
\alpha_{n+1} &= (x_{n-3}\alpha_n)^{a_{n-1}},
\end{align*}
which is true by the above discussion.
\end{proof}
\begin{lemma}
For any $n\geq 2$, we have the following identity:\\
$$m_n^{-1}v_{n-3}\dots v_2v_1 = (m_n)^{x_n^{2^{n-2}+2}}.$$
\end{lemma}
\begin{proof}
We will do induction. It is trivially true for $n=3$. Suppose it is true for $n=N$. Note that
\begin{equation}
(m_{N+1}^{-1}v_{n-2}\dots v_2v_1)(m_N^{-1}v_{N-3}\dots v_2v_1)^{-1} = m_{N+1}^{-1}v_{N-2}m_N.
\end{equation}
Using Lemma $6.1$ we have $m_{N+1} = x_N^2m_Nx_N^{-1}$, and also recalling that $v_{N-2}=x_n^2$, after cancellation $(6.2)$ becomes equal to $x_N$. Hence, to prove the lemma, by the induction, it suffices to prove that
$$x_{N+1}^{2^{N-1}+2}m_{N+1}(x_{N+1}^{2^{N-1}+2})^{-1}x_N^{2^{N-2}+2}m_N^{-1}(x_N^{2^{N-2}+2})^{-1} = x_N.$$
Noting that $x_{N+1}^{2^{N-1}+2} = x_N^{2^{n-2}+1}(x_N^{x_{N+1}})^{2^{n-2}+1}$, and that $(x_N^{x_{N+1}})^{2^{n-2}+1}$ commutes with $m_{N+1}$ (by Lemma $3.11$), and also using Lemma $6.1$, the equality in the last line directly follows.
\end{proof}
Similar to the last section, let $H_n = \langle m_n,v_{n-3},\dots,v_0\rangle$, and $N_n = H_n^{M_n}$. Also set $V_n = \langle v_{n-3},\dots, v_0\rangle$, and $U_n = V_n^{\langle x_n\rangle}$.
\begin{lemma}
For any $n\geq 1$, we have $m_n^2\in \langle M_n', U_n\rangle$.
\end{lemma}
\begin{proof}
We will prove the lemma by proving that there exists $\alpha_n\in M_n$ such that \begin{equation}
v_{n-3}\dots v_2v_1[\alpha_n,m_n] = m_n^2.
\end{equation}
After some straightforward manupilation, $(6.3)$ becomes equivalent to
\begin{equation}
m_n^{-1}v_{n-3}\dots v_2v_1 \sim m_n,
\end{equation}which is true by Lemma $6.2$.
\end{proof}
\begin{proposition}
For any $n\geq 2$, we have $N_{n-1}^{od}\leq N_n^{od}$.
\end{proposition}
Proof of Proposition $6.4$ will directly follow from the next two lemmas.
\begin{lemma}
For any $n\geq 1$, we have $N_n^{od} = \langle M_n', U_n\rangle$.
\end{lemma}
\begin{proof}
For simplicity, we set $X_n = \langle M_n', U_n\rangle$. The proof of the fact that $X_n\leq N_n^{od}$ verbatim follows from the proofs of Lemma $5.2$ and Lemma $5.3$.
\begin{claim}
$M_n^{od} = \langle x_n^2, X_n\rangle \implies$ Lemma $6.5$ holds.
\end{claim}
\begin{proof}[Proof of Claim $6.6$]
Lemma $6.5$ is trivially true for $n=1,2$. Suppose $n\geq 3$.  Note that $x_n^4 = v_{n-3}^2(v_{n-3}^2)^{x_n}\in N_n^{od}$. We will now prove that $x_n^2\notin N_n^{od}$: It suffices to show that $x_n^2\notin N_n$. By the definitions of $N_n$ and $U_n$, we can write \begin{equation}
N_n = \langle m_n, m_n^{x_n}, m_n^{x_n^2}, m_n^{x_n^3}, U_n\rangle,
\end{equation} since $x_n^4\in U_n$. Using the identity $x_n^2 = x_{n-1}x_{n-1}^{x_n}$ and Lemma $3.11$, this becomes
\begin{equation}
N_n = \langle m_n, m_n^{x_n}, m_n^{x_{n-1}}, m_n^{x_{n-1}x_n}, U_n\rangle,
\end{equation}
which, using Lemma $3.11$ and the fact that $U_n = \langle U_{n-1},x_{n-1}^2, U_{n-1}^{x_n},(x_{n-1}^2)^{x_n}\rangle$, can be written as
\begin{equation}
N_n = \langle m_n, m_n^{x_{n-1}}, x_{n-1}^2,U_{n-1}\rangle\times \langle m_n, m_n^{x_{n-1}}, x_{n-1}^2,U_{n-1}\rangle^{x_n}.
\end{equation} Hence, since $x_n^2 = x_{n-1}x_{n-1}^{x_n}$, we get \begin{equation}
x_n^2 \in N_n \iff x_{n-1}\in \langle m_n, m_n^{x_{n-1}},x_{n-1}^2, U_{n-1}\rangle.
\end{equation} Therefore, to show $x_n^2\notin N_n$, it suffices to show that $x_{n-1}\notin N_n$, hence it is enough to show that $x_{n-1}\notin M_{n}$. This is true for $n=3$ (by direct computation). If $n>3$, and $x_{n-1}$ was in $M_n$, we would have $\pi_3\circ\dots\circ \pi_n (x_{n-1}) = x_3\in M_3$, which gives a contradiction. Hence, $x_n^2\notin N_n^{od}$, as claimed.\\

 Using that $x_n^4\in N_n^{od}$ and $x_n^2\notin N_n^{od}$, it easily follows that $N_n^{od}\trianglelefteq M_n^{od}$ with index 2. But, $X_n$ also has index 2 in $\langle x_n^2, X_n\rangle$, since $X_n\trianglelefteq \langle x_n^2, X_n\rangle$ (because $M_n'\leq X_n$), $x_n^4 = v_{n-3}^2(v_{n-3}^2)^{x_n}\in U_n\leq X_n$, and $x_n^2\notin X_n$. Recalling again that $X_n\leq N_n^{od}$, this directly implies Claim $6.6$.
\end{proof}
Hence, to prove Lemma $6.5$, it suffices to prove the equality $M_n^{od} = \langle x_n^2, X_n\rangle$. The containment $\langle x_n^2, X_n\rangle \leq M_n^{od}$ is obvious. To prove the other direction, note that $M_n^{od}$ has index $4$ in $M_n$, which easily follows from the fact that $N_n^{od}$ has index $2$ in $N_n$. Therefore, it suffices to show that $\langle x_n^2,X_n\rangle$ has also index $4$ in $M_n$. This is clearly true, because $m_n^2\in X_n$ by Lemma $6.3$, so we are done.
\end{proof}
\begin{lemma}
Let $n\geq 1$, and $T_n = \langle x_n^2, m_n, V_n\rangle$. Then we have $M_{n-1}'\leq T_n'$.
\end{lemma}
\begin{proof}
Note that both $V_n$ and $m_n$ act trivially on the right half of the tree. Hence, since $x_n^2 = x_{n-1}x_{n-1}^{x_n}$, and $x_{n-1}^{x_n}$ commutes with all the elements of $V_n$ and $m_n$ using Lemma $3.11$, we obtain $T_n' = \langle x_{n-1}, m_n, V_n\rangle'$. Therefore, for proving the lemma, without loss of generality we can take $T_n = \langle x_{n-1}, m_n, V_n\rangle$. We have $m_{n-1}\in T_n$ using Lemma $6.1$, hence $M_{n-1}\leq T_n$, which directly implies the result.
\end{proof}
\begin{proof}[Proof of Proposition $6.4$]
By the proof of Proposition $5.5$, we have $U_{n-1}\leq U_n$. Also, since $T_n\leq M_n$, by Lemma $6.7$ we have $M_{n-1}'\leq M_n'$. Combining these two with Lemma $6.5$ directly gives the result.
\end{proof}
Similar to the previous section, we set $P_{n-1} = N_{n-1}\times N_{n-1}^{x_n}$. By Proposition $6.4$ we have $P_{n-1}\leq N_{n}^{od}$, since $N_n^{od}$ is normalized by $x_n$ (since $N_n^{od}\trianglelefteq M_n$ by Lemma $3.10$). Next proposition will prove that even a stronger statement is true.
\begin{proposition}
$P_{n-1}\trianglelefteq N_n^{od}$.
\end{proposition}
\begin{proof}
By the characterization of $N_n^{od}$ in Lemma $6.5$, we have $N_n^{od} = \langle M_n', U_n\rangle$ and $P_{n-1} = \langle M_{n-1}', U_{n-1}, (M_{n-1}')^{x_n}, U_{n-1}^{x_n}\rangle$. The proof will be a direct consequence of the following lemma.
\begin{lemma}
We have the following:\\

\item[(1)] $M_{n-1}'$ is normalized by $U_n$.\\
\item[(2)] $\langle M_{n-1}', U_{n-1}\rangle$ is normalized by $M_n'$.\\
\item[(3)] $U_{n-1}$ is normalized by $U_n$.\\
\item[(4)] $(M_{n-1}')^{x_n}$ is normalized by $U_n$.\\
\item[(5)] $\langle (M_{n-1}')^{x_n}, U_{n-1}^{x_n}\rangle$ is normalized by $M_n'$.\\
\item[(6)] $U_{n-1}^{x_n}$ is normalized by $U_n$.
\end{lemma}
\begin{proof}[Proof of Lemma $6.9$]
\item[(1)] By the definitions of $U_{n-1}$ and $U_n$, we can write
\begin{equation}
U_n = \langle U_{n-1}, U_{n-1}^{x_n}, x_{n-1}^2, (x_{n-1}^2)^{x_n}\rangle. 
\end{equation}
$M_{n-1}'$ is already normalized by $U_{n-1}$ and $x_{n-1}^2$, since $M_{n-1}'\trianglelefteq M_{n-1}$ and $x_{n-1}^2\in M_{n-1}, U_{n-1}\leq M_{n-1}$. Also by Lemma $3.11$ every element of $U_{n-1}^{x_n}$ and $(x_{n-1}^2)^{x_n}$ commute with every element of $M_{n-1}'$, which in particular shows that $M_{n-1}'$ is normalized by $U_{n-1}^{x_n}$ and $(x_{n-1}^2)^{x_n}$ as well, which finishes the proof.\\

\item[(2)] First note that $\langle M_{n-1}', U_{n-1}\rangle$ is normalized by $M_{n-1}$, since it contains $M_{n-1}'$. We want to show that it is normalized by $M_n'$, where $$M_n = \langle x_n, m_n, v_0, v_1, \dots, v_{n-3}\rangle.$$
Note that $v_0, v_1, \dots, v_{n-4}, v_{n-3} = x_{n-1}^2\in M_{n-1}$, and also that by Lemma $6.1$, $m_n = x_{n-1}^2m_{n-1}x_{n-1}^{-1}\in M_{n-1}$ as well. Therefore,  $v_0, v_1, \dots, v_{n-4}, v_{n-3}, m_n$ already normalize $\langle M_{n-1}', U_{n-1}\rangle$, since $\langle M_{n-1}',U_{n-1}\rangle \trianglelefteq M_{n-1}$. Hence, only possible problematic case might happen when a commutator involves $x_n$. Therefore, to prove the result, it suffices to prove that $[x_n,\alpha]$ normalizes $\langle M_{n-1}', U_{n-1}\rangle$ for any $\alpha \in \{v_0, v_1, \dots, v_{n-4}, v_{n-3}, m_n\}$. Set $X_{n-1} = \langle M_{n-1}', U_{n-1}\rangle$ for simplicity. We have
$$X_{n-1}^{[x_n,\alpha]} = \alpha^{x_n}\alpha^{-1}X_{n-1}\alpha(\alpha^{x_n})^{-1} = \alpha^{x_n} X_{n-1}(\alpha^{x_n})^{-1} = X_{n-1},$$
where the second equality holds because $\alpha \in M_{n-1}$ and $X_{n-1}\trianglelefteq M_{n-1}$, and the third equality is true because $\alpha^{x_n}$ commutes with each element of $X_{n-1}$ by Lemma $3.11$. Hence, we are done.\\

\item[(3)] Using $(6.5)$ and the facts that $U_{n-1}$ is normalized by itself and $x_{n-1}^2$ (by its definition) and it is also normalized by $U_{n-1}^{x_n}$ and $(x_{n-1}^2)^{x_n}$ since each element of $(U_{n-1})^{x_n}$ and $(x_{n-1}^2)^{x_n}$ commute with each element of $U_{n-1}$ (by Lemma $3.11$), the result follows.\\

\item[(4)] Follows similarly to the proof of $(1)$.\\

\item[(5)] Follows similarly to the proof of $(2)$.\\

\item[(6)] Follows similarly to the proof of $(3)$.
\end{proof}
The proof of Proposition $6.8$ is complete using Lemma $6.9$.
\end{proof}
\begin{proposition}
$N_n^{od}/P_{n-1} = \langle \overline{[x_n,m_n]},\overline{v_{n-3}}\rangle$.
\end{proposition}
\begin{proof}
By Lemma $6.5$ we have $N_n^{od} = \langle M_n', U_n\rangle$. Hence, we can write
\begin{equation}
N_n^{od}/P_{n-1} = \langle \overline{M_n'}, \overline{U_n}\rangle.
\end{equation}
By $(6.5)$ and the definition of $P_{n-1}$, $(6.6)$ becomes
\begin{equation}
N_n^{od}/P_{n-1} = \langle \overline{M_n'}, \overline{v_{n-3}}, \overline{v_{n-3}^{x_n}}\rangle.
\end{equation}
\begin{claim}
$\overline{[x_n,m_n]^2} = \overline{v_{n-3}(v_{n-3})^{x_n}}$.
\end{claim}
\begin{proof}[Proof of Claim $6.11$]
We have
\begin{align*}
\overline{[x_n,m_n]^2} = \overline{v_{n-3}(v_{n-3})^{x_n}}&\iff (\text{by Lemma } 3.11)\\
(m_n^2v_{n-3}^{-1})^{x_n}(v_{n-3}m_n^2)^{-1}\in P_{n-1}&\iff \\
v_{n-3}m_n^2\in N_{n-1}^{od}&\iff (\text{by Lemma } 6.1)\\
x_{n-1}^2(x_{n-1}^2m_{n-1}x_{n-1}^{-1})^2\in N_{n-1}^{od}&\iff \\
x_{n-1}^4m_{n-1}^2[m_{n-1}^{-1},x_{n-1}]\in N_{n-1}^{od},
\end{align*}
which is true since $x_{n-1}^4, m_{n-1}^2, [m_{n-1}^{-1},x_{n-1}]\in N_{n-1}^{od}$.
\end{proof}
Using Claim $6.11$, we can write \begin{equation}
N_n^{od}/P_{n-1} = \langle \overline{M_n'}, \overline{v_{n-3}}\rangle.
\end{equation}
Recall that $M_n = \langle x_n, m_n, v_0, \dots v_{n-3}\rangle$. Hence, we have
\begin{equation}
M_n' = \langle \{[x_n,m_n], [x_n,v_i], [m_n,v_i], [v_i,v_j] | i,j = 0,1\dots, n-3\}\rangle^{M_n}.
\end{equation}
Since $m_n = x_{n-1}^2m_{n-1}x_{n-1}^{-1}\in M_{n-1}$ and $v_0,\dots,v_{n-3} = x_{n-1}^2\in M_{n-1}$, and also because $M_{n-1}'\leq M_n'$ (by Lemma $6.7$), $(6.13)$ becomes
\begin{equation}
M_n' = \langle \{[x_n,m_n], [x_n,v_i], M_{n-1}' | i = 0,1\dots, n-3\}\rangle^{M_n}.
\end{equation}
Conjugates of $M_{n-1}'$ under $M_n$ all lie in $P_{n-1}$, because they lie in $\langle M_{n-1}', (M_{n-1}')^{x_n}\rangle\leq P_{n-1}$. Also note that $[x_n,v_i]\in P_{n-1}$ for $i=0,1,\dots,n-4$, because
$$[x_n,v_i] = v_i^{x_n}v_i^{-1}\in \langle N_{n-1}^{od},(N_{n-1}^{od})^{x_n}\rangle = P_{n-1}.$$
All their conjugates under $M_n$ also lie in $P_{n-1}$, because $P_{n-1}$ is normalized by $x_n$, and the other generators of $M_n$ lie in $M_{n-1}$, all of which commute with $v_i^{x_n}$ (by Lemma $3.11$) and conjugate $v_i^{-1}$ into $N_{n-1}^{od}$, since $N_{n-1}^{od}\trianglelefteq M_{n-1}$.\\

We will now prove three simple claims, which together will easily imply Proposition $6.10$.

\begin{claim}
$\overline{[x_n,m_n]^{x_n}} = \overline{[x_n,m_n]^{-1}}$.
\end{claim}
\begin{proof}[Proof of Claim $6.12$]
We need to show that $[x_n,[x_n,m_n]]\in P_{n-1}$.
$$\iff x_n^2m_nx_n^{-2}m_n^{-1}\in P_{n-1}.$$
$$\iff [x_n^2,m_n]\in P_{n-1},$$
which is true because $[x_n^2,m_n] = [x_{n-1},m_n]\in M_{n-1}'\leq P_{n-1},$
where the equality holds because $x_n^2 = x_{n-1}x_{n-1}^{x_n}$ and $x_{n-1}^{x_n}$ commutes with $m_n$.
\end{proof}
\begin{claim}
Let $\alpha \in \{m_n,v_0,\dots, v_{n-3}\}.$ Then $\overline{[x_n,m_n]^{\alpha}} = \overline{[x_n,m_n]}.$
\end{claim}
\begin{proof}[Proof of Claim $6.13$]
We need to show that $[\alpha, [x_n,m_n]]\in P_{n-1}.$
\begin{equation}
\iff \alpha m_n^{x_n}(\alpha^{-1})^{m_n^{-1}}(m_n^{-1})^{x_n}\in P_{n-1}.
\end{equation}
Since $m_n^{x_n}$ commutes with $(\alpha^{-1})^{m_n^{-1}}$ (by Lemma $3.11$), $(6.15)$ becomes
\begin{equation}
\iff [\alpha, m_n^{-1}]\in P_{n-1},
\end{equation}
which is true because $[\alpha, m_n^{-1}]\in M_{n-1}'\leq P_{n-1}.$
\end{proof}
\begin{claim}
$\overline{[x_n,v_{n-3}]} = \overline{[x_n,m_n]^2}.$
\end{claim}
\begin{proof}[Proof of Claim $6.14$]
We need to show that $[x_n,v_{n-3}][x_n,m_n]^{-2}\in P_{n-1}.$ After expanding and organizing by using Lemma $3.11$, this becomes
$$(v_{n-3}m_n^{-2})^{x_n}(v_{n-3}^{-1}m_n^2)\in P_{n-1}.$$
$$\iff v_{n-3}^{-1}m_n^2\in N_{n-1}^{od}$$
$$\iff x_{n-1}^{-2}m_n^2 \in N_{n-1}^{od},$$
which is true because $x_{n-1}^2, m_n^2\in M_{n-1}^{od}\backslash N_{n-1}^{od}$ and $[M_{n-1}^{od}:N_{n-1}^{od}] = 2$ (To argue why $x_{n-1}^2, m_n^2 \notin N_{n-1}^{od}$: $x_{n-1}^2\notin N_{n-1}^{od}$ is true by the proof of Lemma $6.5$. Assume for the sake of contradiction that $m_n^2\in N_{n-1}^{od}$. Using Lemma $6.1$, this becomes $x_{n-1}^2m_{n-1}x_{n-1}m_{n-1}x_{n-1}^{-1}\in N_{n-1}^{od}$, which is equivalent to $x_{n-1}^2m_{n-1}^2[m_{n-1}^{-1},x_{n-1}]\in N_{n-1}^{od}$, which gives $x_{n-1}^2\in N_{n-1}^{od}$, which is a contradiction by the first part. Hence, we get $m_n^2\notin N_{n-1}^{od}$, as desired.).
\end{proof}
Combining Claim $6.12$, Claim $6.13$ and Claim $6.14$, Proposition $6.10$ directly follows. 
\end{proof}
\begin{corollary}
$N_n^{od}/P_{n-1}\cong \mathbb{Z}/4\mathbb{Z}\times \mathbb{Z}/2\mathbb{Z}.$
\end{corollary}
\begin{proof}
To establish the corollary, it is enough to prove the following three assertions:\\

\item[(i)] $\overline{[x_n,m_n]}$ has order $4$.
\item[(ii)] $\overline{v_{n-3}}$ has order $2$.
\item[(iii)] $\overline{[x_n,m_n]}\overline{v_{n-3}} = \overline{v_{n-3}}\overline{[x_n,m_n]}.$\\

We will prove these one by one:\\

\item[(i)] After simplifying, we have $[x_n,m_n]^4 = (x_nm_n^4x_n^{-1})(m_n^4)^{-1}$, since $x_nm_nx_n^{-1}$ and $m_n^{-1}$ commute with each other by Lemma $3.11.$ We already know that $x_nm_n^4x_n^{-1}, m_n^4\in P_{n-1}$, hence $(x_nm_n^4x_n^{-1})(m_n^4)^{-1}\in P_{n-1}$. We need to show $[x_n,m_n]^2\notin P_{n-1}.$ We clearly have $[x_n,m_n]^2 = (x_nm_n^2x_n^{-1})(m_n^2)^{-1}\notin P_{n-1} \iff m_n^2\notin N_{n-1}^{od}$, which is true by the proof of Claim $6.14$.\\

\item[(ii)] $v_{n-3}^2 = v_{n-4}v_{n-4}^{x_{n-1}}\in N_{n-1}^{od}\leq P_{n-1}$. Also, we have $v_{n-3}\notin P_{n-1} \iff v_{n-3}\notin N_{n-1}^{od}$ (since the action of $N_{n-1}^{od}$ on the tree is disjoint from the action of $v_{n-3}$), which is true by the proof of Claim $6.14$ (recall that $v_{n-3} = x_{n-1}^2$.).\\

\item[(iii)] We need to show $[[x_n,m_n],v_{n-3}]\in P_{n-1}.$ Using the fact that $x_nm_nx_n^{-1}$ and $x_nm_n^{-1}x_n^{-1}$ both commute with $m_n$ and $v_{n-3}$ (by Lemma $3.11$), after simplifying it becomes equivalent to $[m_n^{-1}, v_{n-3}]\in P_{n-1}$, which is true because $m_n,x_{n-1}\in M_{n-1}$ (by Lemma $6.1$), and $M_{n-1}'\leq N_{n-1}^{od}\leq P_{n-1}.$\\

Hence, we have proven that $N_n^{od}/P_{n-1}\cong \mathbb{Z}/4\mathbb{Z}\times \mathbb{Z}/2\mathbb{Z}$, as desired.
\end{proof}

From the proofs of Lemma $6.5$ and Corollary $6.15$, we have
\begin{equation}
M_n = \coprod_{\substack{i\in \{0,1,2,3\}\\ j\in \{0,1\}}} x_n^im_n^jN_n^{od}
\end{equation}
and
\begin{equation}
N_n^{od} = \coprod_{\substack{i\in \{0,1,2,3\}\\ j\in \{0,1\}}} [x_n,m_n]^iv_{n-3}^j P_{n-1}.
\end{equation}
Hence, if we combine $(6.17)$ and $(6.18)$, we obtain
\begin{equation}
M_n = \coprod_{\substack{i,k\in \{0,1,2,3\}\\ j,l\in \{0,1\}}} x_n^im_n^j[x_n,m_n]^kv_{n-3}^l P_{n-1}.
\end{equation}

We introduce the following notation for Model $5$:
\begin{align*}
A_1^{(n)}&:= \text{the set of cycle data obtained by applying the Markov process } n-1 \text{ times to }([nn,2],\frac{1}{4}). \\
A_2^{(n)}&:= \text{the set of cycle data obtained by applying the Markov process } n-1 \text{ times to }([sn,2],\frac{1}{4}). \\
A_3^{(n)}&:= \text{the set of cycle data obtained by applying the Markov process } n-1 \text{ times to }([ns,1][ss,1],\frac{1}{8}). \\
A_4^{(n)}&:= \text{the set of cycle data obtained by applying the Markov process } n-1 \text{ times to }([nn,1][sn,1],\frac{1}{8}). \\
A_5^{(n)}&:= \text{the set of cycle data obtained by applying the Markov process } n-1 \text{ times to }([nn,1][nn,1],\frac{1}{16}). \\
A_6^{(n)}&:= \text{the set of cycle data obtained by applying the Markov process } n-1 \text{ times to }([ns,1][ns,1],\frac{1}{16}). \\
A_7^{(n)}&:= \text{the set of cycle data obtained by applying the Markov process } n-1 \text{ times to }([sn,1][sn,1],\frac{1}{16}). \\
A_8^{(n)}&:= \text{the set of cycle data obtained by applying the Markov process } n-1 \text{ times to }([ss,1][ss,1],\frac{1}{16}). 
\end{align*}
\begin{proposition}
We have the following equalities:\\
\item[(a)] $A_1^{(n)} = CD(x_nN_n^{od}\sqcup x_n^3N_n^{od},M_n).$\\
\item[(b)] $A_2^{(n)} = CD(x_nm_nN_n^{od}\sqcup x_n^3m_nN_n^{od},M_n).$\\
\item[(c)] $A_3^{(n)} = CD(m_nN_n^{od},M_n).$\\
\item[(d)] $A_4^{(n)} = CD(x_n^2m_nN_n^{od},M_n).$\\
\item[(e)] $A_5^{(n)} = CD(\coprod_{i,j\in \{0,1\}} x_n^2[x_n,m_n]^{2i}v_{n-3}^jP_{n-1},M_n).$\\
\item[(f)] $A_6^{(n)} = CD(\coprod_{i,j\in \{0,1\}}[x_n,m_n]^{2i+1}v_{n-3}^jP_{n-1},M_n).$\\
\item[(g)] $A_7^{(n)} = CD(\coprod_{i,j\in \{0,1\}} x_n^2[x_n,m_n]^{2i+1}v_{n-3}^jP_{n-1},M_n).$\\
\item[(h)] $A_8^{(n)} = CD(\coprod_{i,j\in \{0,1\}}[x_n,m_n]^{2i}v_{n-3}^jP_{n-1},M_n).$
\end{proposition}
\begin{customrmk}{}\label{A}
	\textbf{Note that proving Proposition} $\boldsymbol{6.16}$ \textbf{is enough for establishing Theorem C, because it shows that the cycle data of } $\boldsymbol{M_n}$ \textbf{match the cycle data of the} $\boldsymbol{n}$\textbf{th level of the even Markov model.}
\end{customrmk}
\begin{proof}[Proof of Proposition $6.16$]
We first give direct proofs for $(a)$ and $(e)$, and then prove the remaining parts using induction.\\

\item[\textbf{Proof of (a).}] Note that by the Markov process, we have $A_1^{(n)} = \{([2^n],\frac{1}{4})\}$. Also, by the definition of $N_n^{od}$, cycle types of all the permutations in $x_nN_n^{od}\sqcup x_n^3N_n^{od}$ are of the form $[2^n]$ as well, and they correspond to $\frac{1}{4}$ of $M_n$, which finishes the proof.\\

\item[\textbf{Proof of (e).}] First note that by the Markov process, we have $A_5^{(n)} = \{([2^{n-1},2^{n-1}],\frac{1}{16})\}$.\\

We will look at the union $\coprod_{i,j\in \{0,1\}} x_n^2[x_n,m_n]^{2i}v_{n-3}^jP_{n-1}.$ We have
\begin{align*}
\coprod_{i,j\in \{0,1\}} x_n^2[x_n,m_n]^{2i}v_{n-3}^jP_{n-1}&= \coprod_{i,j\in \{0,1\}} x_n^2[x_n,m_n]^{2i}v_{n-3}^j (N_{n-1}^{od}\times (N_{n-1}^{od})^{x_n})\\
&= (x_{n-1}^{2j+1}m_n^{-2i}N_{n-1}^{od})\times (x_{n-1}m_n^{2i}N_{n-1}^{od})^{x_n},
\end{align*}
where we used Lemma $3.11$ and the identity $x_n^2 = x_{n-1}x_{n-1}^{x_n}$.
The cosets $m_n^{-2i}N_{n-1}^{od}$ and $m_n^{2i}N_{n-1}^{od}$ both lie in $M_{n-1}^{od}$ (since $m_n^2\in M_{n-1}^{od}$ and $N_{n-1}^{od}\leq M_{n-1}^{od}$), which shows that all the cycle types occuring in this union are of the form $[2^{n-1},2^{n-1}]$. Noting that $$\frac{|\coprod_{i,j\in \{0,1\}} x_n^2[x_n,m_n]^{2i}v_{n-3}^jP_{n-1}|}{|M_n|} = \frac{4|P_{n-1}|}{64|P_{n-1}|} = \frac{1}{16}$$
(where, the first equality follows from $(6.19)$), we are done.\\

We will now prove the remaining parts using induction. All the statements of Proposition $6.16$ are clearly true for $n=1,2$. Suppose that they are true for some $n=k\geq 2$. We will prove each of remaining statements for $n=k+1$.\\

\item[\textbf{Proof of (b).}] We start with the following lemma:
\begin{lemma}
For any $n\geq 2$, we have $A_2^{(n)} = dA_3^{(n-1)}\sqcup dA_4^{(n-1)}$.
\end{lemma}
\begin{proof}
Can be proven very similarly to Lemma $5.9$. Details are omitted.
\end{proof}
We will study the cosets $x_{k+1}m_{k+1}N_{k+1}^{od}$ and $x_{k+1}^3m_{k+1}N_{k+1}^{od}$. We start by proving the following lemma which in particular shows that these two cosets are always conjugate to each other.\\

\begin{lemma}
We have the following equalities:\\

\item[(i)] $(x_{k+1}N_{k+1}^{od})^{a_{k+1}} = x_{k+1}^3N_{k+1}^{od}.$
\item[(ii)] $(x_{k+1}m_{k+1}N_{k+1}^{od})^{a_{k+1}} = x_{k+1}^3m_{k+1}N_{k+1}^{od}.$
\end{lemma}

\begin{proof}[Proof of Lemma $6.18$]
\item[(i)] We have $(N_{k+1}^{od})^{a_{k+1}} = (N_{k+1}^{od})^{x_{k+1}} = N_{k+1}^{od}$, so $a_{k+1}$ normalizes $N_{k+1}^{od}$. Hence, we have
\begin{align*}
(x_{k+1}N_{k+1}^{od})^{a_{k+1}} = x_{k+1}^3N_{k+1}^{od}&\iff \\
x_{k+1}^{-3}a_{k+1}x_{k+1}N_{k+1}^{od}a_{k+1} = N_{k+1}^{od}&\iff \\
x_{k+1}^{-3}a_{k+1}x_{k+1}a_{k+1}^{-1}(N_{k+1}^{od})^{a_{k+1}} = N_{k+1}^{od}&\iff \\
x_{k+1}^{-3}a_{k+1}x_{k+1}a_{k+1}^{-1}\in N_{k+1}^{od}&\iff \\
(x_{k+1}a_{k+1})^2\in N_{k+1}^{od} \text{ (since } x_{k+1}^4,a_{k+1}^2\in N_{k+1}^{od}), 
\end{align*}
which is true because $(x_{k+1}a_{k+1})^2 = x_k^2\in N_{k+1}$.\\

\item[(ii)] We have
\begin{align*}
(x_{k+1}m_{k+1}N_{k+1}^{od})^{a_{k+1}} = x_{k+1}^3m_{k+1}N_{k+1}^{od}&\iff \\
m_{k+1}^{-1}x_{k+1}^{-3}a_{k+1}x_{k+1}m_{k+1}N_{k+1}^{od}a_{k+1}^{-1} = N_{k+1}^{od}&\iff \\
m_{k+1}^{-1}x_{k+1}^{-3}a_{k+1}x_{k+1}m_{k+1}a_{k+1}^{-1}(N_{k+1}^{od})^{a_{k+1}} = N_{k+1}^{od}&\iff \\
m_{k+1}^{-1}x_{k+1}^{-3}a_{k+1}x_{k+1}m_{k+1}a_{k+1}^{-1}N_{k+1}^{od} = N_{k+1}^{od}&\iff \\
 m_{k+1}^{-1}x_{k+1}^{-3}a_{k+1}x_{k+1}m_{k+1}a_{k+1}^{-1}\in N_{k+1}^{od}.
\end{align*}
Since $x_{k+1}^{-3}a_{k+1}x_{k+1}a_{k+1}^{-1}\in N_{k+1}^{od}$ from part (i), the last line becomes
$$[m_{k+1}^{-1},a_{k+1}]\in N_{k+1}^{od},$$
which is true because $[m_{k+1}^{-1},a_{k+1}] = m_{k+1}^{-1}m_{k+1}^{a_{k+1}} = m_{k+1}^{-1}m_{k+1}^{x_{k+1}} = [m_{k+1}^{-1},x_{k+1}]\in N_{k+1}^{od}$, where the inclusion holds by the definiton of $N_{k+1}^{od}$.
\end{proof}
Hence, in the light of the part (ii) of Lemma $6.18$, we have $\text{CD}(x_{k+1}m_{k+1}N_{k+1}^{od}\sqcup x_{k+1}^3m_{k+1}N_{k+1}^{od},M_{k+1}) = 2\text{CD}(x_{k+1}m_{k+1}N_{k+1}^{od},M_{k+1}).$ We will study the coset $x_{k+1}m_{k+1}N_{k+1}^{od}.$\\

Recall that we have 
$$x_{k+1}m_{k+1}N_{k+1}^{od} = \coprod_{\substack{i\in \{0,1,2,3\}\\ j\in \{0,1\}}} x_{k+1}m_{k+1}[x_{k+1},m_{k+1}]^iv_{k-2}^j P_{k}.$$ For simplicity, we set $Y_{i,j}^{(k+1)} = [x_{k+1},m_{k+1}]^iv_{k-2}^j P_{k}.$ Hence, we have
\begin{equation}
x_{k+1}m_{k+1}N_{k+1}^{od} = \coprod_{\substack{i\in \{0,1,2,3\}\\ j\in \{0,1\}}} x_{k+1}m_{k+1}Y_{i,j}^{(k+1)}.
\end{equation}
We first express $Y_{i,j}^{(k+1)}$ for each pair $(i,j)$ in a nice way. Using the facts that $m_{k+1} = x_{k}^2m_{k}x_{k}^{-1}$, $v_{k-2} = x_{k}^2$ and $x_{k+1} =x_{k}a_{k+1}$, by direct computation, we get the following expressions:\\

\begin{equation}
x_{k+1}m_{k+1}Y_{0,0}^{(k+1)} = (x_{k}a_{k+1})[(x_{k}m_{k}N_{k}^{od})\times (N_{k}^{od})^{x_{k+1}}].
\end{equation}
\begin{equation}
x_{k+1}m_{k+1}Y_{1,0}^{(k+1)} = (x_{k}a_{k+1})[(N_{k}^{od})\times (x_{k}m_{k}N_{k}^{od})^{x_{k+1}}].
\end{equation}
\begin{equation}
x_{k+1}m_{k+1}Y_{2,0}^{(k+1)} = (x_{k}a_{k+1})[(x_{k}^3m_{k}N_{k}^{od})\times (x_{k}^2N_{k}^{od})^{x_{k+1}}].
\end{equation}
\begin{equation}
x_{k+1}m_{k+1}Y_{3,0}^{(k+1)} = (x_{k}a_{k+1})[(x_{k}^2N_{k}^{od})\times (x_{k}^3m_{k}N_{k}^{od})^{x_{k+1}}].
\end{equation}
\begin{equation}
x_{k+1}m_{k+1}Y_{0,1}^{(k+1)} = (x_{k}a_{k+1})[(x_{k}^3m_{k}N_{k}^{od})\times (N_{k}^{od})^{x_{k+1}}].
\end{equation}
\begin{equation}
x_{k+1}m_{k+1}Y_{1,1}^{(k+1)} = (x_{k}a_{k+1})[(x_{k}^2N_{k}^{od})\times (x_{k}m_{k}N_{k}^{od})^{x_{k+1}}].
\end{equation}
\begin{equation}
x_{k+1}m_{k+1}Y_{2,1}^{(k+1)} = (x_{k}a_{k+1})[(x_{k}m_{k}N_{k}^{od})\times (x_{k}^2N_{k}^{od})^{x_{k+1}}].
\end{equation}
\begin{equation}
x_{k+1}m_{k+1}Y_{3,1}^{(k+1)} = (x_{k}a_{k+1})[(N_{k}^{od})\times (x_{k}^3m_{k}N_{k}^{od})^{x_{k+1}}].
\end{equation}
By the symmetry, we clearly have
\begin{equation}
\text{CD}(x_{k+1}m_{k+1}Y_{0,0}^{(k+1)},M_{k+1}) = \text{CD}(x_{k+1}m_{k+1}Y_{1,0}^{(k+1)},M_{k+1}),
\end{equation}
\begin{equation}
\text{CD}(x_{k+1}m_{k+1}Y_{2,0}^{(k+1)},M_{k+1}) = \text{CD}(x_{k+1}m_{k+1}Y_{3,0}^{(k+1)},M_{k+1}),
\end{equation}
\begin{equation}
\text{CD}(x_{k+1}m_{k+1}Y_{0,1}^{(k+1)},M_{k+1}) = \text{CD}(x_{k+1}m_{k+1}Y_{3,1}^{(k+1)},M_{k+1}),
\end{equation}
and
\begin{equation}
\text{CD}(x_{k+1}m_{k+1}Y_{1,1}^{(k+1)},M_{k+1}) = \text{CD}(x_{k+1}m_{k+1}Y_{2,1}^{(k+1)},M_{k+1}).
\end{equation}
Hence, it is enough to understand the cosets (say) $x_{k+1}m_{k+1}Y_{0,0}^{(k+1)}$, $x_{k+1}m_{k+1}Y_{2,0}^{(k+1)}$, $x_{k+1}m_{k+1}Y_{0,1}^{(k+1)}$ and $x_{k+1}m_{k+1}Y_{1,1}^{(k+1)}$.\\

\item[\underline{$x_{k+1}m_{k+1}Y_{0,0}^{(k+1)}$}:] We start with the following lemma:\\

\begin{lemma}
The three cosets $x_{k}a_{k+1}x_{k}m_{k}N_{k}^{od}$, $a_{k+1}x_{k}^2m_{k}N_{k}^{od}$ and $a_{k+1}x_{k}^2m_{k}N_{k}^{od}n^{x_{k+1}}$ are conjugate under $W_{k+1}$ for all $n\in N_{k}^{od}.$
\end{lemma}
\begin{proof}
We will first prove that $$a_{k+1}x_{k}^2m_{k}N_{k}^{od} = (x_{k}a_{k+1}x_{k}m_{k}N_{k}^{od})^{x_{k}^{-1}}.$$
It holds, because we have
\begin{align*}
(x_{k}a_{k+1}x_{k}m_{k}N_{k}^{od})^{x_{k}^{-1}}&= a_{k+1}x_{k}m_{k}N_{k}^{od}x_{k}\\
&= a_{k+1}x_{k}m_{k}x_{k}N_{k}^{od}\\
&= a_{k+1}x_{k}^2m_{k}N_{k}^{od}.
\end{align*}
Second, we will prove that
$$a_{k+1}x_{k}^2m_{k}N_{k}^{od} = (a_{k+1}x_{k}^2m_{k}N_{k}^{od}n^{x_{k+1}})^{n^{x_{k+1}}}$$ for all $n\in N_{k}^{od}$. First note that $n^{x_{k+1}} = n^{a_{k+1}}$. The equality we claim holds, because we have
\begin{align*}
(a_{k+1}x_{k}^2m_{k}N_{k}^{od}n^{x_{k+1}})^{n^{x_{k+1}}}&= (a_{k+1}x_{k}^2m_{k}N_{k}^{od}n^{a_{k+1}})^{n^{a_{k+1}}}\\
&= a_{k+1}nx_{k}^2m_{k}N_{k}^{od}\\
&= a_{k+1}x_{k}^2m_{k}N_{k}^{od}.
\end{align*}
This finishes the proof.
\end{proof}
Using $(6.21)$ and Lemma $6.19$, we obtain
\begin{equation}
\text{CD}(x_{k+1}m_{k+1}Y_{0,0}^{(k+1)},M_{k+1}) = |N_{k}^{od}|\text{CD}(a_{k+1}(x_{k}^2m_{k}N_{k}^{od}),M_{k+1})
\end{equation}
By the induction, we already have $\text{CD}(x_{k}^2m_{k}N_{k}^{od},M_{k}) = A_4^{(k)}$. Using Lemma $3.20$, and arguing similarly to the proof of Lemma $5.12$, it follows that
\begin{equation}
|N_{k}^{od}|\text{CD}(a_{k+1}(x_{k}^2m_{k}N_{k}^{od}),M_{k+1}) = \frac{1}{4}d[\text{CD}(x_{k}^2m_{k}N_{k}^{od},M_{k})] =\frac{1}{4}dA_4^{(k)},
\end{equation}
which, using $(6.33)$, gives
\begin{equation}
\text{CD}(x_{k+1}m_{k+1}Y_{0,0}^{(k+1)},M_{k+1}) = \frac{1}{4}dA_4^{(k)}.
\end{equation}

\item[\underline{$x_{k+1}m_{k+1}Y_{2,0}^{(k+1)}$}:] We start with the following lemma:

\begin{lemma}
The following is true:
\item[(i)] $x_{k}a_{k+1}x_{k}^3m_{k}N_{k}^{od}$ is conjugate to $a_{k+1}m_{k}N_{k}^{od}$ under $W_{k+1}$.\\
\item[(ii)] $a_{k+1}m_{k}N_{k}^{od}(x_{k}^2n)^{x_{k+1}}$ is conjugate to $a_{k+1}x_{k}^2m_{k}N_{k}^{od}$ under $W_{k+1}$ for all $n\in N_{k}^{od}.$
\end{lemma}
\begin{proof}
\item[(i)] We will prove that $$(x_{k}a_{k+1}x_{k}^3m_{k}N_{k}^{od})^{x_{k}^{-1}} = a_{k+1}m_{k}N_{k}^{od}.$$
It holds, because we have
\begin{align*}
(x_{k}a_{k+1}x_{k}^3m_{k}N_{k}^{od})^{x_{k}^{-1}}&= a_{k+1}x_{k}^3m_{k}N_{k}^{od}x_{k}\\
&= a_{k+1}x_{k}^3m_{k}x_{k}N_{k}^{od}\\
&= a_{k+1}m_{k}N_{k}^{od}.
\end{align*}
\item[(ii)] We will prove that 
$$(a_{k+1}m_{k}N_{k}^{od}(x_{k}^2n)^{x_{k+1}})^{(x_{k}^2n)^{x_{k+1}}} = a_{k+1}x_{k}^2m_{k}N_{k}^{od}$$
for all $n\in N_{k}^{od}$. Note that $(x_{k}^2n)^{x_{k+1}} = (x_{k}^2n)^{a_{k+1}}$. The equality we claim holds, because we have
\begin{align*}
(a_{k+1}m_{k}N_{k}^{od}(x_{k}^2n)^{x_{k+1}})^{(x_{k}^2n)^{x_{k+1}}}&= (a_{k+1}m_{k}N_{k}^{od}(x_{k}^2n)^{a_{k+1}})^{(x_{k}^2n)^{a_{k+1}}}\\
&= a_{k+1}x_{k}^2nm_{k}N_{k}^{od}\\
&= a_{k+1}x_{k}^2m_{k}N_{k}^{od}.
\end{align*}
\end{proof}
Using $(6.23)$ and Lemma $6.20$, we obtain
\begin{equation}
\text{CD}(x_{k+1}m_{k+1}Y_{2,0}^{(k+1)},M_{k+1}) = |N_{k}^{od}|\text{CD}(a_{k+1}(x_{k}^2m_{k}N_{k}^{od}),M_{k+1}).
\end{equation}
Hence, similar to the previous part, we get
\begin{equation}
\text{CD}(x_{k+1}m_{k+1}Y_{2,0}^{(k+1)},M_{k+1}) = \frac{1}{4}dA_4^{(k)}.
\end{equation}

\item[\underline{$x_{k+1}m_{k+1}Y_{0,1}^{(k+1)}$}:] We start with the following lemma:
\begin{lemma}
The three cosets $x_{k}a_{k+1}x_{k}^3m_{k}N_{k}^{od}$, $a_{k+1}m_{k}N_{k}^{od}$ and \\$a_{k+1}m_{k}N_{k}^{od}n^{x_{k+1}}$ are conjugate under $W_{k+1}$ for all $n\in N_{k}^{od}$.
\end{lemma}
\begin{proof}
We will first prove that
$$x_{k}a_{k+1}x_{k}^3m_{k}N_{k}^{od} = (a_{k+1}m_{k}N_{k}^{od})^{x_{k}}.$$
It holds, because we have
\begin{align*}
(a_{k+1}m_{k}N_{k}^{od})^{x_{k}}&= x_{k}a_{k+1}m_{k}N_{k}^{od}x_{k}^{-1}\\
&= x_{k}a_{k+1}m_{k}x_{k}^{-1}N_{k}^{od}\\
&= x_{k}a_{k+1}x_{k}^3m_{k}N_{k}^{od}.
\end{align*}
Secondly, we will prove that
$$a_{k+1}m_{k}N_{k}^{od} = (a_{k+1}m_{k}N_{k}^{od}n^{x_{k+1}})^{n^{x_{k+1}}}.$$
Note that $n^{x_{k+1}} = n^{a_{k+1}}$. Then the equality we claim holds, because we have
\begin{align*}
(a_{k+1}m_{k}N_{k}^{od}n^{x_{k+1}})^{n^{x_{k+1}}}&= (a_{k+1}m_{k}N_{k}^{od}n^{a_{k+1}})^{n^{a_{k+1}}}\\
&= a_{k+1}nm_{k}N_{k}^{od}\\
&= a_{k+1}m_{k}N_{k}^{od}.
\end{align*}
\end{proof}
Using $(6.25)$ and Lemma $6.21$, we obtain
\begin{equation}
\text{CD}(x_{k+1}m_{k+1}Y_{0,1}^{(k+1)},M_{k+1}) = |N_{k}^{od}|\text{CD}(a_{k+1}(m_{k}N_{k}^{od}),M_{k+1}).
\end{equation}
By the induction, we already have $\text{CD}(m_{k}N_{k}^{od},M_{k}) = A_3^{(k)}$. Using Lemma $3.20$, and arguing similarly to the proof of Lemma $5.12$, it follows that
\begin{equation}
|N_{k}^{od}|\text{CD}(a_{k+1}(m_{k}N_{k}^{od}),M_{k+1}) = \frac{1}{4}d[\text{CD}(m_{k}N_{k}^{od},M_{k})] =\frac{1}{4}dA_3^{(k)},
\end{equation}
which, using $(6.38)$, gives
\begin{equation}
\text{CD}(x_{k+1}m_{k+1}Y_{0,1}^{(k+1)},M_{k+1}) = \frac{1}{4}dA_3^{(k)}.
\end{equation}
\item[\underline{$x_nm_nY_{1,1}^{(k+1)}$}:] We start with the following lemma:
\begin{lemma}
The following is true:\\
\item[(i)] $x_{k}a_{k+1}x_{k}m_{k}N_{k}^{od}$ is conjugate to $a_{k+1}x_{k}^2m_{k}N_{k}^{od}$ under $W_{k+1}$.\\
\item[(ii)] $x_{k}a_{k+1}x_{k}m_{k}N_{k}^{od}(x_{k}^2n)^{x_{k+1}}$ is conjugate to $a_{k+1}m_{k}N_{k}^{od}$ under $W_{k+1}$.
\end{lemma}
\begin{proof}
\item[(i)] We have \begin{align*}
(a_{k+1}x_{k}^2m_{k}N_{k}^{od})^{x_{k}}&= x_{k}a_{k+1}x_{k}^2m_{k}N_{k}^{od}x_{k}^{-1}\\
&=x_{k}a_{k+1}x_{k}^2m_{k}x_{k}^{-1}N_{k}^{od} \\
&= x_{k}a_{k+1}x_{k}m_{k}N_{k}^{od},
\end{align*}
which proves what we want.\\

\item[(ii)] Note that $(x_{k}^2n)^{x_{k+1}} = (x_{k}^2n)^{a_{k+1}}$. We have \begin{align*}
(x_{k}a_{k+1}x_{k}m_{k}N_{k}^{od}(x_{k}^2n)^{x_{k+1}})^{(x_{k}^2n)^{x_{k+1}}}&= ((x_{k}^2n)^{a_{k+1}})^{-1}x_{k}a_{k+1}x_{k}m_{k}N_{k}^{od}\\
&=a_{k+1}n^{-1}x_{k}^{-2}a_{k+1}x_{k}a_{k+1}x_{k}m_{k}N_{k}^{od} \\
&= x_{k}a_{k+1}n^{-1}x_{k}^{-1}m_{k}N_{k}^{od}\\
&= x_{k}a_{k+1}x_{k}^3m_{k}N_{k}^{od},
\end{align*}
which is conjugate to $a_{k+1}m_{k}N_{k}^{od}$ under $W_{k+1}$ by part (i) of Lemma $6.20$, which finishes the proof.
\end{proof}
Using $(6.26)$ and Lemma $6.22$, we obtain
\begin{equation}
\text{CD}(x_{k+1}m_{k+1}Y_{1,1}^{(k+1)},M_{k+1}) = |N_{k}^{od}|\text{CD}( a_{k+1}(m_{k}N_{k}^{od}),M_{k+1}).
\end{equation}
Hence, similar to the previous part, we get
\begin{equation}
\text{CD}(x_{k+1}m_{k+1}Y_{1,1}^{(k+1)},M_{k+1}) = \frac{1}{4}dA_3^{(k)}.
\end{equation}

Combining $(6.29)-(6.32)$, $(6.35)$, $(6.37)$, $(6.40)$ and $(6.42)$, and also using ($6.20$) and Lemma $6.17$, we obtain
$$\text{CD}(x_{k+1}m_{k+1}N_{k+1}^{od},M_{k+1}) = \frac{1}{2}[dA_3^{(k)}\sqcup dA_4^{(k)}] = \frac{1}{2}[A_2^{(k+1)}].$$
This gives $$\text{CD}(x_{k+1}m_{k+1}N_{k+1}^{od}\sqcup x_{k+1}^3m_{k+1}N_{k+1}^{od},M_{k+1}) = 2\text{CD}(x_{k+1}m_{k+1}N_{k+1}^{od},M_{k+1}) = A_2^{(k+1)},$$ which finishes the proof of (b).\\

\item[\textbf{Proof of (c).}] We start with the following lemma:
\begin{lemma}
For any $n\geq 2$, we have $A_3^{(n)} = A_2^{(n-1)}\times (A_5^{(n-1)}\sqcup A_6^{(n-1)}\sqcup A_7^{(n-1)}\sqcup A_8^{(n-1)}).$
\end{lemma} 
\begin{proof}
Can be proven similarly to Lemma $5.13$. Details are omitted.
\end{proof}

Using the computation in the proof of part $(e)$, we can write
\begin{equation}
N_{k+1}^{od} = \coprod_{\substack{i\in \{0,1,2,3\}\\ j\in \{0,1\}}}(x_{k}^{2j}m_{k+1}^{-2i}N_{k}^{od})\times (m_{k+1}^{2i}N_{k}^{od})^{x_{k+1}}.
\end{equation}
Using this and the identity $m_{k+1} = x_{k}^2m_{k}x_{k}^{-1}$, we obtain
\begin{equation}
m_{k+1}N_{k+1}^{od} = \coprod_{\substack{i\in \{0,1,2,3\}\\ j\in \{0,1\}}}(x_{k}^{2i+2j+1}m_{k}^{-2i+1}N_{k}^{od})\times (x_{k}^{-2i}N_{k}^{od})^{x_{k+1}},
\end{equation}
which after simplification becomes
\begin{equation}
m_{k+1}N_{k+1}^{od} = \coprod_{\substack{i\in \{0,1,2,3\}\\ j\in \{0,1\}}}(x_{k}^{2i+2j+1}m_{k}N_{k}^{od})\times (x_{k}^{2i}N_{k}^{od})^{x_{k+1}}
\end{equation}
By the induction assumption, for any $i,j$, we have \begin{equation}
\text{CD}(x_{k}^{2i+2j+1}m_{k}N_{k}^{od},M_{k}) = A_2^{(k)},
\end{equation}
and for $i\in \{1,3\}$ we have \begin{equation}
\text{CD}(x_{k}^{2i}N_{k}^{od},M_{k}) = A_5^{(k)}\sqcup A_7^{(k)},
\end{equation}
and also for $i\in \{0,2\}$ we have
\begin{equation}
\text{CD}(x_{k}^{2i}N_{k}^{od},M_{k}) = A_6^{(k)}\sqcup A_8^{(k)}.
\end{equation}

Combining $(6.46)$ with $(6.47)$, and arguing similarly to the proof of Claim $4.6$, we get
\begin{equation}
\text{CD}(\coprod_{\substack{i\in \{1,3\}\\ j\in \{0,1\}}}(x_{k}^{2i+2j+1}m_{k}N_{k}^{od})\times (x_{k}^{2i}N_{k}^{od})^{x_{k+1}},M_{k+1}) = A_2^{(k)}\times (A_5^{(k)}\sqcup A_7^{(k)}).
\end{equation}
Similarly, combining $(6.46)$ with $(6.48)$, we get
\begin{equation}
\text{CD}(\coprod_{\substack{i\in \{0,2\}\\ j\in \{0,1\}}}(x_{k}^{2i+2j+1}m_{k}N_{n-1}^{od})\times (x_{k}^{2i}N_{k}^{od})^{x_{k+1}},M_{k+1}) = A_2^{(k)}\times (A_6^{(k)}\sqcup A_8^{(k)}).
\end{equation}

Combining $(6.49)$ with $(6.50)$, and using Lemma $6.23$ finishes the proof of (c).\\

\item[\textbf{Proof of (d).}] We start with the following lemma:
\begin{lemma}
For any $n\geq 2$, $A_4^{(n)} = A_1^{(n-1)}\times (A_3^{(n-1)}\sqcup A_4^{(n-1)}).$
\end{lemma}
\begin{proof}
	Can be proven similarly to Lemma $5.13$. Details are omitted.
\end{proof}
Using $(6.45)$ and the identities $x_{k+1}^2 = x_{k}x_{k}^{x_{k+1}}$ and $m_{k+1} = x_{k}^2m_{k}x_{k}^{-1}$, we obtain
\begin{equation}
x_{k+1}^2m_{k+1}N_{k+1}^{od} = \coprod_{\substack{i\in \{0,1,2,3\}\\ j\in \{0,1\}}}(x_{k}^{2i+2j+2}m_{k}N_{k}^{od})\times (x_{k}^{2i+1}N_{k}^{od})^{x_{k+1}}
\end{equation}
By the induction assumption, for all $i,j$ with $i+j$ even, we have
\begin{equation}
\text{CD}(x_{k}^{2i+2j+2}m_{k}N_{k}^{od},M_{k}) = A_4^{(k)},
\end{equation}
and for all $i,j$ with $i+j$ odd we have
\begin{equation}
\text{CD}(x_{k}^{2i+2j+2}m_{k}N_{k}^{od},M_{k}) = A_3^{(k)},
\end{equation}
and we already have 
\begin{equation}
\text{CD}(x_{k}^{2i+1}N_{k}^{od},M_{k}) = A_1^{(k)}.
\end{equation}
Combining $(6.52)$ with $(6.54)$, and arguing similarly to the proof of Claim $4.6$, we get
\begin{equation}
\text{CD}(\coprod_{\substack{i\in \{0,1,2,3\}\\ j\in \{0,1\}\\i+j \equiv 0 \mod 2}}(x_{k}^{2i+2j+2}m_{k}N_{k}^{od})\times (x_{k}^{2i+1}N_{k}^{od})^{x_{k+1}},M_{k+1}) = A_4^{(k)}\times A_1^{(k)}.
\end{equation}
Combining $(6.52)$ with $(6.53)$, and arguing similarly to the proof of Claim $4.6$, we also get
\begin{equation}
\text{CD}(\coprod_{\substack{i\in \{0,1,2,3\}\\ j\in \{0,1\}\\i+j \equiv 1 \mod 2}}(x_{k}^{2i+2j+2}m_{k}N_{k}^{od})\times (x_{k}^{2i+1}N_{k}^{od})^{x_{k+1}},M_{k+1}) = A_3^{(k)}\times A_1^{(k)}.
\end{equation}
Combining $(6.55)$ with $(6.56)$, and using $(6.51)$ and Lemma $6.24$, we are done with (d).\\

\item[\textbf{Proof of (f).}] We have the following lemma:
\begin{lemma}
For any $n\geq 2$, we have $A_6^{(n)} = A_2^{(n-1)}\times A_2^{(n-1)}.$
\end{lemma}
\begin{proof}
Can be proven similarly to Lemma $5.13$. Details are omitted.
\end{proof}	
Similar to the computation in the proof of $(e)$, we have
\begin{equation}
\coprod_{i,j\in \{0,1\}}[x_{k+1},m_{k+1}]^{2i+1}v_{k-2}^jP_{k} = \coprod_{\substack{i,j\in \{0,1\}}}(x_{k}^{-2i+2j-1}m_{k}^{-2i-1}N_{k}^{od})\times (x_{k}^{2i+1}m_{k}^{2i+1}N_{k}^{od})^{x_{k+1}}.
\end{equation}
By simplifying, it becomes
\begin{equation}
\coprod_{i,j\in \{0,1\}}[x_{k+1},m_{k+1}]^{2i+1}v_{k-2}^jP_{k}=\coprod_{\substack{i,j\in \{0,1\}}}(x_{k}^{2i+2j-1}m_{k}N_{k}^{od})\times (x_{k}^{2i+1}m_{k}N_{k}^{od})^{x_{k+1}}.
\end{equation}
By the induction assumption, for all $i,j$, we have
\begin{equation}
\text{CD}(x_{k}^{2i+2j-1}m_{k}N_{k}^{od},M_{k}) = \text{CD}(x_{k}^{2i+1}m_{k}N_{k}^{od},M_{k}) = A_2^{(k)}.
\end{equation}
Using $(6.58)$ and $(6.59)$, and arguing similarly to the proof of Claim $4.6$, we get 
\begin{equation}
\text{CD}(\coprod_{i,j\in \{0,1\}}[x_{k+1},m_{k+1}]^{2i+1}v_{k-2}^jP_{k}) = A_2^{(k)}\times A_2^{(k)},
\end{equation}
as desired.\\

\item[\textbf{Proof of (g).}] We have the following lemma:
\begin{lemma}
For any $n\geq 2$, we have $A_7^{(n)} = (A_3^{(n-1)}\sqcup A_4^{(n-1)})\times (A_3^{(n-1)}\sqcup A_4^{(n-1)}).$
\end{lemma}
\begin{proof}
	Can be proven similarly to Lemma $5.13.$ Details are omitted.
\end{proof}
Similar to the computation in the proof of $(e)$, we have
\begin{equation}
\coprod_{i,j\in \{0,1\}}x_{k+1}^2[x_{k+1},m_{k+1}]^{2i+1}v_{k-2}^jP_{k} = \coprod_{\substack{i,j\in \{0,1\}}}(x_{k}^{-2i+2j}m_{k}^{-2i-1}N_{k}^{od})\times (x_{k}^{2i+2}m_{k}^{2i+1}N_{k}^{od})^{x_{k+1}}.
\end{equation}By simplifying, it becomes
\begin{equation}
\coprod_{i,j\in \{0,1\}}x_{k+1}^2[x_{k+1},m_{k+1}]^{2i+1}v_{k-2}^jP_{k} =\coprod_{\substack{i,j\in \{0,1\}}}(x_{k}^{2i+2j}m_{k}N_{k}^{od})\times (x_{k}^{2i+2}m_{k}N_{k}^{od})^{x_{k+1}}.
\end{equation}
Using the induction assumption, for $i,j$ with $i+j$ even, we have
\begin{equation}
\text{CD}(x_{k}^{2i+2j}m_{k}N_{k}^{od},M_{k}) = A_3^{(k)},
\end{equation}
and for $i,j$ with $i+j$ odd, we have
\begin{equation}
\text{CD}(x_{k}^{2i+2j}m_{k}N_{k}^{od},M_{k}) = A_4^{(k)},
\end{equation}
and for $i$ even, we have
\begin{equation}
\text{CD}(x_{k}^{2i+2}m_{k}N_{k}^{od},M_{k}) = A_3^{(k)},
\end{equation}
and for $i$ odd, we have
\begin{equation}
\text{CD}(x_{k}^{2i+2}m_{k}N_{k}^{od},M_{k}) = A_4^{(k)}.
\end{equation}
Using these, and arguing similarly to the proof of Claim $4.6$, we get
\begin{equation}
\text{CD}(\coprod_{\substack{i=0, j=0}}(x_{k}^{2i+2j}m_{k}N_{k}^{od})\times (x_{k}^{2i+2}m_{k}N_{k}^{od})^{x_{k+1}},M_{k+1}) = A_3^{(k)}\times A_3^{(k)},
\end{equation}
\begin{equation}
\text{CD}(\coprod_{\substack{i=0, j=1}}(x_{k}^{2i+2j}m_{k}N_{k}^{od})\times (x_{k}^{2i+2}m_{k}N_{k}^{od})^{x_{k+1}},M_{k+1}) = A_4^{(k)}\times A_4^{(k)},
\end{equation}
\begin{equation}
\text{CD}(\coprod_{\substack{i=1, j=0}}(x_{k}^{2i+2j}m_{k}N_{k}^{od})\times (x_{k}^{2i+2}m_{k}N_{k}^{od})^{x_{k+1}},M_{k+1}) = A_4^{(k)}\times A_3^{(k)},
\end{equation}
and
\begin{equation}
\text{CD}(\coprod_{\substack{i=1, j=1}}(x_{k}^{2i+2j}m_{k}N_{k}^{od})\times (x_{k}^{2i+2}m_{k}N_{k}^{od})^{x_{k+1}},M_{k+1}) = A_3^{(k)}\times A_4^{(k)}.
\end{equation}
Combining $(6.67)-(6.70)$, and using $(6.62)$ and Lemma $6.26$, the proof of (g) is completed.\\

\item[\textbf{Proof of (h).}] We start with the following lemma:
\begin{lemma}
For any $n\geq 2$, we have $A_7^{(n)} = (A_5^{(n-1)}\sqcup A_6^{(n-1)}\sqcup A_7^{(n-1)}\sqcup A_8^{(n-1)})\times (A_5^{(n-1)}\sqcup A_6^{(n-1)}\sqcup A_7^{(n-1)}\sqcup A_8^{(n-1)}).$
\end{lemma}
\begin{proof}
Can be proven similarly to Lemma $5.13.$ Details are omitted.
\end{proof}

Similar to the computation in the proof of $(e)$, we have
\begin{equation}
\coprod_{i,j\in \{0,1\}}[x_{k+1},m_{k+1}]^{2i}v_{k-2}^jP_{k} = \coprod_{\substack{i,j\in \{0,1\}}}(x_{k}^{-2i+2j}m_{k}^{-2i}N_{k}^{od})\times (x_{k}^{2i}m_{k}^{2i}N_{k}^{od})^{x_{k+1}}.
\end{equation}
By simplifying, it becomes
\begin{equation}
\coprod_{i,j\in \{0,1\}}[x_{k+1},m_{k+1}]^{2i}v_{k-2}^jP_{k}=\coprod_{\substack{i,j\in \{0,1\}}}(x_{k}^{2i+2j}N_{k}^{od})\times (x_{k}^{2i}N_{k}^{od})^{x_{k+1}}.
\end{equation}
Using the induction assumption, for $i,j$ with $i+j$ even, we have
\begin{equation}
\text{CD}(x_{k}^{2i+2j}N_{k}^{od},M_{k}) = A_6^{(k)}\sqcup A_8^{(k)},
\end{equation}
and for $i,j$ with $i+j$ odd, we have
\begin{equation}
\text{CD}(x_{k}^{2i+2j}N_{k}^{od},M_{k}) = A_5^{(k)}\sqcup A_7^{(k)},
\end{equation}
and for $i$ even, we have
\begin{equation}
\text{CD}(x_{k}^{2i}N_{k}^{od},M_{k}) = A_6^{(k)}\sqcup A_8^{(k)},
\end{equation}
and for $i$ odd, we have
\begin{equation}
\text{CD}(x_{k}^{2i}N_{k}^{od},M_{k}) =A_5^{(k)}\sqcup A_7^{(k)} .
\end{equation}
Using these, and arguing similarly to the proof of Claim $4.6$, we get
\begin{equation}
\text{CD}(\coprod_{\substack{i=0, j=0}}(x_{k}^{2i+2j}N_{k}^{od})\times (x_{k}^{2i}N_{k}^{od})^{x_{k+1}},M_{k+1}) = (A_6^{(k)}\sqcup A_8^{(k)})\times (A_6^{(k)}\sqcup A_8^{(k)}),
\end{equation}
\begin{equation}
\text{CD}(\coprod_{\substack{i=0, j=1}}(x_{k}^{2i+2j}N_{k}^{od})\times (x_{k}^{2i}N_{k}^{od})^{x_{k+1}},M_{k+1}) = (A_5^{(k)}\sqcup A_7^{(k)})\times (A_5^{(k)}\sqcup A_7^{(k)}),
\end{equation}
\begin{equation}
\text{CD}(\coprod_{\substack{i=1, j=0}}(x_{k}^{2i+2j}N_{k}^{od})\times (x_{k}^{2i}N_{k}^{od})^{x_{k+1}},M_{k+1}) = (A_5^{(k)}\sqcup A_7^{(k)})\times (A_5^{(k)}\sqcup A_7^{(k)}),
\end{equation}
and
\begin{equation}
\text{CD}(\coprod_{\substack{i=1, j=1}}(x_{k}^{2i+2j}N_{k}^{od})\times (x_{k}^{2i}N_{k}^{od})^{x_{k+1}},M_{k+1}) = (A_6^{(k)}\sqcup A_8^{(k)})\times (A_6^{(k)}\sqcup A_8^{(k)}).
\end{equation}
Combining $(6.77)-(6.80)$, and using $(6.72)$ and Lemma $6.27$, we have completed the proof of (h).\\

Hence, we establish all the parts, which completes the proof of Proposition $6.16$.
\end{proof}
\subsection{Model 6}
\begin{corollary}
Set $M_n(2) = \langle x_n^2, x_nm_n, N_n^{od}\rangle.$ Then $M_n(2)$ satisfies Model $6$ for all $n\geq 1$.
\end{corollary}
\begin{proof}
We have 
$$M_n(2) = N_n^{od}\sqcup x_n^2N_n^{od}\sqcup x_nm_nN_n^{od}\sqcup x_n^3m_nN_n^{od}.$$
Now the result directly follows from the proofs of parts $(b)$, $(e)$, $(f)$, $(g)$ and $(h)$ of Proposition $6.16.$
\end{proof}
\subsection{Hausdorff Dimensions}
As in the previous sections, we denote by $M$ and $M(2)$ the inverse limits of $\{M_n\}_{n\geq 1}$ and $\{M_n(2)\}_{n\geq 1}$, respectively. Their Hausdorff dimensions are clearly same by Corollary $6.28$. We have the following corollary:
\begin{corollary}
$\mathcal{H}(M) = \mathcal{H}(M(2)) = \frac{3}{4}$
\end{corollary}
\begin{proof}
We will first calculate $|M_n|$ for all $n$. Note that we have $|M_1| = 2$.
\begin{claim}
For any $n\geq 2$, we have $|M_n| = 2^{3\cdot 2^{n-2}}$.
\end{claim}
\begin{proof}[Proof of Claim $6.30$]
The claim is trivially true for $n=2$. Suppose it is true for $n=k\geq 2$. By the proof of Claim $6.6$ and Proposition $6.15$, we have $|M_{k+1}| = 8|N_{k+1}^{od}|$ and $|N_{k+1}^{od}| = 8|N_k^{od}|^2$. Hence, we get
$$|M_{k+1}| = 8|N_{k+1}^{od}| = 64|N_k^{od}|^2 = |M_k|^2 = 2^{3\cdot 2^{k-1}},$$
which finishes the proof of the claim.
\end{proof}
Corollary $6.29$ directly follows from Claim $6.30$.
\end{proof}
\section{From markov groups to galois groups}
Let $f(x)\in \mathbb{Z}[x]$ be as in Theorem $1.1$ or Theorem $1.2$, and denote by $G_n(f)$ the Galois group of $f^n$ over $\mathbb{Q}(i)$. Recall that $M_n(f)$ is the level $n$ even Markov group of $f$, as constructed in the previous three sections. In this section, we will explore the connection between $G_n(f)$ and $M_n(f)$.\\

Precisely, we have the following conjecture:\\

\begin{conjecture}
$G_n(f)$ is isomorphic to a subgroup of $M_n(f)$ for all $n\geq 1$.
\end{conjecture}

The motivation for this conjecture is the following natural idea: Recall that $M_n(f)$ is defined using the cycle data that estimate the factorization types of $f^n$ modulo primes of the form $4k+1$. However, although the set of factorization types that arise from Markov process definitely contains the set of actual factorization types of $f^n$ modulo primes of the form $4k+1$, these two sets are not necessarily equal, since there could be missing transitions that the Markov process is not taking into account. Hence, noting that the actual factorization types modulo primes of the form $4k+1$ correspond to the cycle structures of $G_n(f)$ (by Chebotarev's density theorem), one may reasonably expect that $M_n(f)$ would contain a copy of $G_n(f)$ for all $n\geq 1$.\\

We will introduce a purely group theoretical question, and provide an argument indicating that an affirmative answer to a special case of this question would imply Conjecture $7.1$. We will state the question in its general form, because we think that it is interesting in its own right. We will start by giving the following two definitions:\\

\begin{definition}
Let $H,G\leq W_n$ be two subgroups of $W_n$. Set $K_n=\text{Ker}(\pi_n:W_n\twoheadrightarrow W_{n-1})$. We say $\boldsymbol{H}$ \textbf{is  elementwise} $\boldsymbol{K_n}$\textbf{-conjugate into} $\boldsymbol{G}$, if for all $h\in H$, there exists $k_h\in K_n$ such that $h^{k_h}\in G$. We say $\boldsymbol{H}$ \textbf{is globally} $\boldsymbol{K_n}$\textbf{-conjugate into} $\boldsymbol{G}$, if there exists $k\in K_n$ such that $H^k\leq G$. 
\end{definition}
\begin{definition}
Let $n\geq 1$ and $H,G\leq W_n$. We say that the pair $(H,G)$ satisfies the property $\mathcal{P}$ if
\begin{equation}
H \text{ is elementwise } K_n\text{-conjugate into } G \iff H \text{ is globally } K_n\text{-conjugate into } G.
\end{equation}
\end{definition}
\begin{question}
Let $n\geq 1$. For which subgroups $H,G\leq W_n$, does $(H,G)$ satisfy $\mathcal{P}$?
\end{question}

 Although we will not use it in the rest of the paper, we will now state a very partial result (without proof) we get in this direction for the curious reader:\\

\begin{theorem}
Let $n\geq 1$, and $H,G\leq W_n$. Then $(H,G)$ satisfies $\mathcal{P}$ in the following two cases:\\
\item[(1)] $|H| = 2^{n+1}$ and $H$ contains an $n$-odometer.
\item[(2)] $|H|=|G|$ and $H\cap K_n = \{1\}$.
\end{theorem}

Note that we did not include some trivial cases (such as when $H$ is a cyclic group) in Theorem $7.5$. The proof of Theorem $7.5$ requires using the iterated wreath product structure of $W_n$.\\

We go back to the connection between Question $7.4$ and Conjecture $7.1$. We will connect Conjecture 7.1 with the following conjecture:\\

\begin{conjecture}
Let $n\geq 1$ and $f(x)\in \mathbb{Z}[x]$ as in Theorem $1.1$ or Theorem $1.2$, and $M_n(f)$ the level $n$ even Markov group of $f$, as constructed in the previous three sections. Then for any $H\leq W_n$, $(H,M_n(f))$ satisfies $\mathcal{P}$.
\end{conjecture}

We emphasize that we explicitly know $M_n(f)$ for all $n$, by the previous three sections. Although extensive MAGMA computations appear to confirm Conjecture $7.6$ for small values of $n$, whether it is true for all $n$ remains open. We now outline an argument indicating how Conjecture $7.6$ should imply Conjecture 7.1:\\

We will use induction. Conjecture $7.1$ trivially holds for $n=1$. Suppose it holds for $n=k$. We can assume without loss of generality that $G_k(f)\leq M_k(f)$. Let $x\in G_k(f)\leq M_k(f)$. Consider the pre-image sets $A_x:=(\pi_{k+1}|_{G_{k+1}(f)})^{-1}(x)$ and $B_x:=(\pi_{k+1}|_{M_{k+1}(f)})^{-1}(x)$. Let $x=c_1\dots c_i$, where $c_1,\dots\,c_i$ are disjoint cycles. Take any $y\in A_x$ with disjoint cycle decomposition $y=(d_1d_1')\dots(d_id_i')$. Note that this cycle decomposition corresponds to the factorization type of the reduced polynomial $\overline{f^{k+1}}\in \mathbb{F}_p[x]$ for some prime $p \equiv 1$ (mod $4$). Recall also that the group $M_n(f)$ satisfies the even Markov model of $f$. Hence, by the Markov model, there will be $y'\in B_x$ with disjoint cycle decomposition $y' = (e_1e_1')\dots (e_ie_i')$ such that for all $j=1,\dots i$, we have\\

\begin{itemize}
	\item [(i)] $d_jd_j'$ is a doubling of $c_j$ $\iff$ $e_je_j'$ is a doubling of $c_j$.\\
	\item[(ii)] $d_jd_j'$ is a splitting of $c_j$ $\iff$ $e_je_j'$ is a splitting of $c_j$.\\
\end{itemize}
But, this is equivalent to the existence of $k_y\in K_{k+1}$ such that $y' = y^{k_y}$. Doing this over all the pre-images of all the elements of $G_k(f)$ under $\pi_{k+1}|_{G_{k+1}(f)}$, we obtain that $G_{k+1}(f)$ is elementwise $K_{k+1}$-conjugate into $M_{k+1}(f)$. Given this, Conjecture $7.1$ would imply Conjecture $7.6$.
\begin{customrmk}{}
	For the most PCF quadratic polynomials $f(x)\in \mathbb{Z}[x]$, $G_n(f)$ contains an $n$-odometer for all $n\geq 1$. Hence, for attacking Conjecture $7.6$, one can try assuming that $H$ contains an $n$-odometer. This assumption has the potential of being particularly useful, because for the small values of $n$, the author was unable to find any pair of subgroups $(H,G)$ in $W_n$ such that $H$ has an $n$-odometer and $(H,G)$ does not satisfy $\mathcal{P}$, despite exhaustive MAGMA computations.
\end{customrmk}
\section{Final remarks/speculations}
We first make some remarks about Question $2.12$.\\

Empirically, Question $2.12$ also appears to have an affirmative answer, but whether it is true for all $n$ remains open. Let $f(x)\in \mathbb{Z}[x]$ be as in Theorem $1.1$ or Theorem $1.2$. Recall that we denote the Galois group of $f^n$ over $\mathbb{Q}(i)$ by $G_n(f)$. Hence, it follows that the Galois group of $f^n$ over $\mathbb{Q}$ is generated by $G_n(f)$ and $\sigma_n$, where $\sigma_n$ is the image of the complex conjugation in $W_n$. Analogous to this, for small values of $n$, it empirically appears that the group generated by $M_n(f)$ and $\sigma_n$ provides an affirmative answer to Question $2.12$.\\

Secondly, we would like to note that our construction method in all the sections appears to work for some other PCF quadratic polynomials as well, such as the conjugates of $x^2+i$, however, some additional complications arise when one tries to prove it. For instance, one can construct the groups $N_n^{od}$ similarly to the cases $x^2-2$ and $x^2-1$, but the inclusion $N_{n-1}^{od}\times (N_{n-1}^{od})^{x_n}\trianglelefteq N_n^{od}$ does not hold anymore. Hence, one should perhaps use a different subgroup of $N_n^{od}$ instead. Also, for certain conjugates of $x^2+i$, one needs to add one more generator which should be constructed similarly to the elements $m_n$ that are used in Sections $5$ and $6$. This case together with the conjugates of $x^2-2$ that are not treated in this paper are the subject of ongoing work.
\subsection*{Acknowledgments}
The author would like to thank his advisor Nigel Boston for suggesting this project, many helpful conversations and all his guidance. The author also thanks Rafe Jones for the helpful conversation during one of his visits to UW-Madison.



\normalsize


\begin{thebibliography}{HD82}
	
	\bibitem[1]{Benedetto} Benedetto, R.L., Faber, X., Hutz, B., Juul, J., Yasufuku, Y.: A large arboreal Galois representation for a cubic postcritically finite polynomial.  Res. number theory (2017) 3: 29.
	
	\bibitem[2]{Boston}N.Boston. “Large transitive groups with many elements having fixed points.” Contemporary Mathematics 524, 11–15, AMS volume in honor of Marty Isaacs (2010).
	
	\bibitem[3]{JonesBoston1} Nigel Boston and Rafe Jones, Arboreal Galois representations, Geom. Dedicata 124 (2007), 27–35.
	
	\bibitem[4]{JonesBoston2}  Nigel Boston and Rafe Jones, The image of an arboreal Galois representation, Pure Appl. Math. Q. 5 (2009), no. 1, 213–225.
	
	\bibitem[5]{Settled} Nigel Boston and Rafe Jones. Settled polynomials over finite fields. Proc. Amer. Math. Soc., 140(6):1849–1863, 2012.
	
	\bibitem[6]{goksel} Vefa Goksel, Shixiang Xia, and Nigel Boston, A refined conjecture for factorizations of iterates of quadratic polynomials over finite fields, Exp. Math. 24 (2015), no. 3, 304–311.
	
	\bibitem[7]{Gottesman} Gottesman, R., Tang, K.: Quadratic recurrences with a positive density of prime divisors. Int. J. Number Theory 6(5),
	1027–1045 (2010).
	
	\bibitem[8]{Harpe}  Pierre de la Harpe. Topics in geometric group theory. Chicago Lectures in Mathematics. University of Chicago Press, Chicago, IL, 2000.
	
	\bibitem[9]{survey} Jones, R.: Galois representations from pre-image trees: an arboreal survey. Publ. Math. Besançon, pp 107–136 (2013).
	
	\bibitem[10]{Self-Similar} Nekrashevych, V.: Self-similar groups. Mathematical Surveys and Monographs, vol. 117. American Mathematical Society, Providence (2005).
	
	\bibitem[11]{Pink} Pink, R.: Profinite iterated monodromy groups arising from quadratic polynomials. arXiv:1307.5678 [math.GR], preprint, 2013.
	
	\bibitem[12]{Seneta} E. Seneta, Non-negative matrices and Markov chains, Springer Series in Statistics, Springer, New York, 2006, Revised reprint of the second (1981) edition [Springer-Verlag, New York; MR0719544].
	
	 \bibitem[13]{Silverman} Silverman, J.H.: The arithmetic of dynamical systems. Graduate Texts in Mathematics, vol. 241. Springer, New York (2007).
	
	
	
	
	\baselineskip=17pt
	
	
	
	
	
\end{thebibliography}
\end{document}